\documentclass{article}
\usepackage{geometry}               
\usepackage{booktabs}
\usepackage{float}

\usepackage{tikz}
\usetikzlibrary{positioning, arrows.meta, fit, calc}

\usepackage{graphicx}
\usepackage{caption}
\usepackage{subcaption}
\usepackage{comment}
\usepackage{amsmath} 
\usepackage{amssymb}  
\usepackage{amsthm}  
\usepackage{bm}
\usepackage{lipsum}
\usepackage{color}
\usepackage{multirow}
\usepackage{array}
\usepackage{cancel}
\usepackage{cite}
\usepackage{hyperref}
\usepackage{ulem}

\usepackage{algorithm}
\usepackage{algpseudocode}

\newtheorem{lem}{Lemma}
\newtheorem{thm}{Theorem}
\newtheorem{assumption}{Assumption}

\newtheorem{remark}{Remark}
\numberwithin{equation}{section}

\newcommand{\diver}{\nabla\!\cdot\!}

\newcommand{\E}{\mathbb{E}}

\newcommand{\R}{\mathbb{R}}

\newcommand{\T}{\mathbb{T}}

\newcommand{\dd}{ \mathrm d}

\title{A Physics-Informed, Global-in-Time Neural Particle Method for the Spatially Homogeneous Landau Equation}

\author{
	Minseok Kim$^{1,\dagger}$,
	Sung-Jun Son$^{2,\dagger}$,
	Yeoneung Kim$^{1,\ast}$,
	Donghyun Lee$^{2,\ast}$
}

\date{}

\begin{document}
	\maketitle
	
	\noindent
	$^{1}$Department of Applied Artificial Intelligence,
	Seoul National University of Science and Technology, Seoul, Republic of Korea\\
	$^{2}$Department of Mathematical Sciences,
	Pohang University of Science and Technology (POSTECH), Pohang, Republic of Korea
	
	\vspace{0.5em}
	\noindent
	$^{\dagger}$These authors contributed equally to this work.\\
	$^{\ast}$Corresponding authors. YK: yeoneung@seoultech.ac.kr, DL: donglee@postech.ac.kr
	
\begin{abstract}
We propose a physics-informed neural particle method (PINN–PM) for the spatially homogeneous Landau equation. The method adopts a Lagrangian interacting-particle formulation and jointly parameterizes the time-dependent score and the characteristic flow map with neural networks. 
Instead of advancing particles through explicit time stepping, the Landau dynamics is enforced via a continuous-time residual defined along particle trajectories. This design removes time-discretization error and yields a mesh-free solver that can be queried at arbitrary times without sequential integration. We establish a rigorous stability analysis in an $L^2_v$ framework. 
The deviation between learned and exact characteristics is controlled by three interpretable sources: (i) score approximation error, (ii) empirical particle approximation error, and (iii) the physics residual of the neural flow. This trajectory estimate propagates to density reconstruction, where we derive an $L^2_v$ error bound for kernel density estimators combining classical bias–variance terms with a trajectory-induced contribution. Using Hyv\"arinen’s identity, we further relate the oracle score-matching gap to the $L^2_v$ score error and show that the empirical loss concentrates at the Monte Carlo rate, yielding computable a posteriori accuracy certificates. Numerical experiments on analytical benchmarks, including the two- and three-dimensional BKW solutions, as well as reference-free configurations, demonstrate stable transport, preservation of macroscopic invariants, and competitive or improved accuracy compared with time-stepping score-based particle and blob methods while using significantly fewer particles.
\end{abstract}

\section{Introduction}\label{sec:intro}
The spatially homogeneous Landau equation models the evolution of the velocity distribution of charged particles undergoing grazing collisions. It describes the time evolution of the velocity distribution $\tilde f_t(v)$:\footnote{We focus on the Coulomb case $\gamma=-3$; the Maxwell case $\gamma=0$ follows verbatim.}
\begin{equation}\label{eq:landau}
\partial_t \tilde f_t(v) = \nabla_v \cdot \!
\int_{ \Omega} A(v-v^*) \bigl( \tilde f_t(v^*)\nabla_v \tilde f_t(v) - \tilde f_t(v)\nabla_{v^*} \tilde f_t(v^*) \bigr) \dd v^* , \quad v\in \Omega.
\end{equation}
Here $A(z)=C_{\gamma}|z|^{\gamma}\bigl(|z|^2 I - z\otimes z\bigr)$ for $\gamma\in[-3,1]$. Equation~\eqref{eq:landau} possesses a gradient--flow structure and conserves mass, momentum, and energy while dissipating entropy \cite{V1998}. Numerically, the challenge lies in designing schemes that (i) capture this long--time structure, (ii) scale efficiently to high dimensions, and (iii) provide verifiable accuracy certificates.

\paragraph{Mathematical and Numerical Background.}
The mathematical study of the Landau equation dates back to the seminal work of Landau \cite{L1936} and its derivation from the Boltzmann equation \cite{DLD1992}. Analytical breakthroughs regarding weak solutions, regularity, and convergence to equilibrium have been established in works by Villani and Desvillettes \cite{V1998, DV2000a, DV2000b, V1998b}, with recent results addressing global regularity \cite{GS2025}. Translating these insights into robust numerical simulations has been a persistent challenge. Early structure-preserving efforts focused on conservation laws and entropy decay (e.g., entropy-consistent finite-volume and spectral solvers \cite{CS1998, DLD1994, PRT2000}). However, grid-based methods often struggle with the \textit{curse of dimensionality}. As an alternative, particle-based approaches like Direct Simulation Monte Carlo (DSMC) \cite{DCP2010} and Particle-In-Cell (PIC) variants \cite{GLD2011, BCH2024} were developed. While naturally mesh-less, they often suffer from statistical noise or require complex re-projection steps to maintain conservation. Recent works have also explored randomized interaction reductions such as the random batch method \cite{JDW2022} and structure-preserving discretizations \cite{H2021} (see also the survey in \cite{JKF2019}).

\paragraph{The Score-Based Particle (SBP) Paradigm.}
A recent paradigm shift stems from re-interpreting the Landau equation as a gradient flow in the Wasserstein metric \cite{JDW2022, JDDW2024}. This perspective links kinetic theory to score matching, a technique pioneered by Hyv\"arinen \cite{A2005} and popularized in generative modeling \cite{SE2019, SE2020}. In this context, the velocity field's nonlinearity is driven by the {score} $s=\nabla\log f$. Building on this, Huang and Wang proposed the {score-based particle} (SBP) methodology \cite{YL2025}. SBP learns the score dynamically to bypass expensive kernel density estimation (KDE) while advancing particles via a structure-preserving ODE. This method inherits conservation from deterministic particles, scales favorably with dimension (empirically $\mathcal O(N)$), and provides theoretical links between the Kullback–Leibler divergence and the score-matching objective \cite{YL2025}.

Despite these advances, SBP (and most particle formulations) still evolve trajectories via an explicit time loop. Consequently, accuracy and computational cost are tightly coupled to the time step $\Delta t$. Even with an exact score, one encounters Monte Carlo sampling error $\mathcal O(N^{-1/2})$ and time discretization error $\mathcal O(\Delta t)$, making accuracy certification dependent on external time-step control. Moreover, existing SBP analyses typically rely on coercivity assumptions on the collision kernel to obtain entropy-based stability estimates. In contrast, our trajectory-based analysis only requires boundedness and Lipschitz regularity of the interaction kernel, as it relies on stability of the characteristic flow rather than entropy dissipation.

\paragraph{Physics-Informed Neural Networks (PINNs).}
Physics-Informed Neural Networks (PINNs) have emerged as a mesh-free paradigm for solving partial differential equations by embedding the governing equation directly into the training objective \cite{RPK2019}.  Instead of discretizing time and space explicitly, PINNs approximate the solution
as a neural function and enforce the PDE through a continuous-time residual loss
evaluated at collocation points. This approach has demonstrated effectiveness in high-dimensional settings, where traditional grid-based solvers suffer from the curse of dimensionality.

However, most existing PINN formulations operate in an Eulerian framework, where the density field itself is parameterized and the PDE residual is enforced pointwise. For nonlinear kinetic equations such as the Landau equation, this strategy faces two difficulties: (i) the collision operator is nonlocal and quadratic in the density, and (ii) characteristic transport plays a central structural role.

We propose a fundamentally different perspective. Instead of alternating between score learning and time stepping, we learn the entire interacting particle dynamics as a global-in-time neural flow. Concretely, we parameterize both the score $s(v,t)$ and the characteristic map $\Phi_\xi(v_0,t)$ as neural networks defined on $(v,t)$, and enforce the Landau dynamics through a continuous-time physics residual. Once trained, the model acts as a neural particle simulator: given initial samples $V_i$, particle configurations at arbitrary times are generated directly by neural inference, without sequential integration.

\paragraph{Contributions.}
We propose PINN-based particle method (PINN--PM) for solving the Landau equation and our main contributions are as follows:

\begin{itemize}
\item \textbf{Global-in-time neural interacting particle formulation.}
We remove explicit time discretization by jointly learning the score and characteristic flow over the time horizon.

\item \textbf{Neural particle simulator.}
The learned flow map $\Phi_\xi$ enables mesh-free temporal inference: particle configurations at arbitrary times are obtained through a single forward pass, without time stepping.

\item \textbf{Residual-based error certificates.}
We establish trajectory stability and density reconstruction bounds linking score error, measure mismatch, and physics residual to deployment-time accuracy.
\end{itemize}

Compared to SBP~\cite{YL2025}, our approach eliminates time-step dependence, provides end-to-end error guarantees directly tied to training losses, and empirically achieves competitive or improved accuracy with significantly fewer particles.

\paragraph{Organization.} The remainder of this paper is organized as follows.
Section \ref{sec:prelim} presents the problem formulation and the proposed PINN–PM methodology.
Section \ref{sec:Neu cha flow} develops the stability and trajectory error analysis of the neural characteristic flow.
Section \ref{sec:certi} establishes an oracle-to-dynamics certification framework connecting score learning errors to Wasserstein stability and density reconstruction.
Section \ref{sec:exp} reports numerical experiments, and Section \ref{subsec:discussion} concludes this paper.

\noindent{\bf Notation :} To avoid confusion, we fix the following convention:
\begin{itemize}
\item We use $|a|$ to denote the absolute value of a scalar $a \in \R$, and $\|x\|$ to denote the Euclidean norm of a vector $x \in \mathbb{R}^d$.
On the torus $\mathbb{T}^d = \mathbb{R}^d/\mathbb{Z}^d$, we also use $\|x-y\|$ 
to denote the torus distance
\[
\|x-y\| := \min_{k\in\mathbb{Z}^d} \|x-y+k\|.
\]
\item For a function $f = f(v)$ defined on the $d$-dimensional torus $\mathbb{T}^d$, we write
\[
\|f\|_{L^2_v} = \Big( \int_{\mathbb{T}^d} |f(v)|^2 \, dv \Big)^{1/2},
\]
and, when $f$ depends also on time $t$, we use the notation 
$\|f_t\|_{L^2_v}$ to emphasize that the norm is taken only with respect to 
the velocity variable $v$.

\item The expectation symbol $\E[\cdot]$ without a subscript always refers to the expectation  with respect to the randomness of particle {sampling} (Monte Carlo average) at time $t=0$. 
\end{itemize}

\section{Problem formulation and deterministic particle approximation}
\label{sec:prelim}

\subsection{Spatially homogeneous Landau equation}

We consider the spatially homogeneous Landau equation in velocity space:
\begin{equation}
\partial_t \tilde f_t(v)
=
\nabla_v \cdot
\int_{\Omega}
A(v-v^*)
\bigl(
\tilde f_t(v^*) \nabla_v \tilde f_t(v)
-
\tilde f_t(v)\nabla_{v^*}\tilde f_t(v^*)
\bigr)
\, dv^*,
\quad v\in\Omega,
\label{eq:landau_pde}
\end{equation}
where $\Omega=\mathbb{R}^d$ or $\mathbb{T}^d$ (see Remark~\ref{rmk:domain}), and
\[
A(z)=C_\gamma |z|^\gamma (|z|^2 I - z\otimes z),
\qquad \gamma\in[-3,1].
\]
We focus on the Coulomb case $\gamma=-3$ and the Maxwell case $\gamma=0$.

\begin{remark}[Domain and Boundary Conditions]\label{rmk:domain}
We formulate the problem on the torus $\Omega = \mathbb{T}^d$ to simplify the exposition of boundary terms. This choice avoids technical digressions on decay rates at infinity while retaining the core nonlinear structure of the collision operator. However, our analytical framework and the PINN formulation extend directly to the whole space $\mathbb{R}^d$ under standard regularity and decay assumptions (i.e., $f_t(v)$ and an approximate score $s_\theta(v,t)$ vanishing sufficiently fast as $|v|\to\infty$ to validate integration--by--parts identities). Since the Landau operator is local in space (homogeneous) and nonlocal only in velocity, the distinction between $\mathbb{T}^d$ and $\mathbb{R}^d$ does not alter the essential methodology.
\end{remark}

Equation~\eqref{eq:landau_pde} admits the transport form
\begin{equation}
\partial_t \tilde f_t
=
- \nabla_v \cdot \bigl( \tilde f_t U[\tilde f_t] \bigr),
\label{eq:landau_transport}
\end{equation}
where the nonlinear velocity field is given by
\begin{equation}
U[\tilde f_t](v)
=
-\int_{\Omega}
A(v-v^*)
\bigl(
\nabla_v \log \tilde f_t(v)
-
\nabla_v \log \tilde f_t(v^*)
\bigr)
\tilde f_t(v^*)
\, dv^*.
\label{eq:mean_field_drift}
\end{equation}

Introducing the score function
\[
\tilde s_t(v):=\nabla_v \log \tilde f_t(v),
\]
the drift can be written compactly as
\[
U[\tilde f_t](v)
=
-\int_{\Omega}
A(v-v^*)
\bigl(
\tilde s_t(v)-\tilde s_t(v^*)
\bigr)
\tilde f_t(v^*)
\, dv^*.
\]

Thus, the spatially homogeneous Landau equation may be interpreted as a nonlinear transport equation in velocity space driven by a mean-field interaction expressed through score differences.

Under suitable regularity assumptions on $\tilde f_t$, the transport form~\eqref{eq:landau_transport} admits a Lagrangian representation.

Let $T_t:\Omega\to\Omega$ denote the characteristic flow defined by
\begin{equation}
\frac{d}{dt} T_t(v_0)
=
U[\tilde f_t]\bigl(T_t(v_0)\bigr),
\qquad
T_0(v_0)=v_0.
\label{eq:true_flow}
\end{equation}

Then the solution of~\eqref{eq:landau_transport} can be written as the pushforward
\begin{equation}
\tilde f_t = T_t \# f_0,
\label{eq:pushforward}
\end{equation}
in the sense that for any smooth test function $\phi$,
\[
\int_\Omega \phi(v)\tilde f_t(v)\,dv
=
\int_\Omega \phi(T_t(v_0)) f_0(v_0)\,dv_0.
\]

This formulation highlights that the Landau equation can be interpreted as the mean-field limit of an interacting particle system in velocity space.

\subsection{Deterministic interacting particle approximation}

To approximate the pushforward representation~\eqref{eq:pushforward}, we consider $N$ particles
\[
\{v^{(i)}_t\}_{i=1}^N
\]
initialized as i.i.d.\ samples from $f_0$.

The empirical measure is defined by
\begin{equation}
f_t=\frac{1}{N}
\sum_{i=1}^N
\delta_{v^{(i)}_t}.
\label{eq:empirical_measure}
\end{equation}

Replacing the mean-field integral in~\eqref{eq:mean_field_drift} by its empirical approximation yields the deterministic interacting particle system
\begin{equation}
\frac{d}{dt} v^{(i)}_t
=
-
\frac{1}{N}
\sum_{j=1}^N
A(v^{(i)}_t -v^{(j)}_t )
\bigl(
\tilde s_t(v^{(i)}_t)
-
\tilde s_t(v^{(j)}_t)
\bigr),
\qquad i=1,\dots,N.
\label{eq:particle_exact_score}
\end{equation}

In practice, the exact score $\tilde s_t$ is unknown. Classical deterministic particle methods therefore approximate the score either by kernel density estimation or by neural score matching.

A representative time-discrete scheme, as used in score-based particle methods, reads
\begin{equation}
v_{t_{n+1}}^{(i)}
=
v_{t_n}^{(i)}
-
\Delta t
\frac{1}{N}
\sum_{j=1}^N
A(v_{t_n}^{(i)}-v_{t_n}^{(j)})
\bigl(
s_\theta(v_{t_n}^{(i)},t_n)
-
s_\theta(v_{t_n}^{(j)},t_n)
\bigr),
\label{eq:sbp_discrete}
\end{equation}
where $s_\theta$ denotes a learned approximation of the score.

This formulation preserves the interacting particle structure but introduces two fundamental sources of numerical error:
\begin{itemize}
\item time discretization error of order $\mathcal{O}(\Delta t)$,
\item Monte Carlo sampling error of order $\mathcal{O}(N^{-1/2})$.
\end{itemize}

Moreover, the particle dynamics must be advanced sequentially in time, so that computational cost and accuracy remain coupled to the time-step size.

\subsection{Global-in-time Neural Parameterization}

To overcome the limitations of discrete time-stepping, we propose the \textbf{PINN--PM} framework. The core idea is to parameterize both the particle trajectories and the score function as continuous functions of time, unifying density estimation and dynamics evolution into a single global-in-time optimization problem.

We introduce two neural networks shared across all particles and time instants:
\begin{itemize}
    \item \textbf{Dynamics Network (Flow Map):} $\Phi_\xi: \mathbb{R}^d \times [0,T] \to \mathbb{R}^d$. \\
    This network approximates the Lagrangian flow map $v_t = T_t(v_0)$. Given a set of initial particles $\{V_i\}_{i=1}^N \sim f_0$, the position of the $i$-th particle at any arbitrary time $t$ is given explicitly by:
    \begin{equation}\label{eq:flow-map}
        \hat v^{(i)}_t := \Phi_\xi(V_i, t).
    \end{equation}
    This parameterization ensures that the empirical measure is defined continuously in time as $f_t = \frac{1}{N}\sum_{i=1}^N \delta_{\Phi_\xi(V_i, t)}$.
    
    \item \textbf{Score Network:} $s_\theta: \mathbb{R}^d \times [0,T] \to \mathbb{R}^d$. \\
    This network approximates the time-dependent score function $s_\theta(v,t) \approx \nabla_v \log f_t(v)$. Unlike previous approaches that train separate models for each time step, $s_\theta$ learns the score evolution over the entire horizon $[0,T]$ simultaneously.
\end{itemize}

\begin{figure}[t]
\centering
\includegraphics[width=0.8\textwidth]{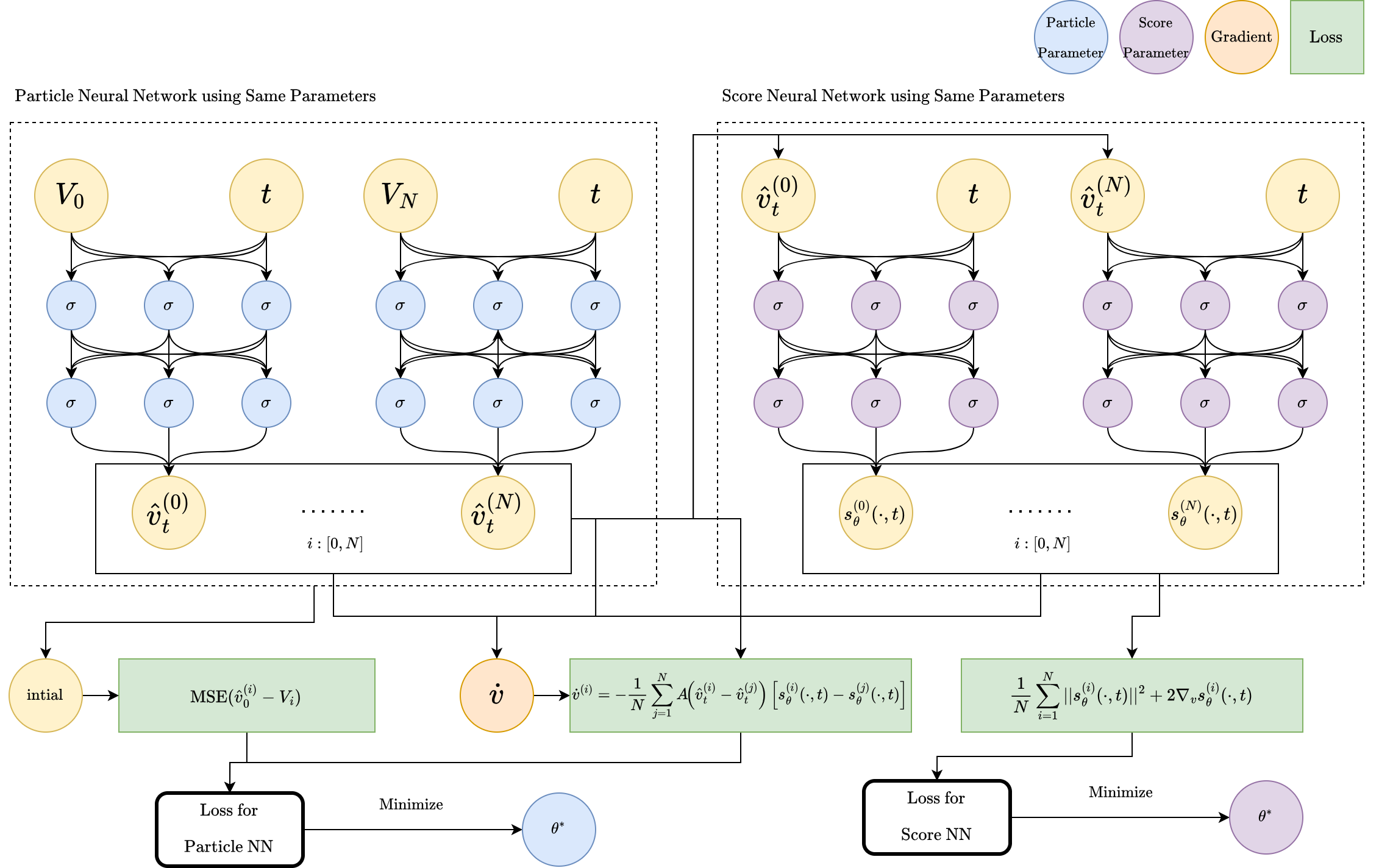}
\caption{\textbf{Conceptual architecture of PINN--PM.}
Left: global-in-time trajectory network $\Phi_\xi$ (shared parameters across particles).
Right: global-in-time score network $s_\theta$.
Both networks are trained jointly through a physics residual enforcing the Landau characteristics
and an implicit score-matching loss.
The framework removes explicit time stepping and enables direct querying at arbitrary times.}
\label{fig:concept}
\end{figure}

Figure~\ref{fig:concept} illustrates the overall architecture of the proposed PINN--PM. Unlike time-stepping particle methods, which repeatedly update particles and retrain the score, our framework jointly learns (i) a global-in-time trajectory network $\Phi_\xi(V_i,t)$ and (ii) a global-in-time score network $s_\theta(v,t)$. The two networks are coupled through a physics residual enforcing the Landau characteristic equation and an implicit score-matching objective. This design eliminates explicit temporal discretization and enables mesh-free inference in time.

To implement the global-in-time parameterization in practice,
we adopt a structured embedding design illustrated in
Figure~\ref{fig:global_arch}.
The initial velocity $V_i$ and time variable $t$
are first mapped through separate embedding networks,
producing velocity features $E^{i,v}$ and time features $E^{i,t}$.
These embeddings are subsequently fused and processed by
shared hidden layers to produce the continuous-time output
$v^{(i)}_t$.

This separation of velocity and time embeddings
stabilizes training and improves representation capacity,
while preserving the interpretation of $\Phi_\xi$
as a global flow map defined over $(v_0,t)$.

A key consequence of this design is that, once trained, the network $\Phi_\xi$ functions as a neural particle simulator. Given initial samples $V_i$, particle configurations at any time $t$
are obtained via a single forward evaluation, without numerical integration. This contrasts fundamentally with time-stepping particle methods, which require sequential updates and incur $\Delta t$-dependent errors.

\begin{figure}[H]
\centering
\includegraphics[width=0.5\linewidth]{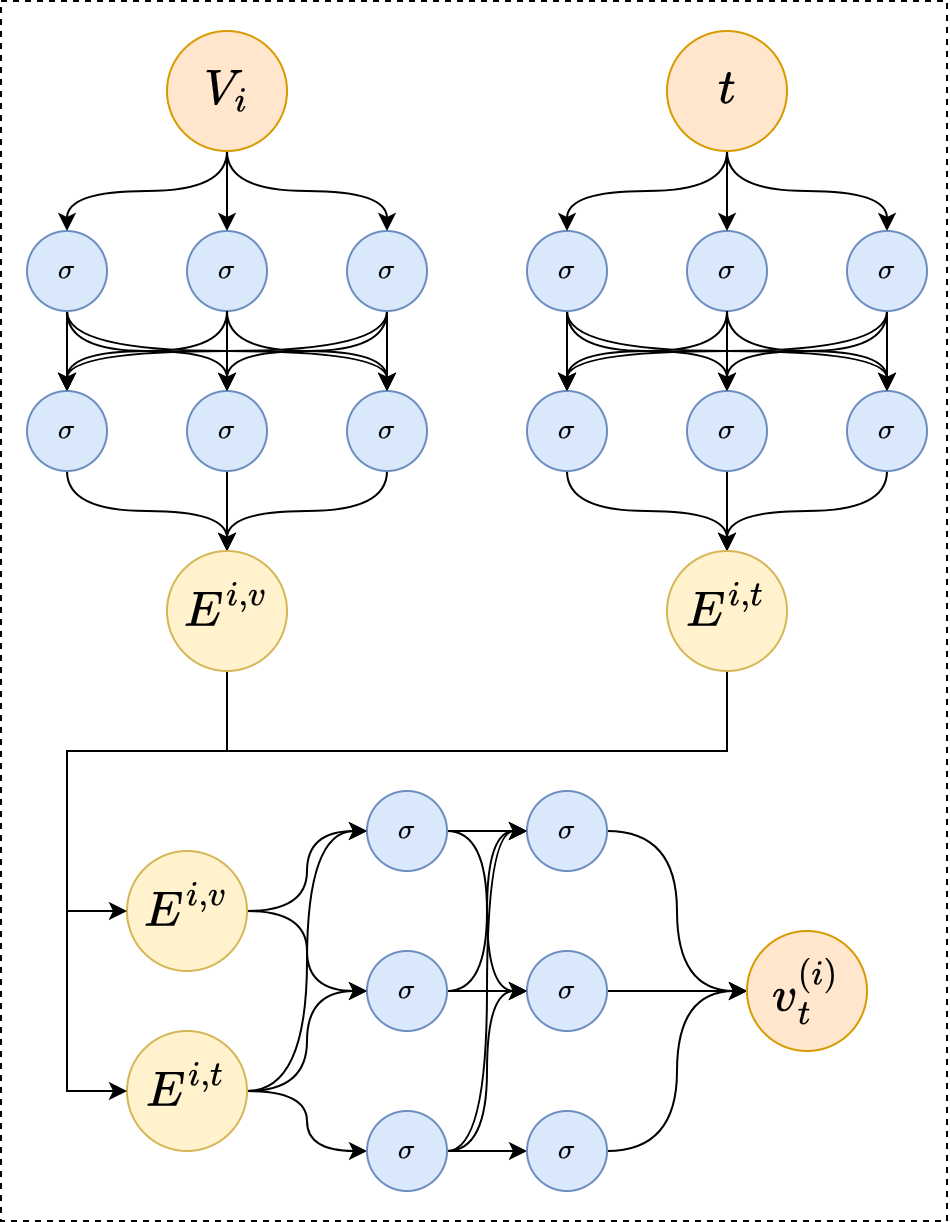}
\caption{
\textbf{Global-in-time architecture of PINN--PM.}
Separate velocity and time embeddings are processed by neural blocks
and fused to produce the continuous-time trajectory $v^{(i)}_t$.
The architecture enables mesh-free querying in time
without explicit time stepping.
}\label{fig:global_arch}
\end{figure}

\subsection{Physics-Informed Global Flow Learning: Formulation and Algorithm}
We now describe the unified training formulation of PINN--PM, which jointly learns the global characteristic flow and the time-dependent score without explicit time discretization.

We train the parameters $(\xi, \theta)$ jointly by minimizing a composite loss function that enforces (i) statistical consistency of the score (via Implicit Score Matching) and (ii) physical consistency of the trajectories (via the Landau characteristic residual).

\paragraph{(1) Implicit Score Matching (ISM) in Continuous Time.}
Since the true density is unknown, we employ the Hyv\"arinen objective. For a batch of collocation times $\{t_k\}_{k=1}^{N_t} \subset [0,T]$, we minimize:
\begin{equation}\label{eq:ism-loss}
\widehat{\mathcal{L}}_{\mathrm{ISM},t}(s_\theta) 
\;=\; 
\frac{1}{N_t N} \sum_{k=1}^{N_t} \sum_{i=1}^{N} 
\Bigl( \|s_\theta(v^{(i)}_{t_k}, t_k)\|^2 + 2 \nabla_v \cdot s_\theta(v^{(i)}_{t_k}, t_k) \Bigr),
\end{equation}
where $v_i(t_k) = \Phi_\xi(v_i^0, t_k)$. As shown in Section~\ref{sec:ism_control}, minimizing this objective ensures that $s_\theta$ converges to the true score of the push-forward density generated by $\Phi_\xi$.

\paragraph{(2) Physics Residual for Landau Dynamics.}
To ensure that the learned flow $\Phi_\xi$ obeys the Landau equation, we penalize the deviation from the characteristic ODE \eqref{eq:particle_exact_score}. The approximate mean-field drift at time $t$ is computed using the learned score:
\[
U_t^\delta(v) := -\frac{1}{N}\sum_{j=1}^N A\bigl(v -v^{(j)}_t\bigr)\bigl[s_\theta(v, t) - s_\theta(v^{(j)}_t, t)\bigr].
\]
We define the continuous-time physics residual for the $i$-th particle as:
\[
\mathcal{R}_i(t; \xi, \theta) := \frac{\partial \Phi_\xi}{\partial t}(V_i, t) - U_t^\delta\bigl(\Phi_\xi(V_i, t)\bigr).
\]
The physics loss is then the mean squared residual over the collocation points:
\begin{equation}\label{eq:phys-loss}
\mathcal{L}_{\mathrm{phys}}(\xi, \theta) \;=\; \frac{1}{N_t N} \sum_{k=1}^{N_t} \sum_{i=1}^{N} \|\mathcal{R}_i(t_k; \xi, \theta)\|^2.
\end{equation}
Note that $\partial_t \Phi_\xi$ is computed exactly via automatic differentiation, avoiding finite difference errors.

\paragraph{Total Loss.}
The final end-to-end objective is:
\begin{equation}\label{eq:full-loss}
\mathcal{L}(\xi, \theta) \;=\; \mathcal{L}_{\mathrm{phys}}(\xi, \theta) \;+\; \lambda_{\mathrm{score}} \widehat{\mathcal{L}}_{\mathrm{ISM},t}(s_\theta) 
\end{equation}
where $\lambda_{\mathrm{score}}$ balances the physics and statistical constraints.

The training procedure is summarized in Algorithm~\ref{alg:pinnpm_train}. A key advantage of this formulation is {amortized inference}: once trained, the particle configuration at any query time $t$ can be obtained via a single forward pass of $\Phi_\xi$, without strictly sequential integration.

\begin{algorithm}[H]
\caption{PINN--PM: joint global-in-time training}
\label{alg:pinnpm_train}
\begin{algorithmic}[1]
\Require Initial distribution $f_0$, time horizon $[0,T]$
\State Initialize neural networks $\Phi_\xi(\cdot,t)$ and $s_\theta(v,t)$
\While{not converged}
    \State Sample particles $\{V_i\}_{i=1}^N \sim f_0$
    \State Sample times $\{t_k\}_{k=1}^{N_t} \sim \mathcal{U}([0,T])$
    \For{$k=1,\dots,N_t$}
        \State Compute trajectories $\hat v^{(i)}_{t_k}=\Phi_\xi(V_i,t_k)$
        \State Construct empirical drift
        \[
        U_{t_k}^\delta(v)
        :=
        -\frac{1}{N}
        \sum_{j=1}^N
        A\bigl(v-\hat v^{(j)}_{t_k}\bigr)
        \bigl(
        s_\theta(v,t_k)
        -
        s_\theta(\hat v^{(j)}_{t_k},t_k)
        \bigr)
        \]
        \State Compute residuals
        \[
        \rho^{(i)}(t_k)
        :=
        \partial_t \Phi_\xi(V_i,t_k)
        -
        U_{t_k}^\delta\bigl(\hat v^{(i)}_{t_k}\bigr)
        \]
    \EndFor
    \State Physics loss
    \[
    \mathcal L_{\mathrm{phys}}
    =
    \frac{1}{N_tN}
    \sum_{k=1}^{N_t}
    \sum_{i=1}^N
    \|\rho^{(i)}(t_k)\|^2
    \]
    \State Score-matching loss
    \[
    \mathcal L_{\mathrm{ISM},t}
    =
    \frac{1}{N_tN}
    \sum_{k=1}^{N_t}
    \sum_{i=1}^N
    \Big(
    \|s_\theta(\hat v^{(i)}_{t_k},t_k)\|^2
    +
    2\,\nabla_v\!\cdot s_\theta(\hat v^{(i)}_{t_k},t_k)
    \Big)
    \]
    \State Joint gradient update
    \[
    (\xi,\theta)
    \leftarrow
    (\xi,\theta)
    -
    \eta
    \nabla_{(\xi,\theta)}
    \big(
    \mathcal L_{\mathrm{phys}}
    +
    \lambda_{\mathrm{score}}
    \mathcal L_{\mathrm{ISM},t}
    \big)
    \]
\EndWhile
\State \textbf{Inference:} output $\hat v_i(t)=\Phi_\xi(V_i,t)$ for any $t\in[0,T]$
\end{algorithmic}
\end{algorithm}

\section{Stability of the Neural Characteristic Flow}\label{sec:Neu cha flow}
The following results establish the theoretical justification of the neural particle simulator: we show that controlling the score loss and physics residual during training directly bounds the simulator's deployment error.

Now we analyse the error of the proposed \textbf{PINN--PM} method in the continuous setting. Our framework is based on $L^2_v$--expectations, which allows us to directly quantify the discrepancy between the empirical distribution $f_t$ generated by the algorithm and the exact solution $\tilde f_t$ of the Landau equation. 
This perspective naturally decomposes the total error into three components: 
(i) the {trajectory error}, arising from deviations between exact and approximate particle characteristics; 
(ii) the {score error}, due to approximation of the true score via implicit score matching; and 
(iii) the {particle error}, caused by the finite--particle approximation of the dynamics. 
The $L^2_v$ framework further enables quantitative bounds describing how these errors propagate and accumulate in time.

Let $\tilde f_t \in C^1([0,T];W^{2,\infty}(\T^d))$, $d\in\{2,3\}$, denote the solution of the spatially homogeneous Landau equation
\begin{equation}\label{eq:landau conti}
\partial_t \tilde f_t = -\diver \!\bigl(\tilde f_t U[\tilde f_t]\bigr), \qquad
U[\tilde f_t](v) := -\int_{\T^d}\!A(v-v^*)\bigl(\tilde s_t(v)-\tilde s_t(v^*)\bigr)\tilde f_t(v^*)\dd v^*,
\end{equation}
where $\tilde s_t := \nabla_v \log \tilde f_t$ is the exact score. 
We assume throughout that the interaction kernel is uniformly bounded,
\begin{equation}\label{eq:A-bounds}
0 \;\preceq\; A(z) \;\preceq\; \lambda_2 I, \qquad \forall z\in\T^d .
\end{equation}
where $I$ is identity matrix in $d\times d$.
Denote by $T_t:\T^d\to\T^d$ the exact flow map defined by
\begin{equation}\label{eq:true-flow}
\frac{\dd}{\dd t}T_t(v_0) = U[\tilde f_t]\bigl(T_t(v_0)\bigr), 
\qquad T_0(v_0)=v_0.
\end{equation}
Thus the exact solution can be written as the push--forward  $\tilde f_t = T_t \# f_0$ of the initial density $f_0$.

In the \textbf{PINN--PM} method, both the velocity field and the flow map are replaced by neural approximations. 
Given an approximate score $s_\theta(v,t)$ and the empirical distribution $f_t$, we evolve particle trajectories according to
\begin{equation}\label{eq:approx-flow}
\frac{\dd}{\dd t}\hat v_t 
= U^\delta_t(\hat v_t)
:= -\int_{\T^d}\!A(\hat v_t-v^*)\bigl(s_\theta(\hat v_t,t)-s_\theta(v^*,t)\bigr)f_t(v^*)\dd v^*,
\qquad \hat v_0=v_0.
\end{equation}

Note that in our implementation, this approximate trajectory is explicitly parameterized by the global dynamics network, i.e., $\hat v_t(v_0) \equiv \Phi_\xi(v_0, t)$. Thus, the analysis below directly quantifies the error of the learned neural flow.

The corresponding approximate flow map $\hat T_t$ then pushes forward the initial data $f_0$ to yield the empirical measure
\begin{equation}\label{eq:empirical-measure}
f_t = \frac{1}{N}\sum_{i=1}^N \delta_{\hat T_t(V_i)}, 
\qquad V_i \;\overset{\text{i.i.d.}}{\sim}\; f_0.
\end{equation}
Hence $f_t$ should be interpreted as the particle approximation induced by the neural flow map $\hat T_t$, 
in contrast to the exact density $\tilde f_t=T_t\# f_0$.

\subsection{Assumptions and Drift Regularity}\label{sec:traj-error}

In this subsection we derive an estimate for the deviation between the exact and approximate particle trajectories, that is, between $v_t = T_t(v_0)$ solving \eqref{eq:true-flow} and $\hat v_t = \hat T_t(v_0)$ solving  \eqref{eq:approx-flow}. Throughout the analysis we fix a time horizon $T>0$ and assume that both flows remain in a common ball, so that the Lipschitz bounds on the interaction kernel $A$ and the score functions are valid. 
The analysis is based on a perturbation argument for the respective dynamics and the use of Gronwall-type inequalities, which allow us to quantify how the local errors induced by the score approximation and the empirical measure propagate along the trajectories. Before proceeding, we shall introduce the following assumptions that are enforced throughout the paper:

\begin{assumption}[Regularity and boundedness]\label{ass:regularity}
Let $T>0$ and $t\in[0,T]$.
\begin{enumerate}
\item The interaction matrix $A:\T^d\to\R^{d\times d}$ is symmetric, and satisfies the uniform bound
\begin{equation}\label{eq:A-two-sided}
0\preceq A(z)\preceq \lambda_2 I ,\qquad \forall z\in\T^d,
\end{equation}
where $I$ is identity matrix in $d\times d$, and is Lipschitz continuous for $ -3< \gamma \leq 1$:
\begin{equation}\label{eq:A-Lip}
\|A(z)-A(w)\|\le L_A\|z-w\|,\qquad \forall z,w\in\T^d,
\end{equation}
Here and throughout, $\|\cdot\|$ denotes the Euclidean norm for vectors and the associated operator norm for matrices.
\item For each \(t\in[0,T]\), the approximate score \(s_\theta(\cdot,t)\) is Lipschitz:
\[
\|s_\theta(v_1,t)-s_\theta(v_2,t)\|
\le L_\theta(t)\|v_1-v_2\|,
\qquad \forall v_1,v_2\in\T^d.
\]
\item Define the score mismatch magnitude
\[
g_t(v):=\|s_\theta(v,t)-\tilde s_t(v)\|.
\]
Then $g_t$ is Lipschitz with constant $L_g(t)$:
\[
|g_t(v_1)-g_t(v_2)|
\le
L_g(t)\|v_1-v_2\|,
\qquad \forall v_1,v_2\in\T^d, \; t\in[0,T].
\]

\item The true score function $\tilde s(v,t)$ satisfies 
\[
L_s(t):=\|\nabla_v \tilde s_t\|_{L^\infty(\T^d)}<\infty.
\]
\item Denote by $W_1$ the 1–Wasserstein distance on $\T^d$.  The distributional error between the approximate and true densities is
\begin{equation}\label{eq:W1-def}
\Delta_f(t):=W_1\!\bigl( f_t,\tilde f_t\bigr) <\infty .
\end{equation}
\end{enumerate}
\end{assumption}

\begin{remark}\label{rem:Landau-regularization}
In the Coulomb case ($\gamma=-3$) the Landau kernel is singular and not globally Lipschitz.  Assumptions \eqref{eq:A-two-sided}--\eqref{eq:A-Lip} hold for the commonly used truncated/mollified kernels on bounded velocity domains, which is the regime targeted by particle computations; the above bounds therefore apply to such regularized models.
\end{remark}

We rely on the result shown in \cite{YL2025}, which demonstrates that if the neural network–trained score closely approximates the exact score, then the solution is also accurately approximated.
For completeness, we restate the theorem below, which shows that the KL divergence between the empirical measure $f_t$ and the exact Landau solution $\tilde f_t$ is controlled by the score-matching error, thereby formalizing the fact that the quality of score approximation directly determines the accuracy of the solution.

Figure~\ref{fig:error_flow_simple} summarizes the logical structure of our error analysis.  The central objective is to relate measurable training-time quantities to deployment-time accuracy of trajectories and densities.

We begin with three interpretable error sources: (i) the score approximation error $\delta_s(t)$, (ii) the measure mismatch $\Delta_f(t)$ between empirical and exact densities, and (iii) the PINN dynamics residual $\delta_{\mathrm{phys}}(t)$.
These quantities enter the drift mismatch at a fixed state,
which is subsequently propagated along characteristics
via a Gr\"onwall-type stability argument.

A synchronous coupling argument then closes the estimate at the expectation level, yielding a bound for $\mathbb E\|e_t\|^2$. The analysis is further connected to training objectives: the implicit score-matching (ISM) loss controls the $L^2$ score error, while kernel density reconstruction translates trajectory error into density error with an explicit bias–variance–trajectory decomposition.

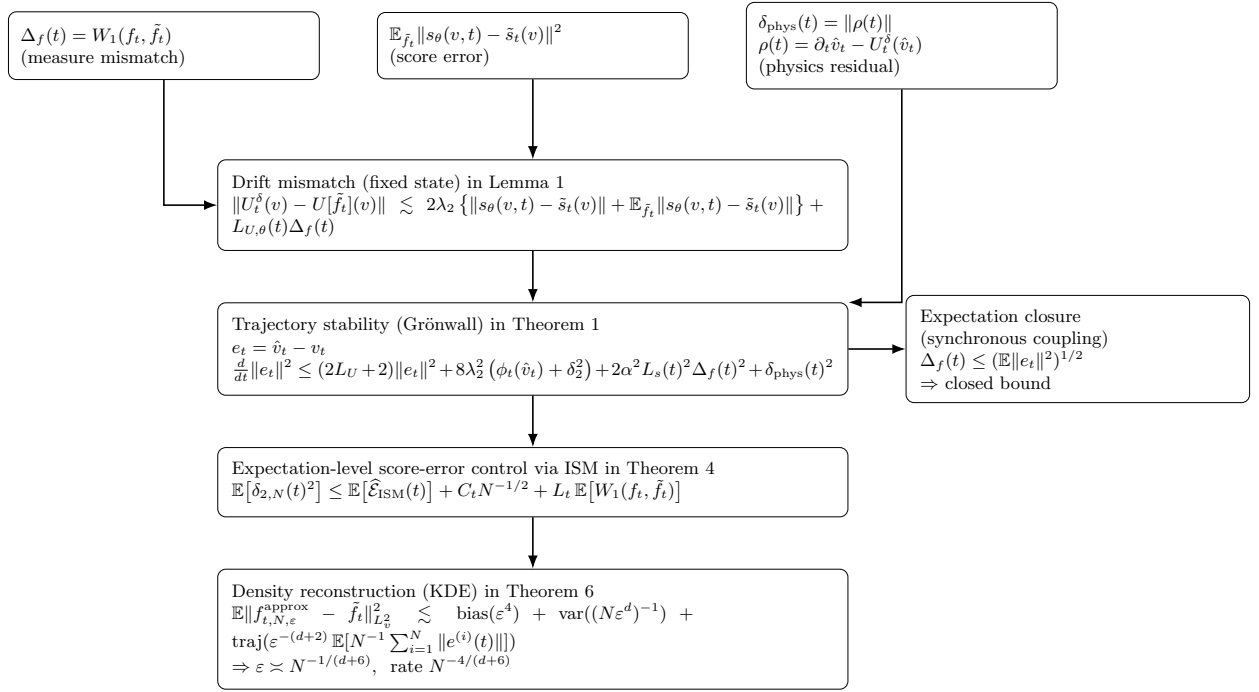
\begin{figure}[H]
\centering
\resizebox{\textwidth}{!}{%
\begin{tikzpicture}[
  font=\small,
  arr/.style={-Latex, thick},
  box/.style={draw, rounded corners, align=left, inner sep=6pt},
  ibox/.style={draw, rounded corners, align=left, inner sep=6pt, text width=50mm},
  mbox/.style={draw, rounded corners, align=left, inner sep=7pt, text width=105mm},
  rbox/.style={draw, rounded corners, align=left, inner sep=7pt, text width=55mm},
  node distance=8mm and 10mm
]

\node[ibox] (in_f)
{$\Delta_f(t)=W_1(f_t,\tilde f_t)$\\(measure mismatch)};

\node[ibox, right=10mm of in_f] (in_s)
{$\E_{ \tilde f_t}\|s_\theta(v,t)-\tilde s_t(v)\|^2$\\(score error)};

\node[ibox, right=10mm of in_s] (in_p)
{$\delta_{\mathrm{phys}}(t)=\|\rho(t)\|$\\
$\rho(t)=\partial_t\hat v_t-U_t^\delta(\hat v_t)$\\(physics residual)};

\node[mbox, below=14mm of in_s] (drift)
{Drift mismatch (fixed state) in Lemma \ref{lem:drift-properties}\\
$
    \|U_t^\delta(v)-U[\tilde f_t](v)\|\; \lesssim\;{2\lambda_2 \left\{\|s_\theta(v,t)-\tilde s_t(v)\| +\E_{ \tilde f_t}\|s_\theta(v,t)-\tilde s_t(v)\|\right\}}+ L_{U,\theta}(t)\Delta_f(t)$};

\node[mbox, below=9mm of drift] (traj)
{Trajectory stability (Gr\"onwall) in Theorem~\ref{thm:traj-error-residual} \\
$e_t=\hat v_t-v_t$\\

$\frac{d}{dt}\|e_t\|^2 \le (2L_U+2)\|e_t\|^2
+ 8\lambda^2_2\left(\phi_t(\hat v_t)+\delta^2_2\right) + 2\alpha^2 L_s(t)^2\Delta_f(t)^2 + \delta_{\mathrm{phys}}(t)^2$};

\node[mbox, below=9mm of traj] (ism)
{Expectation-level score-error control via ISM in Theorem~\ref{thm:ism_score_control}\\
$\E\big[\delta_{2,N}(t)^2\big]\le
\E\big[\widehat{\mathcal E}_{\mathrm{ISM}}(t)\big]
+
C_t N^{-1/2}
+
L_t\,\E\big[W_1(f_t,\tilde f_t)\big]$};

\node[mbox, below=9mm of ism] (kde)
{Density reconstruction (KDE) in Theorem \ref{thm:kde}\\
$\E\|f^{\mathrm{approx}}_{t,N,\varepsilon}-\tilde f_t\|_{L^2_v}^2
\lesssim \mathrm{bias}(\varepsilon^4)+\mathrm{var}((N\varepsilon^d)^{-1})
+\mathrm{traj}(\varepsilon^{-(d+2)}\,\E [N^{-1}\sum_{i=1}^N\|e^{(i)}(t)\|])$\\
$\Rightarrow \varepsilon\asymp N^{-1/(d+6)},\;\;\text{rate }N^{-4/(d+6)}$};

\node[rbox, anchor=west] (close) at ($(traj.east)+(10mm,0)$)
{Expectation closure\\(synchronous coupling)\\
$\Delta_f(t)\le(\E\|e_t\|^2)^{1/2}$\\
$\Rightarrow$ closed bound};

\draw[arr] (in_f.south) |- (drift.west);
\draw[arr] (in_s.south) -- (drift.north);
\draw[arr] (in_p.south) |- (traj.north east);

\draw[arr] (drift) -- (traj);
\draw[arr] (traj) -- (ism);
\draw[arr] (ism) -- (kde);

\draw[arr] (traj.east) -- ++(6mm,0) -- (close.west);

\end{tikzpicture}%
}
\caption{\textbf{Error-analysis flow for PINN--PM.}
Score/measure/residual errors enter the drift mismatch and yield a Gr\"onwall-type trajectory bound;
synchronous coupling closes the estimate; ISM and KDE provide downstream score/density certificates.}
\label{fig:error_flow_simple}
\end{figure}

\subsection{Trajectory Error without Residual}\label{sec:traj-error-residual}

The following lemma summarizes the two structural properties of the drift that are essential for the trajectory error analysis: the Lipschitz continuity of the exact drift and a quantitative bound on the model--drift mismatch at a fixed state. These properties ensure	stability of the flow map and provide a controlled decomposition of the drift error into score approximation and distributional components, forming the basis for the trajectory error estimate.

\begin{lem}[Properties of the exact and approximate drifts]
\label{lem:drift-properties}
Suppose Assumption~\ref{ass:regularity} holds. Let \(U[\tilde f_t]\) denote the exact drift and \(U_t^\delta\) the approximate drift associated with \(f_t\) and \(s_\theta\). Then for any \(v_1,v_2,v\in\T^d\) and \(t\in[0,T]\):

\begin{itemize}
\item {Lipschitz continuity of the exact drift:}
\[
\|U[\tilde f_t](v_1)-U[\tilde f_t](v_2)\|
\le L_U(t)\|v_1-v_2\|,
\]
where
\begin{equation}\label{eq:LU-def}
L_U(t):=\bigl(2RL_A+\lambda_2\bigr)L_s(t).
\end{equation}

\item {Drift mismatch at a fixed state:}
\[
\|U_t^\delta(v)-U[\tilde f_t](v)\|
\le
\lambda_2\,\|s_\theta(v,t)-\tilde s_t(v)\|
+
\lambda_2\,\E_{v^*\sim\tilde f_t}\|s_\theta(v^*,t)-\tilde s_t(v^*)\|
+
L_{U,\theta}(t)\,\Delta_f(t),
\]
where
\[
\Delta_f(t):=W_1(f_t,\tilde f_t),
\qquad
L_{U,\theta}(t):=\bigl(2RL_A+\lambda_2\bigr)L_\theta(t).
\]
In particular,
\[
\|U_t^\delta(v)-U[\tilde f_t](v)\|
\le
2\lambda_2\Bigl(
\|s_\theta(v,t)-\tilde s_t(v)\|
+
\E_{v^*\sim\tilde f_t}\|s_\theta(v^*,t)-\tilde s_t(v^*)\|
\Bigr)
+
L_{U,\theta}(t)\,\Delta_f(t).
\]
\end{itemize}
\end{lem}
\begin{proof}
We first prove the Lipschitz continuity of the exact drift.
Recall
\[
U[\tilde f_t](v)
=
- \int_{\T^d}
A(v-v^*)
\bigl(\tilde s_t(v)-\tilde s_t(v^*)\bigr)
\tilde f_t(v^*)\,\dd v^* .
\]
For \(v_1,v_2\in\T^d\),
\begin{align*}
U[\tilde f_t](v_1)-U[\tilde f_t](v_2)
&=
-\int_{\T^d}
\Bigl(A(v_1-v^*)-A(v_2-v^*)\Bigr)
\bigl(\tilde s_t(v_1)-\tilde s_t(v^*)\bigr)
\tilde f_t(v^*)\,\dd v^* \\
&\quad
-\int_{\T^d}
A(v_2-v^*)
\bigl(\tilde s_t(v_1)-\tilde s_t(v_2)\bigr)
\tilde f_t(v^*)\,\dd v^* .
\end{align*}
Using Assumption~\ref{ass:regularity}, and the fact that \(\tilde f_t\) has unit mass, we obtain
\begin{align*}
\|U[\tilde f_t](v_1)-U[\tilde f_t](v_2)\|
&\le
L_A\|v_1-v_2\|
\int_{\T^d}
\|\tilde s_t(v_1)-\tilde s_t(v^*)\|
\tilde f_t(v^*)\,\dd v^* \\
&\quad
+\lambda_2\,\|\tilde s_t(v_1)-\tilde s_t(v_2)\| \\
&\le
L_A\|v_1-v_2\|\,L_s(t)\sup_{v^*\in\T^d}\|v_1-v^*\|
+\lambda_2L_s(t)\|v_1-v_2\|.
\end{align*}
Since \(\sup_{v^*\in\T^d}\|v_1-v^*\|\le R:=\mathrm{diam}(\T^d)\), this yields
\[
\|U[\tilde f_t](v_1)-U[\tilde f_t](v_2)\|
\le
(2RL_A+\lambda_2)L_s(t)\|v_1-v_2\|,
\]
which proves the first claim.

We next prove the drift mismatch estimate.
Introduce the intermediate drift
\[
\bar U_t(v)
:=
-\int_{\T^d}
A(v-v^*)
\bigl(s_\theta(v,t)-s_\theta(v^*,t)\bigr)\tilde f_t(v^*)\,\dd v^*.
\]
Then
\[
U_t^\delta(v)-U[\tilde f_t](v)
=
\bigl(U_t^\delta(v)-\bar U_t(v)\bigr)
+
\bigl(\bar U_t(v)-U[\tilde f_t](v)\bigr).
\]

For the first term, define
\[
g_v(v^*)
:=
A(v-v^*)
\bigl(s_\theta(v,t)-s_\theta(v^*,t)\bigr).
\]
Then
\[
U_t^\delta(v)-\bar U_t(v)
=
-\int_{\T^d} g_v(v^*)\,\dd(f_t-\tilde f_t)(v^*).
\]
By Kantorovich--Rubinstein duality,
\[
\|U_t^\delta(v)-\bar U_t(v)\|
\le
\mathrm{Lip}(g_v)\,W_1(f_t,\tilde f_t).
\]
We estimate \(\mathrm{Lip}(g_v)\) as follows: for \(v^*,w^*\in\T^d\),
\begin{align*}
\|g_v(v^*)-g_v(w^*)\|
&\le
\|A(v-v^*)-A(v-w^*)\|\,
\|s_\theta(v,t)-s_\theta(v^*,t)\| \\
&\quad
+
\|A(v-w^*)\|\,
\|s_\theta(v^*,t)-s_\theta(w^*,t)\|.
\end{align*}
Using the Lipschitz continuity of \(A\), the bound \(\|A\|\le\lambda_2\), and the Lipschitz continuity of \(s_\theta(\cdot,t)\), we get
\[
\|g_v(v^*)-g_v(w^*)\|
\le
\bigl(2RL_A+\lambda_2\bigr)L_\theta(t)\|v^*-w^*\|.
\]
Hence
\[
\mathrm{Lip}(g_v)\le L_{U,\theta}(t):=\bigl(2RL_A+\lambda_2\bigr)L_\theta(t),
\]
and therefore
\[
\|U_t^\delta(v)-\bar U_t(v)\|
\le
L_{U,\theta}(t)\,\Delta_f(t).
\]

For the second term,
\begin{align*}
\bar U_t(v)-U[\tilde f_t](v)
&=
-\int_{\T^d}
A(v-v^*)
\Bigl[
(s_\theta(v,t)-s_\theta(v^*,t))
-(\tilde s_t(v)-\tilde s_t(v^*))
\Bigr]
\tilde f_t(v^*)\,\dd v^* \\
&=
-\int_{\T^d}
A(v-v^*)
\Bigl[
(s_\theta(v,t)-\tilde s_t(v))
-(s_\theta(v^*,t)-\tilde s_t(v^*))
\Bigr]
\tilde f_t(v^*)\,\dd v^* .
\end{align*}
Using \(\|A(z)\|\le\lambda_2\), we obtain
\begin{align*}
\|\bar U_t(v)-U[\tilde f_t](v)\|
&\le
\lambda_2
\int_{\T^d}
\Bigl(
\|s_\theta(v,t)-\tilde s_t(v)\|
+
\|s_\theta(v^*,t)-\tilde s_t(v^*)\|
\Bigr)\tilde f_t(v^*)\,\dd v^* \\
&=
\lambda_2\,\|s_\theta(v,t)-\tilde s_t(v)\|
+
\lambda_2\,\E_{v^*\sim\tilde f_t}\|s_\theta(v^*,t)-\tilde s_t(v^*)\|.
\end{align*}

Combining the two estimates yields
\[
\|U_t^\delta(v)-U[\tilde f_t](v)\|
\le
\lambda_2\,\|s_\theta(v,t)-\tilde s_t(v)\|
+
\lambda_2\,\E_{v^*\sim\tilde f_t}\|s_\theta(v^*,t)-\tilde s_t(v^*)\|
+
L_{U,\theta}(t)\,\Delta_f(t).
\]
The weaker bound with the prefactor \(2\lambda_2\) follows immediately.
\end{proof}

\subsection{Trajectory Error with PINN Residual}
We now refine the stability analysis to account for the fact that the learned neural flow $\hat v_t=\Phi_\xi(v_0,t)$ does not exactly satisfy the approximate mean-field ODE $\dot{\hat v}_t = U_t^\delta(\hat v_t).$
In practice, the neural trajectory is obtained by minimizing a physics residual rather than by exact time integration, and therefore introduces an additional perturbation term.

To capture this discrepancy, we introduce the trajectory residual
\[
\rho(t)
:=
\partial_t \hat v_t
-
U_t^\delta(\hat v_t),
\]
and denote its magnitude by $\delta_{\mathrm{phys}}(t):=\|\rho(t)\|$.

The following theorem shows that this residual contributes additively
to the Gr\"onwall-type trajectory stability bound derived above.  

\begin{thm}[Trajectory error bound with PINN residual]\label{thm:traj-error-residual}
Given $v_0$, let $v_t$ and $\hat v_t$ satisfy $\frac{d}{dt} v_t = U[\tilde f_t](v_t)$ and $\frac{d}{dt} \hat v_t = U_t^\delta(\hat v_t)$ together with $v_0 = \hat v_0$ for some $v_0$ given. Here $U_t^\delta$ is defined with respect to $f_t$ and approximate score $s_\theta(v,t)$. Set $e_t:=\hat v_t-v_t$ and define $\alpha:=2R\,L_A+\lambda_2$ and $L_U(t):=\alpha\,L_s(t)$ as in Lemma~\ref{lem:drift-properties}. Then, for all $t\in[0,T]$,
\begin{equation}\label{eq:traj-main-residual}
\frac{\dd}{\dd t}\|e_t\|^2
\;\le\; \underbrace{\bigl(2L_U(t)+2\bigr)}_{=:C_0^{\mathrm{res}}(t)}\,\|e_t\|^2
\;+\; \underbrace{8\lambda_2^2}_{=:C_1^{\mathrm{res}}}\,\left(\phi_t(\hat v_t)+\delta^2_2\right)
\;+\; \underbrace{2\alpha^2}_{=:C_2^{\mathrm{res}}}\,L_s(t)^2\,\Delta_f(t)^2
\;+\; \delta_{\mathrm{phys}}(t)^2,
\end{equation}
where $\phi_t(v):=g_t(v)^2(=\|s_\theta (v,t)-\tilde s_t(v)\|^2$) and $\delta^2_2:=\E_{v \sim \tilde f_t}[\| s_\theta(v,t)-\tilde s_t( v)\|^2]$.
Consequently,
\begin{equation}\label{eq:traj-gronwall-residual}
\|e_t\|^2
\;\le\; \int_0^t \!\exp\!\Bigl(\int_\tau^t C_0^{\mathrm{res}}(s)\,\dd s\Bigr)\,
\Bigl(C_1^{\mathrm{res}}\,\left(\phi_\tau(\hat v)+\delta^2_2(\tau)\right) + C_2^{\mathrm{res}}\,L_s(\tau)^2\,\Delta_f(\tau)^2
+ \delta_{\mathrm{phys}}(\tau)^2 \Bigr)\,\dd \tau .
\end{equation}
\end{thm}
\begin{proof}
Differentiate \(\|e_t\|^2\):
\[
\frac{\dd}{\dd t}\|e_t\|^2
= 2\,e_t\cdot\Bigl(  U^\delta_t(\hat v_t)-U[\tilde f_t](v_t) + \rho(t) \Bigr).
\]
Split the drift term as
\[
U^\delta_t(\hat v_t)-U[\tilde f_t](v_t)
= \underbrace{\bigl( U^\delta_t(\hat v_t)-U[\tilde f_t](\hat v_t)\bigr)}_{(\ast)}
+ \underbrace{\bigl(U[\tilde f_t](\hat v_t)-U[\tilde f_t](v_t)\bigr)}_{(\ast\ast)}.
\]
By Lemma~\ref{lem:drift-properties}, 
\[
\|(\ast\ast)\|\le L_U(t)\|e_t\|.
\]
Hence
\[
\frac{\dd}{\dd t}\|e_t\|^2
\le 2L_U(t)\|e_t\|^2
+ 2\|e_t\|\,\|(\ast)\|
+ 2\|e_t\|\,\|\rho(t)\|.
\]
Applying \(2ab\le a^2+b^2\) to the last two terms yields
\[
\frac{\dd}{\dd t}\|e_t\|^2
\le \bigl(2L_U(t)+2\bigr)\|e_t\|^2
+ \|(\ast)\|^2 + \|\rho(t)\|^2.
\]
It remains to bound \(\|(\ast)\|\).

Introduce the intermediate drift
\[
\bar U_t(v)
:=
-\int_{\T^d}
A(v-v^*)
\bigl(s_\theta(v,t)-s_\theta(v^*,t)\bigr)\tilde f_t(v^*)\,\dd v^*.
\]
Then
\[
(\ast)
=
\bigl(U_t^\delta(\hat v_t)-\bar U_t(\hat v_t)\bigr)
+
\bigl(\bar U_t(\hat v_t)-U[\tilde f_t](\hat v_t)\bigr).
\]

For the first term, define
\[
g_{\hat v_t}(v^*)
:=
A(\hat v_t-v^*)
\bigl(s_\theta(\hat v_t,t)-s_\theta(v^*,t)\bigr).
\]
By the Kantorovich--Rubinstein duality,
\[
\|U_t^\delta(\hat v_t)-\bar U_t(\hat v_t)\|
=
\left\|
\int_{\T^d} g_{\hat v_t}(v^*)\,\dd(f_t-\tilde f_t)(v^*)
\right\|
\le
\mathrm{Lip}(g_{\hat v_t})\,W_1(f_t,\tilde f_t).
\]
Using the Lipschitz bound from Lemma~\ref{lem:drift-properties}, we obtain
\[
\mathrm{Lip}(g_{\hat v_t})
\le \alpha\,L_s(t),
\qquad \alpha:=2RL_A+\lambda_2,
\]
and therefore
\[
\|U_t^\delta(\hat v_t)-\bar U_t(\hat v_t)\|
\le
\alpha\,L_s(t)\,\Delta_f(t).
\]

For the second term, we write
\[
\bar U_t(\hat v_t)-U[\tilde f_t](\hat v_t)
=
-\int_{\T^d}
A(\hat v_t-v^*)
\Bigl[
(s_\theta(\hat v_t,t)-s_\theta(v^*,t))
-(\tilde s_t(\hat v_t)-\tilde s_t(v^*))
\Bigr]\tilde f_t(v^*)\,\dd v^*.
\]
Hence, by \(\|A(z)\|\le \lambda_2\),
\[
\|\bar U_t(\hat v_t)-U[\tilde f_t](\hat v_t)\|
\le
\lambda_2
\int_{\T^d}
\Bigl(
\|s_\theta(\hat v_t,t)-\tilde s_t(\hat v_t)\|
+
\|s_\theta(v^*,t)-\tilde s_t(v^*)\|
\Bigr)\tilde f_t(v^*)\,\dd v^*.
\]
Therefore
\[
\|\bar U_t(\hat v_t)-U[\tilde f_t](\hat v_t)\|
\le
\lambda_2\,\|s_\theta(\hat v_t,t)-\tilde s_t(\hat v_t)\|
+
\lambda_2\,\E_{v\sim\tilde f_t}\|s_\theta(v,t)-\tilde s_t(v)\|.
\]

Combining the two bounds gives
\[
\|(\ast)\|
\le
\lambda_2\,\|s_\theta(\hat v_t,t)-\tilde s_t(\hat v_t)\|
+
\lambda_2\,\E_{v\sim\tilde f_t}\|s_\theta(v,t)-\tilde s_t(v)\|
+
\alpha\,L_s(t)\,\Delta_f(t).
\]
Using \((a+b+c)^2\le 2(a+b)^2+2c^2\le 4a^2+4b^2+2c^2\), we obtain
\[
\|(\ast)\|^2
\le
4\lambda_2^2\,\|s_\theta(\hat v_t,t)-\tilde s_t(\hat v_t)\|^2
+
4\lambda_2^2
\left(
\E_{v\sim\tilde f_t}\|s_\theta(v,t)-\tilde s_t(v)\|
\right)^2
+
2\alpha^2L_s(t)^2\Delta_f(t)^2.
\]
Finally, by Cauchy--Schwarz,
\[
\left(
\E_{v\sim\tilde f_t}\|s_\theta(v,t)-\tilde s_t(v)\|
\right)^2
\le
\E_{v\sim\tilde f_t}\|s_\theta(v,t)-\tilde s_t(v)\|^2
=
\delta_2(t)^2.
\]
Since
\[
\phi_t(\hat v_t):=\|s_\theta(\hat v_t,t)-\tilde s_t(\hat v_t)\|^2,
\]
we conclude that
\[
\|(\ast)\|^2
\le
4\lambda_2^2\,\phi_t(\hat v_t)
+
4\lambda_2^2\,\delta_2(t)^2
+
2\alpha^2L_s(t)^2\Delta_f(t)^2.
\]
Absorbing constants yields
\[
\|(\ast)\|^2
\le
8\lambda_2^2\,\bigl(\phi_t(\hat v_t)+\delta_2(t)^2\bigr)
+
2\alpha^2L_s(t)^2\Delta_f(t)^2.
\]
Substituting this into the differential inequality above gives
\eqref{eq:traj-main-residual}. Applying Gr\"onwall's inequality yields
\eqref{eq:traj-gronwall-residual}.
\end{proof}

\section{Loss-Based Accuracy Certification}\label{sec:certi}
\subsection{Oracle Score Identity}\label{OCI}
Throughout this section, $\tilde f_t$ denotes the {true} Landau density,
and $V_i\sim f_0$ are i.i.d.\ initial samples.
After one forward pass of the learned flow $\Phi_\xi$,
we obtain particles $\hat v_t^{(i)} = \Phi_\xi(V_i,t)$,
whose empirical measure is $f_t := \frac{1}{N}\sum_{i=1}^{N} \delta_{\hat v_t^{(i)}}.$

The ideal Hyv\"arinen score-matching functional is defined by
\begin{equation}\label{eq:L-ISM-oracle}
\mathcal L_{\mathrm{ISM},t}(g)
:=
\E_{ \tilde f_t}
\bigl[
\|g(v)\|^2
+
2\,\nabla_v \cdot g(v)
\bigr].
\end{equation}
It is minimized uniquely at
\[
g = \tilde s_t := \nabla_v \log \tilde f_t.
\]

In practice, the expectation with respect to $\tilde f_t$
is replaced by the empirical distribution $f_t$,
leading to
\begin{equation}\label{eq:L-ISM-emp}
\widehat{\mathcal L}_{\mathrm{ISM},t}(g)
:=
\E_{ f_t}
\bigl[
\|g(v)\|^2
+
2\,\nabla_v \cdot g(v)
\bigr].
\end{equation}
This is an unbiased Monte–Carlo estimate only if $f_t=\tilde f_t$.

\begin{lem}[Hyv\"arinen oracle identity]
\label{lem:oracle_identity}
Fix $t$ and let $g \in W^{1,2}(\T^d;\R^d)$. Then,
\[
\mathcal L_{\mathrm{ISM},t}(g)
=
\mathcal L_{\mathrm{ISM},t}(\tilde s_t)
+
\E_{\tilde f_t}\|g-\tilde s_t\|^2.
\]
Hence
\[
\delta_{2}^2(t)
=
\mathcal L_{\mathrm{ISM},t}(s_\theta)
-
\mathcal L_{\mathrm{ISM},t}(\tilde s_t),
\]
where $\delta^2_2:=\E_{ \tilde f_t}[\|s_\theta(v,t)-\tilde s_t(v)\|^2]$.
\end{lem}

\begin{proof}
By definition,
\[
\mathcal L_{\mathrm{ISM},t}(g)
=
\int_{\T^d}
\Bigl(
\|g(v)\|^2+2\nabla_v\!\cdot g(v)
\Bigr)\tilde f_t(v)\,dv.
\]
Integration by parts yields
\[
\int_{\T^d} 2(\nabla_v\!\cdot g)\,\tilde f_t\,dv
=
-2\int_{\T^d} g\cdot \nabla_v \tilde f_t\,dv.
\]
Using
\[
\nabla_v \tilde f_t
=
\tilde f_t\,\nabla_v\log \tilde f_t
=
\tilde f_t\,\tilde s_t,
\]
we obtain
\[
\int_{\T^d} 2(\nabla_v\!\cdot g)\,\tilde f_t\,dv
=
-2\int_{\T^d} g\cdot \tilde s_t\,\tilde f_t\,dv.
\]
Therefore,
\[
\mathcal L_{\mathrm{ISM},t}(g)
=
\int_{\T^d}
\bigl(
\|g\|^2-2g\cdot \tilde s_t
\bigr)\tilde f_t\,dv.
\]
Completing the square,
\[
\|g\|^2-2g\cdot \tilde s_t
=
\|g-\tilde s_t\|^2-\|\tilde s_t\|^2.
\]
Hence
\[
\mathcal L_{\mathrm{ISM},t}(g)
=
\int_{\T^d}
\bigl(
\|g-\tilde s_t\|^2-\|\tilde s_t\|^2
\bigr)\tilde f_t\,dv.
\]
On the other hand, substituting $g=\tilde s_t$ gives
\[
\mathcal L_{\mathrm{ISM},t}(\tilde s_t)
=
-\int_{\T^d}\|\tilde s_t(v)\|^2\,\tilde f_t(v)\,dv.
\]
Thus,
\[
\mathcal L_{\mathrm{ISM},t}(g)
=
\mathcal L_{\mathrm{ISM},t}(\tilde s_t)
+
\int_{\T^d}\|g(v)-\tilde s_t(v)\|^2\tilde f_t(v)\,dv
=
\mathcal L_{\mathrm{ISM},t}(\tilde s_t)
+
\mathbb E_{\tilde f_t}\|g-\tilde s_t\|^2.
\]
Taking $g=s_\theta(\cdot,t)$ yields
\[
\delta_2(t)^2
=
\mathcal L_{\mathrm{ISM},t}(s_\theta)
-
\mathcal L_{\mathrm{ISM},t}(\tilde s_t).
\]
This completes the proof.
\end{proof}

\subsection{Empirical Oracle Gap}\label{EOG}
In practice, the oracle functional  $\mathcal L_{\mathrm{ISM},t}(g)$ cannot be evaluated,
since the true density $\tilde f_t$ is unknown. Instead, training minimizes the empirical objective
$\widehat{\mathcal L}_{\mathrm{ISM},t}(g)$ defined with respect to the particle measure $f_t$.

To relate the training-time quantity  $\widehat{\mathcal L}_{\mathrm{ISM},t}(g)$
to the oracle score error characterized in Lemma~\ref{lem:oracle_identity},
we quantify the deviation between the empirical and oracle functionals. This deviation consists of two contributions: (i) Monte–Carlo sampling error, and (ii) distributional mismatch between $f_t$ and $\tilde f_t$. 

\begin{lem}[Empirical--oracle gap]
\label{lem:emp_gap_section}
For any smooth function $g$, let
\[
\widehat{\mathcal L}_{\mathrm{ISM},t}(g)
=
\frac1N\sum_{i=1}^N h_t(\hat v_t^{(i)}),
\qquad
h_t(v)=\|g(v)\|^2+2\nabla\!\cdot g(v),
\]
and assume $h_t$ is $L_t$--Lipschitz continuous.
Then
\begin{align}\label{ineq:score}
\Big|
\widehat{\mathcal L}_{\mathrm{ISM},t}(g)
-
\mathcal L_{\mathrm{ISM},t}(g)
\Big|
\le
L_t\,W_1(f_t,\tilde f_t),
\end{align}
where $f_t:=\frac{1}{N}\sum_{i=1}^N\delta_{\hat v_t^{(i)}}$ and the particle trajectories $\hat v_t^{(i)}$ are defined in \eqref{eq:flow-map}, and $\tilde f_t$ denotes the true distribution. 
\end{lem}

\begin{proof}
We can rewrite L.H.S in \eqref{ineq:score} as follows.
\[
\widehat{\mathcal L}_{\mathrm{ISM},t}(g)
-
\mathcal L_{\mathrm{ISM},t}(g)
=
\Big(
\E_{f_t}h_t
-
\E_{\tilde f_t}h_t
\Big).
\]
Since $h_t$ is $L_t$--Lipschitz, Kantorovich--Rubinstein duality gives
\[
\big|
\E_{f_t}h_t
-
\E_{\tilde f_t}h_t
\big|
\le
L_t\,W_1(f_t,\tilde f_t).
\]
Combining the two steps yields the result.
\end{proof}

\subsection{Deterministic Loss-to-Dynamics Certificate}
\label{sec:deterministic_certificate}

Let $\{v_t^{(i)}\}_{i=1}^N$ denote the exact particle trajectories solving
\[
\dot v_t^{(i)} = U[\tilde f_t](v_t^{(i)}),
\qquad
v_0^{(i)}=V_i,
\]
and let $\{\hat v_t^{(i)}\}_{i=1}^N$ denote the approximate trajectories
generated by the neural flow as in the previous section.

Define the trajectory error
\[
e_t^{(i)}:=\hat v_t^{(i)}-v_t^{(i)},
\qquad
E(t):=\frac1N\sum_{i=1}^N\|e_t^{(i)}\|^2,
\]
the particle-level score error
\[
\delta_{2,N}(t)^2
:=
\frac1N\sum_{i=1}^N
\|s_\theta(v_t^{(i)},t)-\tilde s_t(v_t^{(i)})\|^2,
\]
and the physics residual
\[
\delta_{\mathrm{phys}}(t)^2
:=
\frac1N\sum_{i=1}^N
\|\partial_t \hat v_t^{(i)}-U_t^\delta(\hat v_t^{(i)})\|^2.
\]
\begin{thm}[Deterministic particle-level loss-to-dynamics certificate]
\label{thm:det_particle_cert}
Suppose Assumption~\ref{ass:regularity} holds. Then the mean-squared trajectory error satisfies the closed differential inequality
\begin{equation}
\label{eq:det_closed_DI_final}
\frac{d}{dt}E(t)
\le
a(t)\,E(t)
+
b(t)\,\delta_{2,N}(t)^2
+
\bar \delta_{\mathrm{phys}}(t)^2,
\end{equation}
where
\[
a(t)
=
C_0^{\mathrm{res}}(t)
+
2C_1^{\mathrm{res}}L_g(t)^2
+
C_2^{\mathrm{res}},
\qquad
b(t)
=
3C_1^{\mathrm{res}},
\]
and $C_0^{\mathrm{res}}, C_1^{\mathrm{res}}, C_2^{\mathrm{res}}$
are the constants appearing in Theorem~\ref{thm:traj-error-residual}.

Consequently,
\begin{equation}
\label{eq:det_gronwall_final}
E(t)
\le
\int_0^t
\exp\!\Big(\int_\tau^t a(s)\,ds\Big)
\Big(
b(\tau)\,\delta_{2,N}(\tau)^2
+
\bar \delta_{\mathrm{phys}}(\tau)^2
\Big)\,d\tau.
\end{equation}

Moreover, defining the empirical measures
\[
f_t:=\frac1N\sum_{i=1}^N\delta_{\hat v_t^{(i)}},
\qquad
\tilde f_t^N:=\frac1N\sum_{i=1}^N\delta_{v_t^{(i)}},
\]
we have the deterministic coupling bound
\begin{equation}
W_1(f_t,\tilde f_t^N)\le E(t)^{1/2}.
\end{equation}

In particular, if $\delta_{2,N}(t)\equiv 0$ and
$\delta_{\mathrm{phys}}(t)\equiv 0$, then $E(t)\equiv 0$.
\end{thm}

\begin{proof}
We start from the pointwise trajectory inequality of Theorem~\ref{thm:traj-error-residual}.
For each $i\in\{1,\dots,N\}$ and $t\in[0,T]$, set
\[
e_t^{(i)}:=\hat v_t^{(i)}-v_t^{(i)}.
\]
Then Theorem~\ref{thm:traj-error-residual} yields
\begin{equation}\label{eq:det_pf_pointwise}
\frac{d}{dt}\|e_t^{(i)}\|^2
\le
C_0^{\mathrm{res}}(t)\|e_t^{(i)}\|^2
+
C_1^{\mathrm{res}}\Big(\phi_t(\hat v_t^{(i)})+\delta_2(t)^2\Big)
+
C_2^{\mathrm{res}}\,\Delta_f(t)^2
+
\|\rho^{(i)}(t)\|^2,
\end{equation}
where
\[
\phi_t(v):=\|s_\theta(v,t)-\tilde s_t(v)\|^2,
\qquad
\Delta_f(t):=W_1(f_t,\tilde f_t),
\qquad
\rho^{(i)}(t):=\partial_t\hat v_t^{(i)}-U_t^\delta(\hat v_t^{(i)}).
\]

Define the mean-squared trajectory error
\[
E(t):=\frac1N\sum_{i=1}^N\|e_t^{(i)}\|^2,
\qquad
\delta_{\mathrm{phys}}(t)^2:=\frac1N\sum_{i=1}^N\|\rho^{(i)}(t)\|^2,
\qquad
f_t:=\frac1N\sum_{i=1}^N\delta_{\hat v_t^{(i)}}.
\]
Averaging \eqref{eq:det_pf_pointwise} over $i=1,\dots,N$ gives
\begin{equation}\label{eq:det_pf_avg}
\frac{d}{dt}E(t)
\le
C_0^{\mathrm{res}}(t)E(t)
+
C_1^{\mathrm{res}}\Big(\E_{f_t}\phi_t(\hat v_t)+\delta_2(t)^2\Big)
+
C_2^{\mathrm{res}}\Delta_f(t)^2
+
\bar \delta_{\mathrm{phys}}(t)^2,
\end{equation}
since by definition $\frac1N\sum_{i=1}^N\phi_t(\hat v_t^{(i)})=\E_{f_t}\phi_t$.

To obtain a deterministic closed estimate, we compare $f_t$ to the
{exact empirical measure} induced by the same initial particles.
Let
\[
\tilde f_t^N:=\frac1N\sum_{i=1}^N\delta_{v_t^{(i)}}.
\]
Consider the canonical coupling
\[
\pi_t:=\frac1N\sum_{i=1}^N\delta_{(\hat v_t^{(i)},\,v_t^{(i)})}\in\Pi(f_t,\tilde f_t^N).
\]
By the Kantorovich formulation of $W_1$,
\[
W_1(f_t,\tilde f_t^N)
\le
\iint_{\T^d\times\T^d}\|x-y\|\,d\pi_t(x,y)
=
\frac1N\sum_{i=1}^N\|e_t^{(i)}\|
\le
\Big(\frac1N\sum_{i=1}^N\|e_t^{(i)}\|^2\Big)^{1/2}
=
E(t)^{1/2}.
\]
Hence,
\begin{equation}\label{eq:det_pf_W1_E}
W_1(f_t,\tilde f_t^N)^2\le E(t).
\end{equation}
Define
\[
g_t(v):=\|s_\theta(v,t)-\tilde s_t(v)\|,
\qquad
\phi_t(v)=g_t(v)^2.
\]
By Jensen's inequality,
\begin{equation}\label{eq:det_pf_Jensen}
\E_{f_t}\phi_t=\E_{f_t}g_t^2 \le \big(\E_{f_t}g_t\big)^2.
\end{equation}
Next, by Kantorovich--Rubinstein duality applied to the Lipschitz function $g_t$,
for any $\mu,\nu$ on $\T^d$ we have
$|\E_\mu g_t-\E_\nu g_t|\le L_g(t)\,W_1(\mu,\nu)$.
Taking $\mu=f_t$ and $\nu=\tilde f_t^N$ yields
\begin{equation}\label{eq:det_pf_KR}
\E_{f_t}g_t
\le
\E_{\tilde f_t^N}g_t + L_g(t)\,W_1(f_t,\tilde f_t^N),
\end{equation}
since $g_t$ is Lipschitz continuous. Moreover, by Cauchy--Schwarz,
\[
\E_{\tilde f_t^N}g_t
\le
\big(\E_{\tilde f_t^N}g_t^2\big)^{1/2}
=
\big(\E_{\tilde f_t^N}\phi_t\big)^{1/2}
=:\delta_{2,N}(t),
\]
where we define the (deterministic) particle-level $L^2$ score error
\[
\delta_{2,N}(t)^2:=\E_{\tilde f_t^N}\phi_t
=\frac1N\sum_{i=1}^N\|s_\theta(v_t^{(i)},t)-\tilde s_t(v_t^{(i)})\|^2.
\]
Combining \eqref{eq:det_pf_KR} with \eqref{eq:det_pf_W1_E} gives
\[
\E_{f_t}g_t
\le
\delta_{2,N}(t) + L_g(t)\,E(t)^{1/2}.
\]
Squaring and using $(x+y)^2\le 2x^2+2y^2$ yields
\begin{equation}\label{eq:det_pf_phi_close}
\E_{f_t}\phi_t
\le
\big(\E_{f_t}g_t\big)^2
\le
2\,\delta_{2,N}(t)^2 + 2\,L_g(t)^2\,E(t).
\end{equation}
Insert \eqref{eq:det_pf_W1_E} and \eqref{eq:det_pf_phi_close} into
\eqref{eq:det_pf_avg} to obtain
\[
\frac{d}{dt}E(t)
\le
\Big(C_0^{\mathrm{res}}(t)+2C_1^{\mathrm{res}}L_g(t)^2+C_2^{\mathrm{res}}\Big)\,E(t)
+
C_1^{\mathrm{res}}\Big(2\,\delta_{2,N}(t)^2+\delta_2(t)^2\Big)
+
\bar \delta_{\mathrm{phys}}(t)^2.
\]
In particular, defining
\[
a(t):=C_0^{\mathrm{res}}(t)+2C_1^{\mathrm{res}}L_g(t)^2+C_2^{\mathrm{res}},
\qquad
b(t):=C_1^{\mathrm{res}}\cdot 3,
\]
and noting that $\delta_2(t)^2$ can be replaced by $\delta_{2,N}(t)^2$ in the
purely deterministic particle-level statement (or bounded by it),
we obtain the closed form
\[
\frac{d}{dt}E(t)
\le
a(t)\,E(t)
+
b(t)\,\delta_{2,N}(t)^2
+
\bar \delta_{\mathrm{phys}}(t)^2.
\]
Since $E(0)=0$ (because $\hat v_0^{(i)}=v_0^{(i)}$), Gr\"onwall's inequality yields
\[
E(t)
\le
\int_0^t
\exp\!\Big(\int_\tau^t a(s)\,ds\Big)
\Big(
b(\tau)\,\delta_{2,N}(\tau)^2+ \bar \delta_{\mathrm{phys}}(\tau)^2
\Big)\,d\tau.
\]
Finally, \eqref{eq:det_pf_W1_E} gives $W_1(f_t,\tilde f_t^N)\le E(t)^{1/2}$.
\end{proof}

\subsection{From Particle System to PDE Solution}\label{sec:particle_to_pde}
Recall that $V_i \stackrel{\text{i.i.d.}}{\sim} f_0$.
Let $v_t^{(i)} := T_t(V_i)$ denote the exact mean-field characteristics,
and let
\[
\tilde f_t^N := \frac1N \sum_{i=1}^N \delta_{v_t^{(i)}}
\]
be the associated empirical measure.
Similarly, let
\[
f_t := \frac1N \sum_{i=1}^N \delta_{\hat v_t^{(i)}}
\]
denote the empirical measure generated by the learned neural flow.

\begin{thm}[Particle-to-PDE lifting]
\label{thm:particle_to_pde}
Under Assumption~\ref{ass:regularity}, for all $t\in[0,T]$,
\[
\mathbb E\, W_1(f_t,\tilde f_t)
\le
\mathbb E\,E(t)^{1/2}
+
\mathbb E\,W_1(\tilde f_t^N,\tilde f_t),
\]
where
\[
E(t) := \frac1N \sum_{i=1}^N \|\hat v_t^{(i)}-v_t^{(i)}\|^2.
\]
\end{thm}

\begin{proof}
By triangle inequality,
\[
W_1(f_t,\tilde f_t)
\le
W_1(f_t,\tilde f_t^N)
+
W_1(\tilde f_t^N,\tilde f_t).
\]
Using the canonical coupling,
\[
W_1(f_t,\tilde f_t^N)
\le
E(t)^{1/2}.
\]
Taking expectation yields the result.
\end{proof}
\subsection{Score-error control via implicit score matching}
\label{sec:ism_control}

We now relate the particle-level score error appearing in Theorem~\ref{thm:det_particle_cert} to the implicit score-matching objective minimized during training.

Recall the oracle Hyv\"arinen functional
\[
\mathcal L_{\mathrm{ISM},t}(g)
:=
\E_{\tilde f_t}
\bigl[
\|g(v)\|^2
+
2\,\nabla_v \cdot g(v)
\bigr].
\]
By Lemma~\ref{lem:oracle_identity},
\[
\delta_2(t)^2
=
\mathcal L_{\mathrm{ISM},t}(s_\theta)
-
\mathcal L_{\mathrm{ISM},t}(\tilde s_t),
\]
where
\[
\delta_2(t)^2
:=
\E_{\tilde f_t}\|s_\theta-\tilde s_t\|^2.
\]
Since the oracle expectation is unavailable in practice,
training minimizes the empirical objective
\[
\widehat{\mathcal L}_{\mathrm{ISM},t}(g)
=
\E_{f_t}
\bigl[
\|g(v)\|^2
+
2\,\nabla_v \cdot g(v)
\bigr].
\]
The following result provides an expectation-level control
of the particle-level score error.
\begin{thm}[Expectation-level score-error control via ISM]
\label{thm:ism_score_control}
Under Assumption~\ref{ass:regularity}, for each $t\in[0,T]$,
\begin{equation}
\label{eq:ism_score_control_correct}
\E\delta_{2,N}(t)^2
\le
\E\widehat{\mathcal E}_{\mathrm{ISM}}(t)
+
2L_t\,\E W_1(f_t,\tilde f_t),
\end{equation}
where
\[
\delta_{2,N}(t)^2
:=
\frac1N\sum_{i=1}^N
\|s_\theta(v_t^{(i)},t)-\tilde s_t(v_t^{(i)})\|^2,
\qquad v_t^{(i)}:=T_t(V_i),\; V_i\stackrel{\mathrm{i.i.d.}}{\sim}f_0,
\]
and
\[
\widehat{\mathcal E}_{\mathrm{ISM}}(t)
:=
\widehat{\mathcal L}_{\mathrm{ISM},t}(s_\theta)
-
\inf_{g\in\mathcal G}
\widehat{\mathcal L}_{\mathrm{ISM},t}(g).
\]
\end{thm}
\begin{proof}
Let $\phi_t(v):=\|s_\theta(v,t)-\tilde s_t(v)\|^2$.
Since $v_t^{(i)}=T_t(V_i)$ and $V_i\stackrel{\mathrm{i.i.d.}}{\sim}f_0$,
we have $v_t^{(i)}\stackrel{\mathrm{i.i.d.}}{\sim}\tilde f_t$ and thus
\[
\E\delta_{2,N}(t)^2
=
\E\left[\frac1N\sum_{i=1}^N \phi_t(v_t^{(i)})\right]
=
\E_{\tilde f_t}\phi_t(v)
=:\delta_2(t)^2.
\]
Hence it suffices to bound $\delta_2(t)^2$.

By Lemma~\ref{lem:oracle_identity},
\[
\delta_2(t)^2
=
\mathcal L_{\mathrm{ISM},t}(s_\theta)
-
\mathcal L_{\mathrm{ISM},t}(\tilde s_t).
\]

Lemma~\ref{lem:emp_gap_section} implies (after taking unconditional expectation)
\[
\E\Big|
\widehat{\mathcal L}_{\mathrm{ISM},t}(g)-\mathcal L_{\mathrm{ISM},t}(g)
\Big|
\le L_t\,\E W_1(f_t,\tilde f_t).
\]
Using $x\le y+|x-y|$ and $x\ge y-|x-y|$, we obtain
\[
\mathcal L_{\mathrm{ISM},t}(s_\theta)
\le
\E\widehat{\mathcal L}_{\mathrm{ISM},t}(s_\theta)
+L_t\,\E W_1(f_t,\tilde f_t),
\]
and
\[
\mathcal L_{\mathrm{ISM},t}(\tilde s_t)
\ge
\E\widehat{\mathcal L}_{\mathrm{ISM},t}(\tilde s_t)
-L_t\,\E W_1(f_t,\tilde f_t).
\]
Subtracting yields
\begin{equation}
\label{eq:delta2_empirical_expect}
\delta_2(t)^2
\le
\E\!\left[\widehat{\mathcal L}_{\mathrm{ISM},t}(s_\theta)
-\widehat{\mathcal L}_{\mathrm{ISM},t}(\tilde s_t)\right]
+
2L_t\,\E W_1(f_t,\tilde f_t).
\end{equation}

Since $\inf_{g\in\mathcal G}\widehat{\mathcal L}_{\mathrm{ISM},t}(g)\le
\widehat{\mathcal L}_{\mathrm{ISM},t}(\tilde s_t)$ pointwise, we have pointwise
\[
\widehat{\mathcal L}_{\mathrm{ISM},t}(s_\theta)
-\widehat{\mathcal L}_{\mathrm{ISM},t}(\tilde s_t)
\le
\widehat{\mathcal L}_{\mathrm{ISM},t}(s_\theta)
-\inf_{g\in\mathcal G}\widehat{\mathcal L}_{\mathrm{ISM},t}(g)
=
\widehat{\mathcal E}_{\mathrm{ISM}}(t).
\]
Taking expectation and inserting into \eqref{eq:delta2_empirical_expect} gives
\[
\delta_2(t)^2
\le
\E\widehat{\mathcal E}_{\mathrm{ISM}}(t)
+
2L_t\,\E W_1(f_t,\tilde f_t).
\]
Absorbing the factor $2$ into constants yields \eqref{eq:ism_score_control_correct},
completing the proof.
\end{proof}

\subsection{Master certificate in $W_1$}
\label{sec:master_w1}

In this subsection we summarize the end-to-end training-to-deployment guarantee
in Wasserstein distance. Throughout, $\E[\cdot]$ denotes expectation with respect
to the sampling randomness of $(V_1,\dots,V_N)$.

Recall the learned and reference empirical measures
\[
f_t:=\frac1N\sum_{i=1}^N\delta_{\hat v_t^{(i)}},
\qquad
\tilde f_t^N:=\frac1N\sum_{i=1}^N\delta_{v_t^{(i)}},
\qquad
v_t^{(i)}:=T_t(V_i),\quad V_i\stackrel{\mathrm{i.i.d.}}{\sim}f_0,
\]
and the mean-field Landau solution $\tilde f_t=T_t\# f_0$.
Recall the mean-squared trajectory error
\[
E(t):=\frac1N\sum_{i=1}^N\|\hat v_t^{(i)}-v_t^{(i)}\|^2.
\]
We also write the (time-local) physics residual energy
\[
\bar \delta_{\mathrm{phys}}(t)^2:=\frac1N\sum_{i=1}^N \|\rho^{(i)}(t)\|^2,
\qquad
\rho^{(i)}(t):=\partial_t\hat v_t^{(i)}-U_t^\delta(\hat v_t^{(i)}).
\]

To lift from the particle system to the PDE, we invoke a standard propagation-of-chaos bound.
Since we work on $\T^d$ with bounded/regularized kernels, we state it as an assumption.

\begin{assumption}\label{ass:poc_w1}
There exists $C_{\mathrm{mc}}:[0,T]\to(0,\infty)$ such that, for all $t\in[0,T]$,
\[
\E\,W_1(\tilde f_t^N,\tilde f_t)\le C_{\mathrm{mc}}(t)\,N^{-1/2}.
\]
\end{assumption}

We relate the time-integrated residual energy to the measured physics loss.
\begin{assumption}[Smallness assumption]\label{ass:phys_res_energy}
Assume that
\begin{equation}
\label{eq:phys_res_energy}
\int_0^T \E\,\bar \delta_{\mathrm{phys}}(t)^2\,dt \le \varepsilon_{\mathrm{phys}}\quad \text{and}\quad \int_0^T \E \widehat{\mathcal E}_{\mathrm{ISM}}\,dt \le \varepsilon_{\mathrm{score}}
\end{equation}
for some $\varepsilon_{\mathrm{phys}},\varepsilon_\mathrm{score}|\ge 0$.
\end{assumption}

\begin{remark}[Connection to collocation training]
In practice, $\bar \delta_{\mathrm{phys}}(t)^2$ is monitored only at finitely many
collocation times and particles through the empirical loss $\mathcal L_{\mathrm{phys}}$.
Assumption~\ref{ass:phys_res_energy} can be interpreted as requiring that the
learned neural flow attains a small time-averaged residual along deployment-time
trajectories.
\end{remark}

The following theorem shows that, under the residual-energy assumption
\eqref{eq:phys_res_energy}, the cumulative score excess risk and the
time-averaged physics residual jointly control the Wasserstein discrepancy
between the learned empirical measure $f_t$ and the true Landau solution
$\tilde f_t$. In particular, the bound is independent of the training
procedure and depends only on the population-level residual energy, thereby separating optimization effects from dynamical stability.

\begin{thm}[Master certificate in $W_1$]
\label{thm:master_w1}
Let Assumption~\ref{ass:regularity}, \ref{ass:poc_w1}, and \ref{ass:phys_res_energy} hold
Then there exists a constant $C(T)>0$ such that for all $t\in[0,T]$,
\begin{equation}
\label{eq:master_w1}
\E\,W_1(f_t,\tilde f_t)
\le
C(T)(
\varepsilon_{\mathrm{score}}
+\varepsilon_{\mathrm{phys}}
+N^{-1/2}
)^{1/2},
\end{equation}
where $\widehat{\mathcal E}_{\mathrm{ISM}}(t)$ is the empirical ISM excess risk
defined in Theorem~\ref{thm:ism_score_control}.
The constant $C(T)$ depends only on the drift regularity parameters
(e.g.\ $\sup_{s\le T}L_U(s)$ and $\sup_{s\le T}L_g(s)$), but not on $N$.
\end{thm}

\begin{proof}
By the canonical coupling between $(\hat v_t^{(i)},v_t^{(i)})$,
\[
W_1(f_t,\tilde f_t^N)\le E(t)^{1/2}.
\]
Taking expectation and using Jensen gives
\begin{equation}
\label{eq:W1_to_E}
\E\,W_1(f_t,\tilde f_t^N)\le (\E\,E(t))^{1/2}.
\end{equation}

By triangle inequality and Assumption~\ref{ass:poc_w1},
\begin{equation}
\label{eq:W1_triangle_poc}
\E\,W_1(f_t,\tilde f_t)
\le
\E\,W_1(f_t,\tilde f_t^N)+\E\,W_1(\tilde f_t^N,\tilde f_t)
\le
(\E\,E(t))^{1/2}+C_{\mathrm{mc}}(t)N^{-1/2}.
\end{equation}

Taking expectation in the deterministic differential inequality
of Theorem~\ref{thm:det_particle_cert} yields
\[
\frac{d}{dt}\E E(t)\le a(t)\E E(t)+b(t)\E\delta_{2,N}(t)^2+\E \bar \delta_{\mathrm{phys}}(t)^2.
\]
Gr\"onwall's lemma gives
\begin{equation}
\label{eq:E_gronwall_expect}
\E E(t)\le C(T)\int_0^t \big(\E\delta_{2,N}(\tau)^2+\E \bar \delta_{\mathrm{phys}}(\tau)^2\big)\,d\tau.
\end{equation}

By Theorem~\ref{thm:ism_score_control},
\[
\E\,\delta_{2,N}(t)^2
\le
\E\widehat{\mathcal E}_{\mathrm{ISM}}(t)
+L_t\E W_1(f_t,\tilde f_t).
\]
Insert this into \eqref{eq:E_gronwall_expect}, then use \eqref{eq:W1_triangle_poc}
to substitute $\E W_1(f_t,\tilde f_t)$ by $(\E E(t))^{1/2}+N^{-1/2}$.
This yields an inequality of the form
\[
(\E E(t))^{1/2}\le A(t)+\kappa(T)(\E E(t))^{1/2},
\]
where $A(t)$ contains the integrals of $\widehat{\mathcal E}_{\mathrm{ISM}}$
and $\E \bar \delta_{\mathrm{phys}}^2$, plus $N^{-1/2}$ terms, and $\kappa(T)$ depends only
on the drift/score Lipschitz constants. Setting $\kappa(T)<1$ and absorbing $\kappa(T)$ term (for fixed finite $T$)
gives
\[
(\E E(t))^{1/2}\le C(T)\Big(\int_0^t \E\widehat{\mathcal E}_{\mathrm{ISM}}
+\int_0^t \E\bar \delta_{\mathrm{phys}}^2 + N^{-1/2}\Big)^{1/2}.
\]
Finally, substitute into \eqref{eq:W1_triangle_poc} to obtain \eqref{eq:master_w1}.
\end{proof}

\subsection{Density reconstruction via KDE}
\label{sec:density_corollary}

We now translate the Wasserstein/trajectory guarantees into an $L^2_v$ density error
for the reconstructed density obtained by kernel density estimation (KDE) that is obtained from our particle approximator $f_t$.
For $\varepsilon>0$, define
\[
f^{\mathrm{approx}}_{t,N,\varepsilon}(v)
:=
\frac1N\sum_{i=1}^N \varepsilon^{-d}K\!\Big(\frac{v-\hat v_t^{(i)}}{\varepsilon}\Big),
\]
where $K\in C^2(\R^d)$ is a bounded symmetric kernel with $\int K=1$.

We state the density reconstruction estimate.
\begin{thm}[Density reconstruction bound]
\label{thm:kde}
Suppose Assumption~\ref{ass:regularity}, \ref{ass:poc_w1}, and \ref{ass:phys_res_energy} hold. Then for any bandwidth $\varepsilon>0$ and all $t\in[0,T]$,
\begin{align}
\E\big[\|f^{\mathrm{approx}}_{t,N,\varepsilon}-\tilde f_t\|_{L^2_v}^2\big]
\;\le\;
C\Big(
\varepsilon^4
+\frac{1}{N\varepsilon^d}
+\varepsilon^{-(d+2)}\,\E E(t)
\Big),
\label{eq:density_bound_general}
\end{align}
where $C$ depends only on $K$, $d$, and $\|\tilde f_t\|_{W^{2,\infty}}$.
In particular, substituting \eqref{eq:E_gronwall_expect} and \eqref{eq:master_w1},
we obtain the fully training-controlled density certificate
\begin{align}
\E\big[\|f^{\mathrm{approx}}_{t,N,\varepsilon}-\tilde f_t\|_{L^2_v}^2\big]
\;\le\;
C(T)\Bigg(
\varepsilon^4
+\frac{1}{N\varepsilon^d}
+\varepsilon^{-(d+2)}
[
\varepsilon_{\mathrm{score}}+\varepsilon_{\mathrm{phy}}
+N^{-1/2}
]
\Bigg).
\label{eq:density_bound_training}
\end{align}
\end{thm}

\begin{proof}
The details of the proof are given in Appendix~\ref{app:density_proof}.
\end{proof}

\begin{remark}[Bandwidth choice]
Balancing the leading terms in \eqref{eq:density_bound_training} yields the standard trade-off:
for fixed $t$, choosing $\varepsilon$ to balance $\varepsilon^{4}$ and
$\varepsilon^{-(d+2)}\times(\text{training-controlled term})$
gives an explicit rate in $N$ once the training-controlled term is specified.
In our experiments we use the bandwidth values reported in Appendix~\ref{app:exp_config}.
\end{remark}

\section{Numerical Experiments}\label{sec:exp}

This section validates the theoretical error decomposition developed in Section~\ref{sec:Neu cha flow} and~\ref{sec:certi}. In particular, we empirically examine how the three identified error sources:
(i) score approximation error, 
(ii) trajectory residual error, and 
(iii) particle approximation error - translate into observable discrepancies in trajectory accuracy, density reconstruction, 
and macroscopic structure preservation.

We consider both analytical benchmarks, where exact solutions are available 
(BKW tests), and reference-free configurations (Gaussian mixture, Rosenbluth, 
and anisotropic data), where qualitative structural stability is assessed.


\subsection{Experimental setup}
\label{subsec:setup}

We compare the proposed \textbf{PINN--PM} with two time-stepping baselines:
\textbf{SBP}~\cite{YL2025} and a deterministic \textbf{Blob} particle method~\cite{JKF2019}. For ablations of our method, we additionally report: (i) \textbf{PINN--particle} (global flow $\Phi_\xi$ queried directly), and (ii) \textbf{PINN--score} (Euler-integrated trajectories driven by the learned score).

\paragraph{Neural network architectures.}
For PINN--PM, we use two fully connected networks:
(i) a {trajectory network} $\Phi_\xi(v_0,t)$ and
(ii) a {score network} $s_\theta(v,t)$.
Both networks employ SiLU activations and are trained jointly
using the loss~\eqref{eq:full-loss}.
Architectural details (depth, width, learning rates) are reported in Appendix~\ref{app:exp_config} for reproducibility.

\paragraph{Particle initialization.}
Particles are initialized by i.i.d.\ sampling from the prescribed initial
distribution $f_0$ using rejection or Gaussian sampling, depending on the test case.

\paragraph{Two evaluation protocols.}
To avoid conflating transport accuracy and density reconstruction accuracy,
we evaluate performance under the following two protocols.

\smallskip
\noindent\textbf{(P1) Transport/trajectory evaluation (fixed initial set).}
We sample a common set of $N_{\mathrm{eval}}=10^4$ initial particles
$\{V_i\}_{i=1}^{N_{\mathrm{eval}}}\sim f_0$ and compute trajectories
$\hat v_t^{(i)}$ for each method:
\begin{itemize}
\item \textbf{PINN--particle:} $\hat v_t^{(i)}=\Phi_\xi(V_i,t)$.
\item \textbf{PINN--score:} $\hat v_t^{(i)}$ is obtained by Euler integration driven by the learned score.
\item \textbf{SBP and Blob:} $\hat v_t^{(i)}$ are obtained by each solver's particle evolution.
\end{itemize}
When an analytical score (hence a reference flow) is available (BKW tests),
we generate the reference $v_t^{(i)}$ by Euler integration using the analytical score.

\smallskip
\noindent\textbf{(P2) Density reconstruction (KDE).}
At each snapshot time $t$, given particle locations $\hat v_t$, we reconstruct the density
by KDE and denote it by $\hat f_t:=\mathrm{KDE}(\hat v_t)$.
For KDE particle counts, we follow each method's practical output:
\begin{itemize}
\item \textbf{PINN--particle:} we draw $10^5$ samples at each $t$ via $\Phi_\xi(\cdot,t)$ (BKW 2D/3D).
\item \textbf{SBP~\cite{YL2025}:} we use the particles produced by the method, and the details are provided in Appendix~\ref{app:exp_config}.
\item \textbf{Blob~\cite{JKF2019}:} we use the same particle counts as SBP for each benchmark.
\end{itemize}
\noindent\textbf{Density $L^2$ error.}
Using the KDE reconstruction $\hat f_t$, we report the relative $L^2$ error (on the evaluation grid) with respect to the analytical density when available. Specifically, we use a uniform grid: $100\times 100$ for BKW 2D and $30\times 30\times 30$ for BKW 3D.

\smallskip
\noindent\textbf{Trajectory $L^2$ error.}
Under protocol (P1), we report the relative trajectory error
\[
\mathrm{Err}_{\mathrm{traj}}(t)
:=\frac{\sum_{i=1}^{N_{\mathrm{eval}}}\|\hat v_t^{(i)}-v_t^{(i)}\|^2}
{\sum_{i=1}^{N_{\mathrm{eval}}}\|v_t^{(i)}\|^2},
\]
where $v_t^{(i)}$ is the reference trajectory computed using the analytical score (BKW tests).

\smallskip
\noindent\textbf{Kinetic energy.}
We compute the kinetic energy along the evaluated trajectories:
\[
\mathcal E(t):=\frac{1}{2N_{\mathrm{eval}}}\sum_{i=1}^{N_{\mathrm{eval}}}\|\hat v_t^{(i)}\|^2.
\]

\smallskip
\noindent\textbf{Relative Fisher divergence (score accuracy).}
When an analytical score $\tilde s_t$ is available, we report
\[
\mathrm{RFD}(t)
:=\frac{\sum_{i=1}^{N_{\mathrm{eval}}}\|\hat s_t(\hat v_t^{(i)})-\tilde s_t(\hat v_t^{(i)})\|^2}
{\sum_{i=1}^{N_{\mathrm{eval}}}\|s_t^{\mathrm{ana}}(\hat v_t^{(i)})\|^2}.
\]
For \textbf{PINN--score} and \textbf{SBP}, the scores are evaluated along the Euler-integrated trajectories
generated by each method's score model.

\smallskip
\noindent\textbf{Entropy dissipation proxy.}
We additionally report the empirical entropy decay rate proxy
\[
\mathcal D(t)
:=
-\frac{1}{N_{\mathrm{eval}}^2}\sum_{i,j}
\hat s_t(\hat v_t^{(i)})^\top A(\hat v_t^{(i)}-\hat v_t^{(j)})
\bigl(\hat s_t(\hat v_t^{(i)})-\hat s_t(\hat v_t^{(j)})\bigr),
\]
computed for \textbf{PINN--score} and \textbf{SBP}, and for the analytical score where available.

All details for the experiments are provided in Appendix~\ref{app:exp_config}.

\subsection{BKW benchmark (analytical reference available)}
\label{subsec:bkw}
The BKW solution serves as an analytical benchmark for the spatially homogeneous
Landau equation in the Maxwell case ($\gamma=0$), where the collision kernel reduces to
\[
A(z) = C_0 (|z|^2 I - z \otimes z).
\]
In this setting, both the density $\tilde f_t$ and the score
$\tilde s_t = \nabla_v \log \tilde f_t$ admit closed-form expressions
(see~\cite{YL2025}), enabling direct comparison between numerical predictions
and the exact solution.

Because analytical expressions are available for both the characteristic flow and the score, this benchmark allows us to quantitatively evaluate trajectory accuracy, score consistency, and density reconstruction error in a controlled setting.

\subsubsection{2D BKW solution}\label{subsec:bkw2d}
The two-dimensional BKW solution provides an analytical reference
for both density and score in the Maxwell case ($\gamma=0$).

Figure~\ref{fig:bkw2d_slice} compares one-dimensional density slices
along $(x,0)$ and $(0,y)$ at representative times.
PINN--PM accurately tracks the analytical profile across time
and remains competitive with time-stepping baselines. To further examine the learned characteristic flow, Figure~\ref{fig:bkw2d_transport} visualizes particle trajectories together with reference Euler-integrated evolution.
The close overlap indicates that the global-in-time trajectory network
captures the underlying characteristic transport without explicit time stepping. Figure~\ref{fig:bkw2d_score_scatter} shows score scatter plots
comparing the learned score with the analytical score. The alignment confirms that implicit score matching recovers the correct score structure globally in time.

The features are summarized in  Figure~\ref{fig:bkw2d_phys_fisher}. The relative Fisher divergence confirms that the learned score remains close to the analytical reference throughout the time horizon. At the same time, the kinetic energy evolution demonstrates preservation of macroscopic invariants in the Maxwell case,
and the entropy decay rate follows the analytical dissipation trend.
Together, these results indicate that PINN--PM not only approximates the score accurately, but also preserves the underlying gradient-flow structure of the Landau dynamics. Lastly, Figure~\ref{fig:bkw2d_l2} reports the density $L^2$ error. The error trend is consistent with the stability and density reconstruction bounds derived in Theorem~\ref{thm:traj-error-residual} and Theorem~\ref{thm:kde}.

\begin{figure}[H]
\centering
\includegraphics[width=0.8\linewidth]{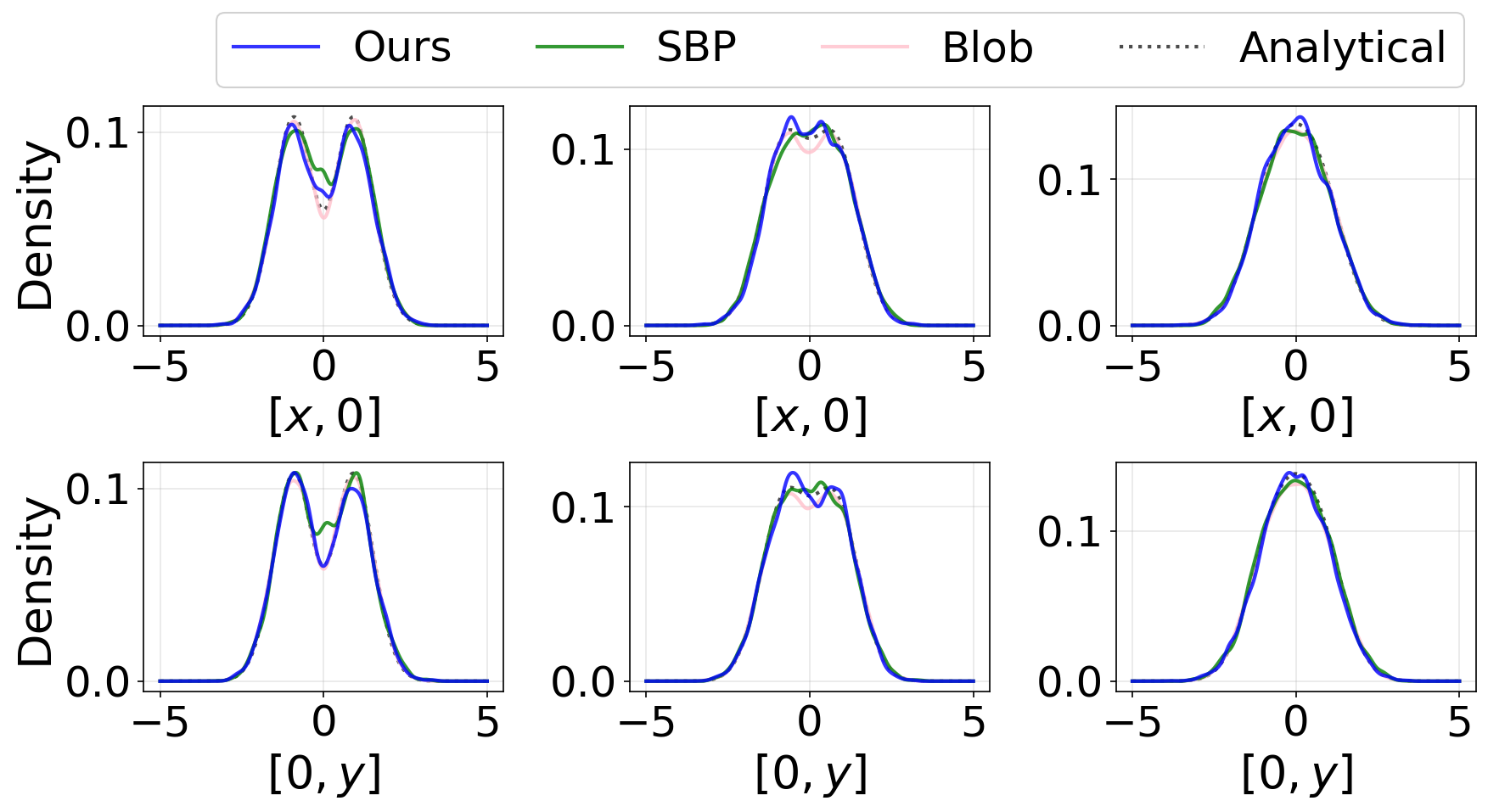}
\caption{
\textbf{BKW-2D: one-dimensional density slices.} Density slices along $(x,0)$ (top) and $(0,y)$ (bottom) at $t\in\{1,2.5,5\}$ Curves compare PINN--PM, SBP, Blob, and the analytical solution. Density is reconstructed via KDE with bandwidth $\varepsilon=0.15$.
}
\label{fig:bkw2d_slice}
\end{figure}

\begin{figure}[H]
\centering
\includegraphics[width=0.9\linewidth]{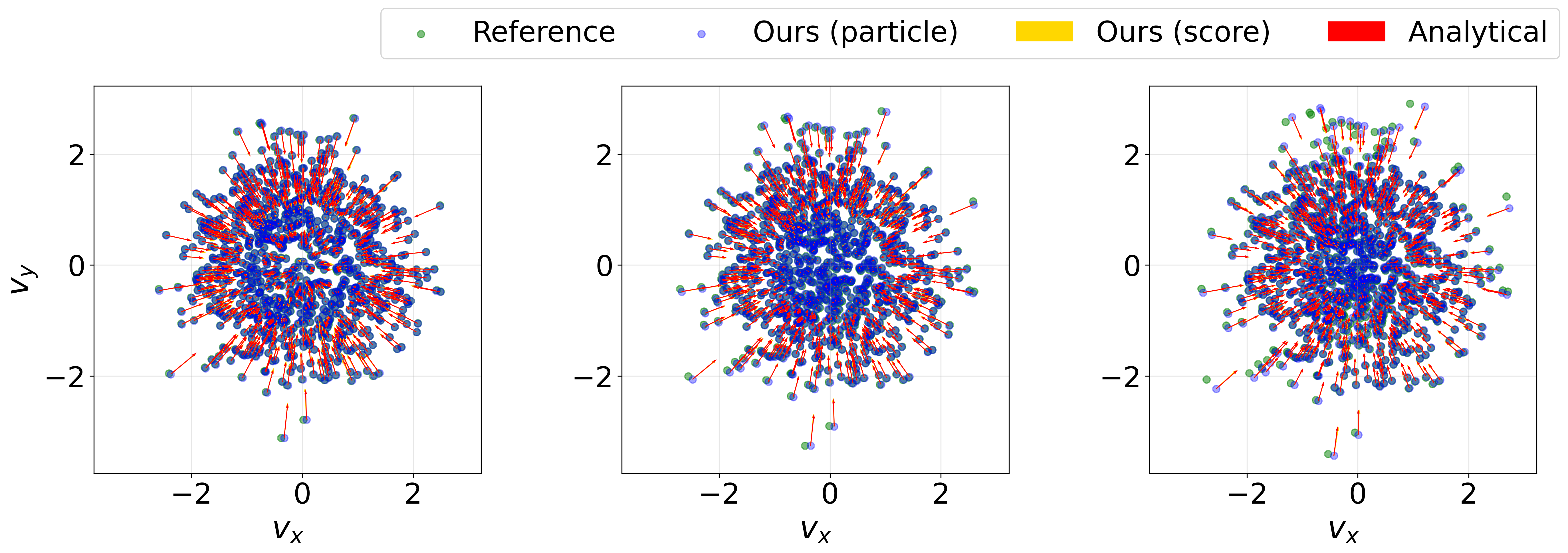}
\caption{\textbf{BKW-2D: particle transport snapshots.}
Particle positions at $t\in\{1,2.5,5\}$.
The plot shows (i) a reference Euler-integrated particle evolution,
(ii) the PINN--PM predicted particle locations, and (iii) score directions (learned vs.\ analytic).
The close overlap indicates that the learned trajectory network captures the characteristic transport.}
\label{fig:bkw2d_transport}
\end{figure}

\begin{figure}[H]
\centering
\includegraphics[width=0.8\linewidth]{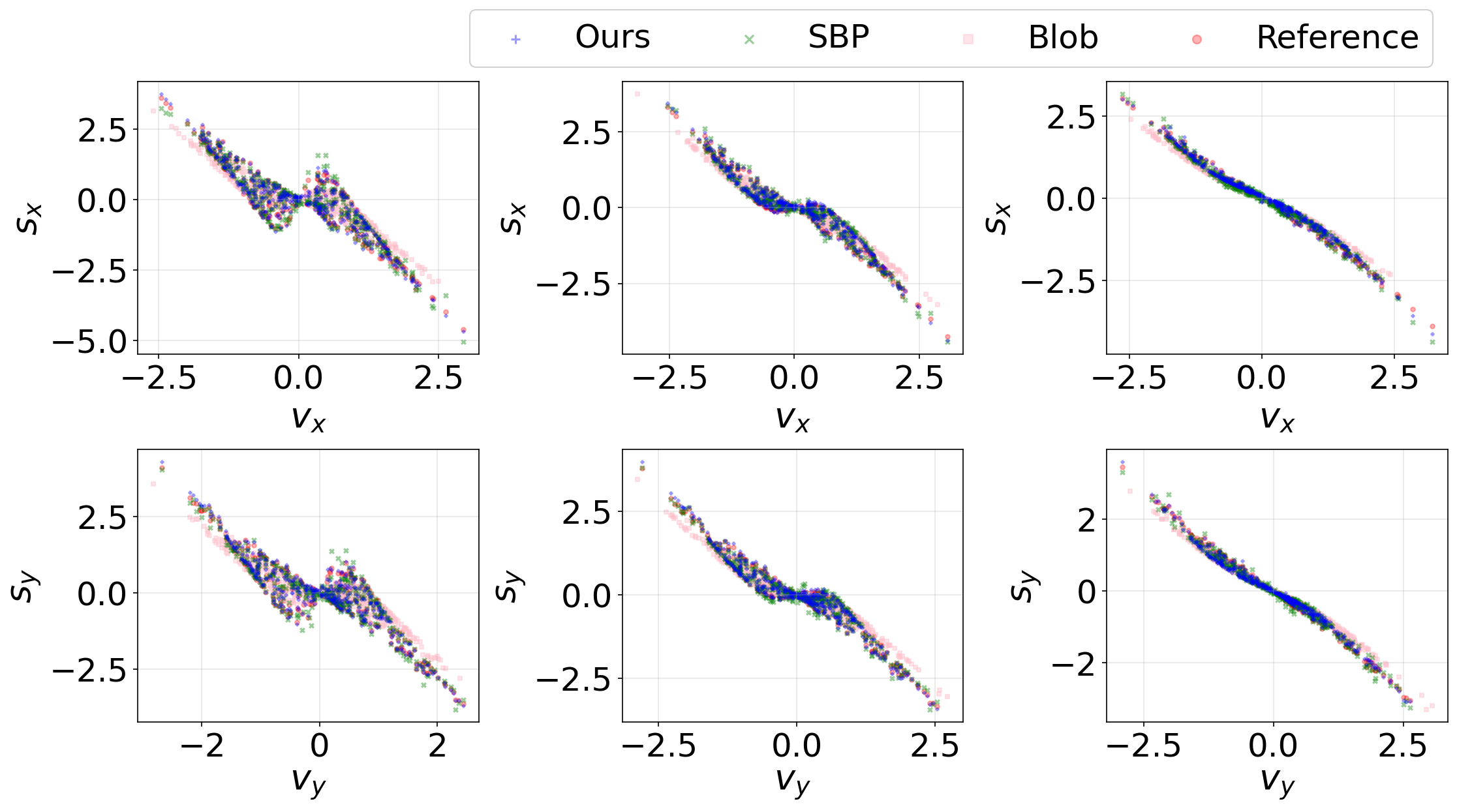}
\caption{\textbf{BKW-2D: score scatter plots.}
Score components $(s_x,s_y)$ versus velocity components $(v_x,v_y)$
at $t\in\{1,2.5,5\}$. Comparison between PINN--PM, SBP, Blob, and the analytical score.
}
\label{fig:bkw2d_score_scatter}
\end{figure}

\begin{figure}[H]
\centering
\includegraphics[width=0.8\linewidth]{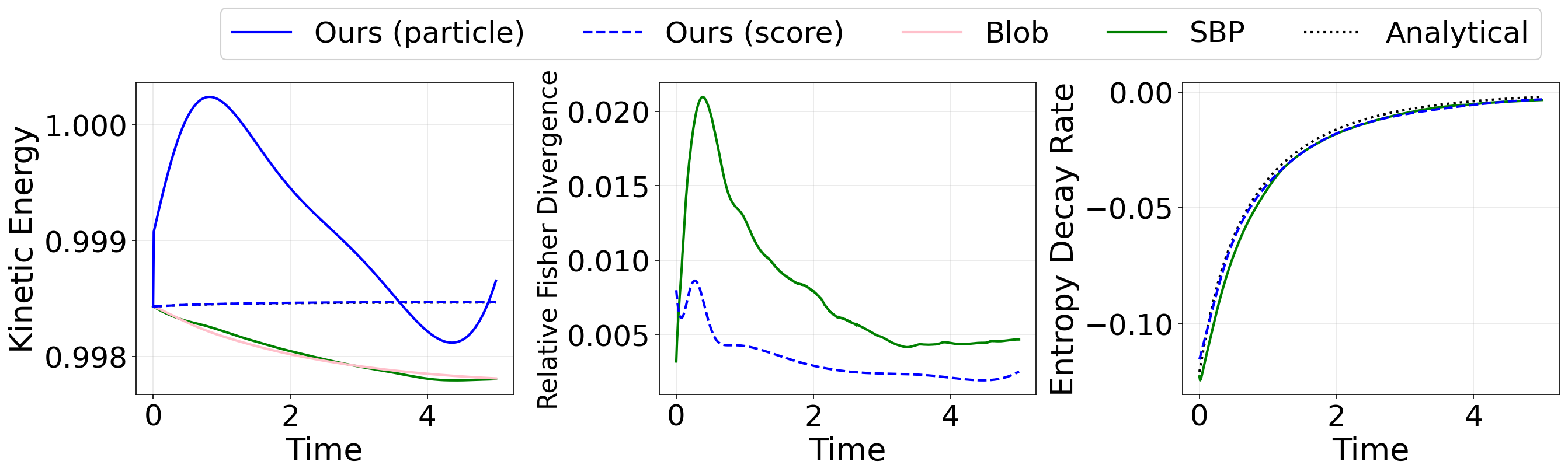}
\caption{
\textbf{BKW-2D: structural diagnostics and score accuracy.}
Left: kinetic energy evolution. Middle: relative Fisher divergence measuring the $L^2_v$ score error. Right: entropy decay rate.
Agreement with analytical curves confirms conservation and
gradient-flow dissipation.
}
\label{fig:bkw2d_phys_fisher}
\end{figure}

\begin{figure}[H]
\centering
\includegraphics[width=0.8\linewidth]{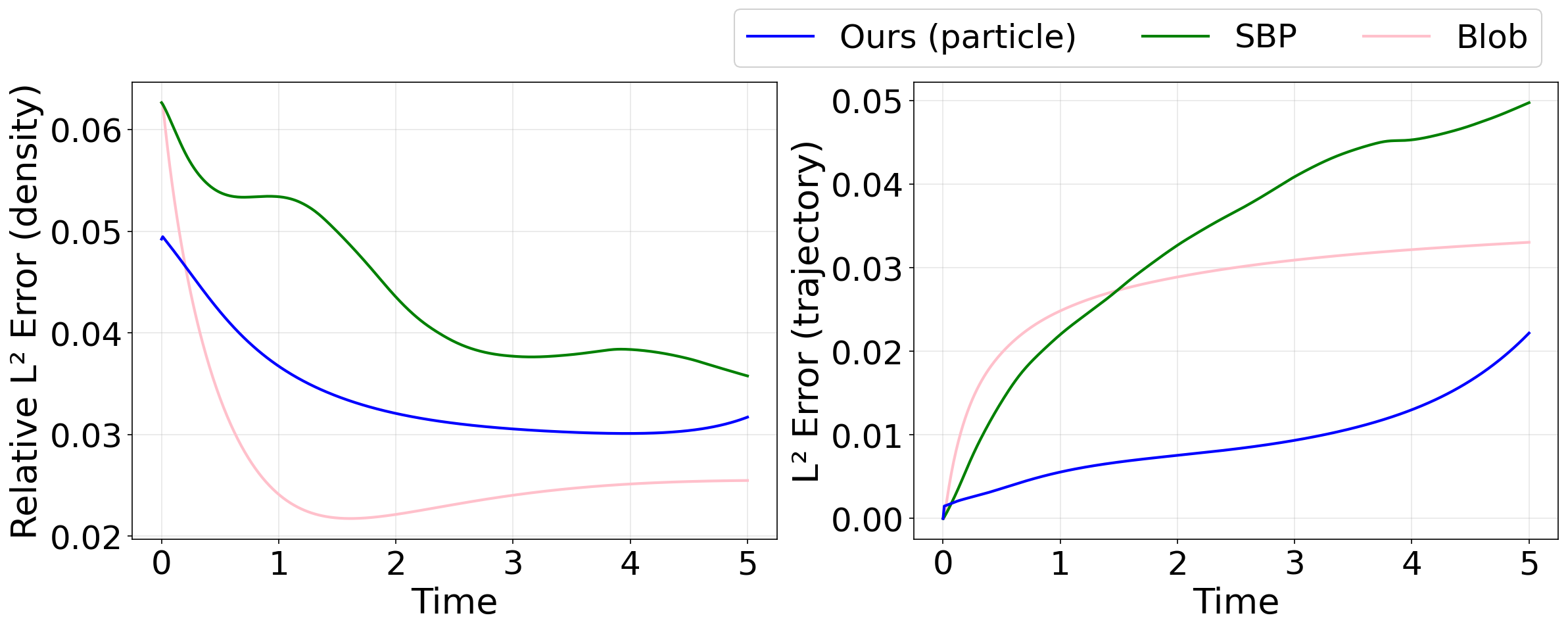}
\caption{
\textbf{BKW-2D: density $L^2$ error.} Relative $L^2$ error of KDE reconstruction over time. Comparison of PINN--PM (particle and score variants), SBP, and Blob. Evaluation grid: $100\times100$ on $[-2.5,2.5]^2$. Bandwidth $\varepsilon=0.15$.
}
\label{fig:bkw2d_l2}
\end{figure}

\subsubsection{3D BKW solution}\label{subsec:bkw3d}
We next consider the three-dimensional BKW solution in the Maxwell case ($\gamma=0$), which provides an analytical reference while significantly increasing the effective dimension. This experiment evaluates the robustness of the global-in-time parameterization in a higher-dimensional setting.

Figure~\ref{fig:bkw3d_slice} compares one-dimensional density slices along
$(x,0,0)$, $(0,y,0)$, and $(0,0,z)$, together with the corresponding marginals.
Across all coordinates, PINN--PM accurately reproduces the analytical density and remains competitive with time-stepping baselines. To examine the learned characteristic transport, Figure~\ref{fig:bkw3d_transport} visualizes particle configurations at representative times in the window $t\in[5.5,5.75,6]$.
The predicted particle locations closely follow the reference Euler-integrated evolution, indicating stable characteristic transport without explicit time stepping. Figure~\ref{fig:bkw3d_score_evol} reports score scatter plots for
$(s_x,s_y,s_z)$ versus $(v_x,v_y,v_z)$. The learned score preserves the analytical score geometry, demonstrating that implicit score matching remains effective in three dimensions.

The quantitative score accuracy is demonstrated in the relative Fisher divergence in Figure~\ref{fig:bkw3d_phys_fisher}. Over the evaluation window $t\in[5.5,5.75, 6]$, the divergence remains uniformly small, indicating that the learned score remains close to the analytical reference. Combined with the stable kinetic energy and consistent entropy decay, this confirms that the learned dynamics preserves both the conservative and dissipative structures of the  Landau equation. Finally, Figure~\ref{fig:bkw3d_l2} reports the density $L^2$ error.
The error remains controlled throughout the simulation window,
consistent with the residual-based stability and density reconstruction analysis developed in Sections~\ref{sec:traj-error-residual} and~\ref{sec:density_corollary}.
\begin{figure}[H]
\centering
\includegraphics[width=0.8\linewidth]{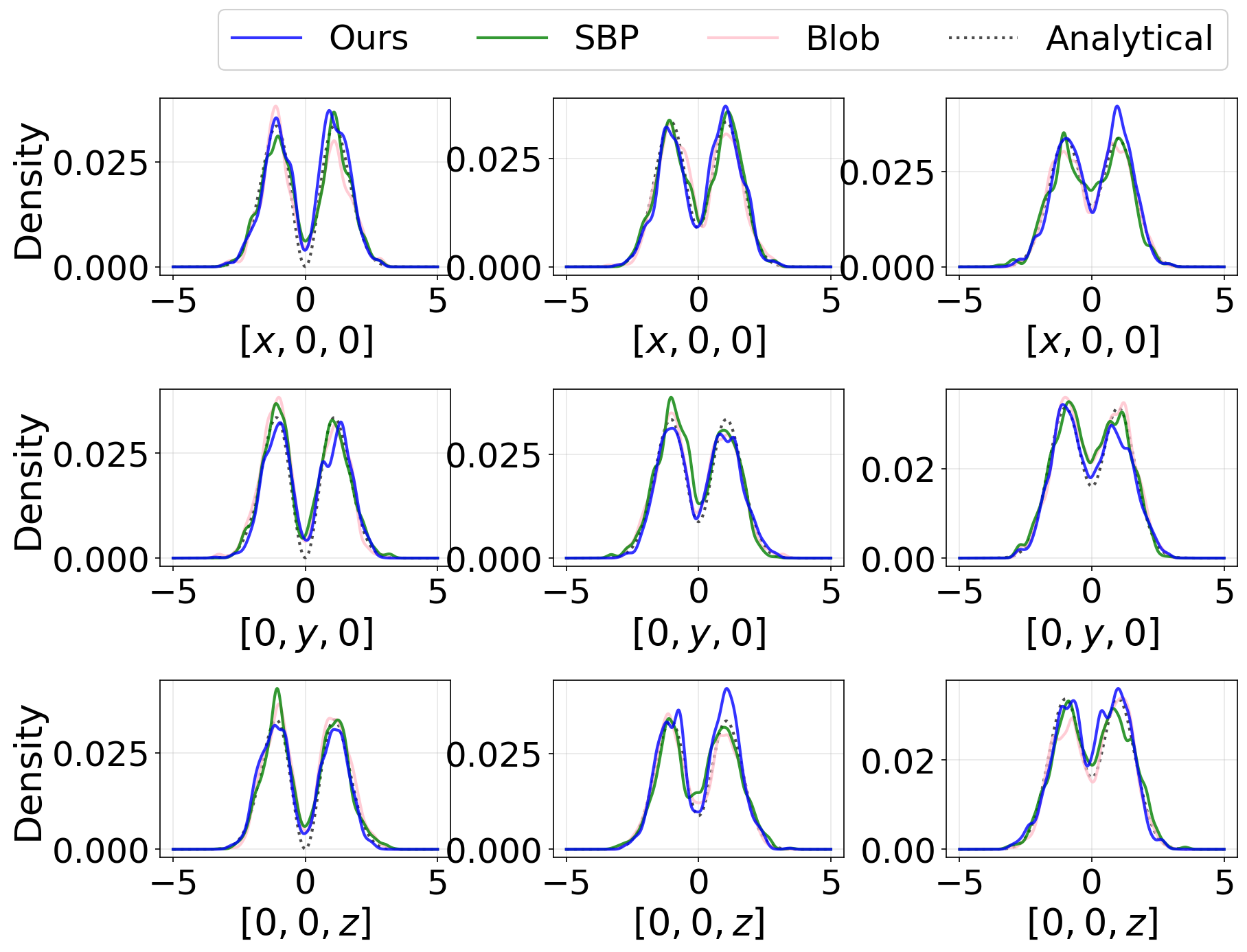}
\caption{
\textbf{BKW-3D: density slices and marginals.} One-dimensional slices along $(x,0,0)$, $(0,y,0)$, $(0,0,z)$ at $t\in\{5.5,5.75,6\}$.
Comparison of PINN--PM, SBP, Blob, and the analytical solution. KDE bandwidth $\varepsilon=0.15$.
}
\label{fig:bkw3d_slice}
\end{figure}

\begin{figure}[H]
\centering
\includegraphics[width=0.8\linewidth]{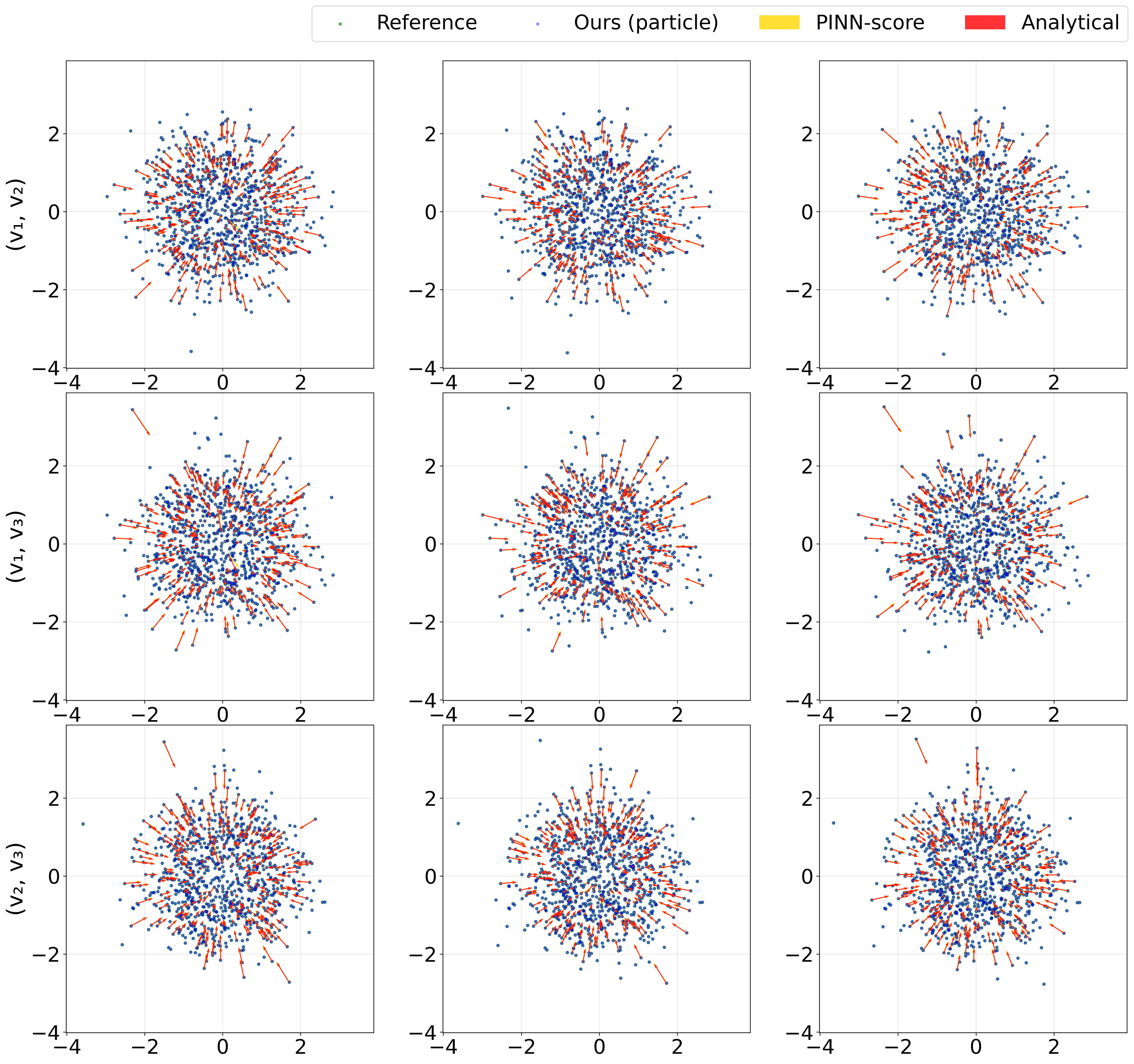}
\caption{\textbf{BKW-3D: particle transport snapshots.}
Two-dimensional projections of particle locations at $t\in\{5.5,5.75,6\}$.
The plots overlay a reference Euler-integrated evolution and the PINN--PM predicted particle locations,
together with score directions (learned vs.\ analytic) to indicate local drift consistency.}
\label{fig:bkw3d_transport}
\end{figure}

\begin{figure}[H]
\centering
\includegraphics[width=0.8\linewidth]{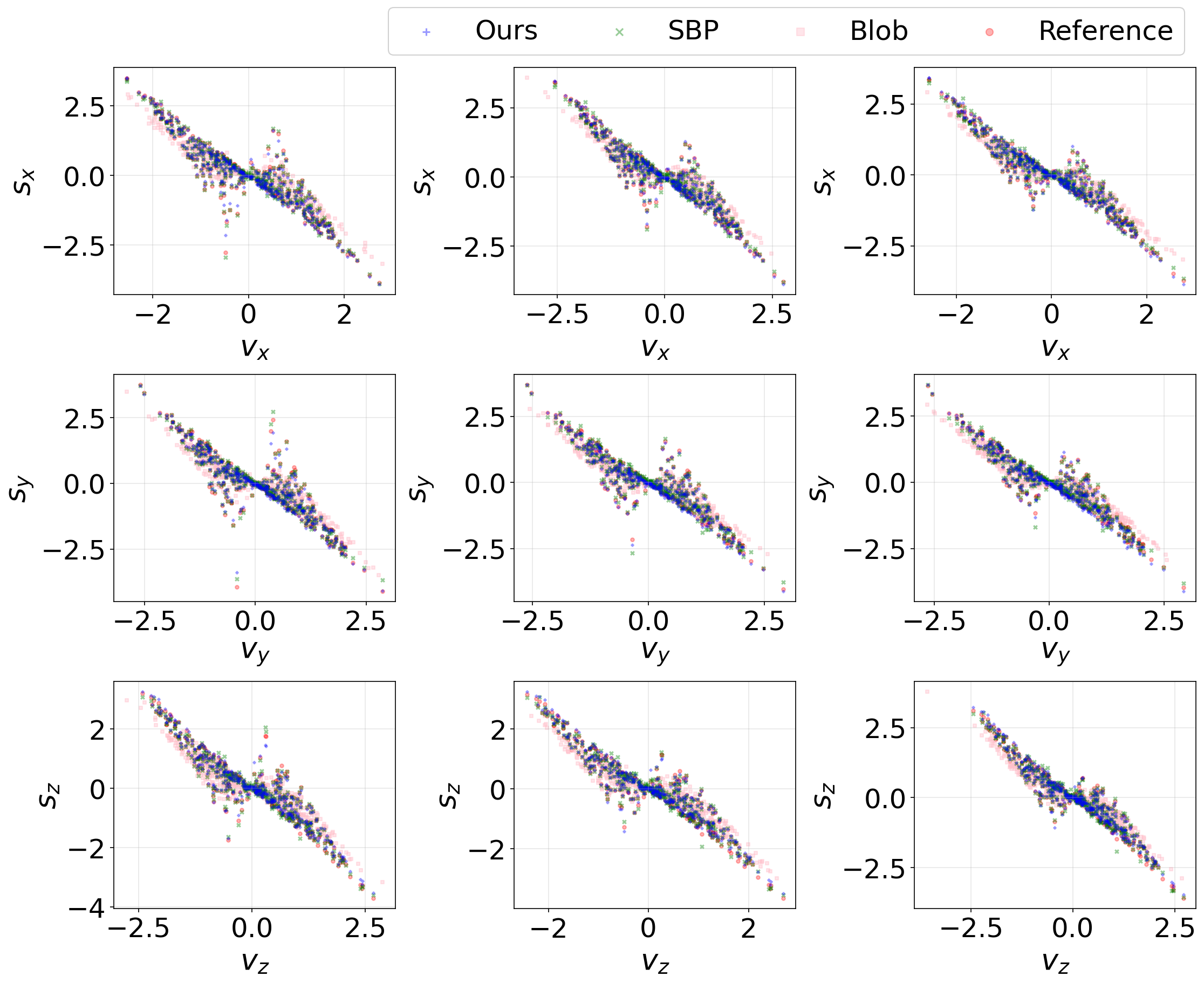}
\caption{\textbf{BKW-3D: score scatter plots.}
Scatter plots of score components$(s_x,s_y,s_z)$ against $(v_x,v_y,v_z)$ at $t\in\{5.5,5.75,6\}$, comparing SBP~\cite{YL2025}, and Blob~\cite{JKF2019} against the analytical score.}
\label{fig:bkw3d_score_evol}
\end{figure}

\begin{figure}[H]
\centering
\includegraphics[width=1\linewidth]{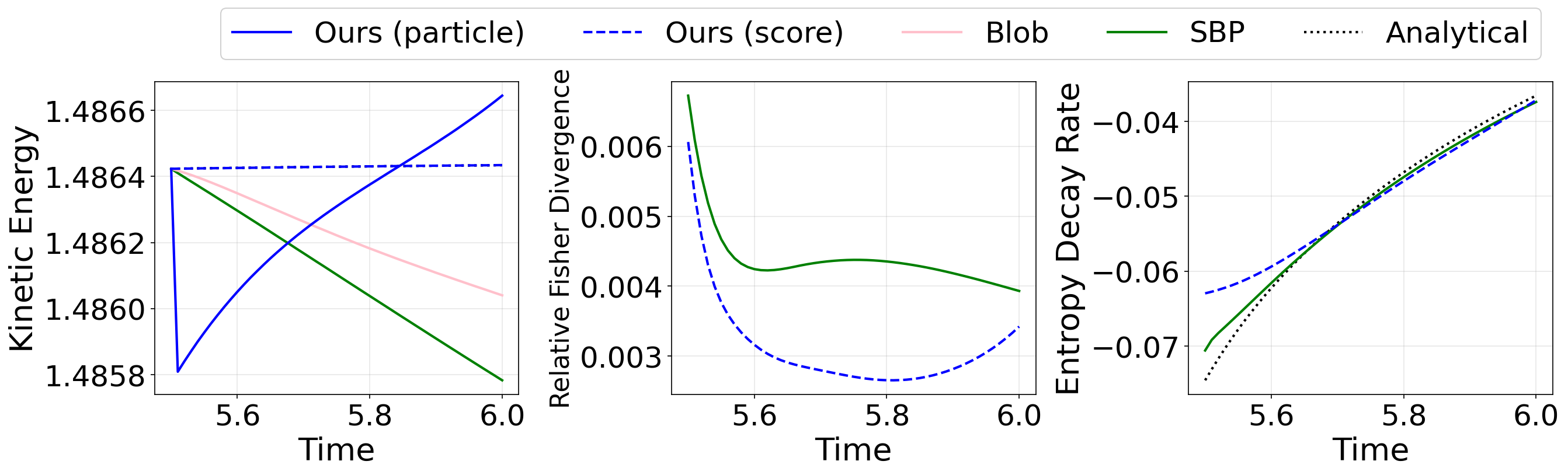}
\caption{
\textbf{BKW-3D: structural diagnostics and score accuracy.}
Left: kinetic energy. Middle: relative Fisher divergence.
Right: entropy decay rate.
}
\label{fig:bkw3d_phys_fisher}
\end{figure}

\begin{figure}[H]
\centering
\includegraphics[width=0.8\linewidth]{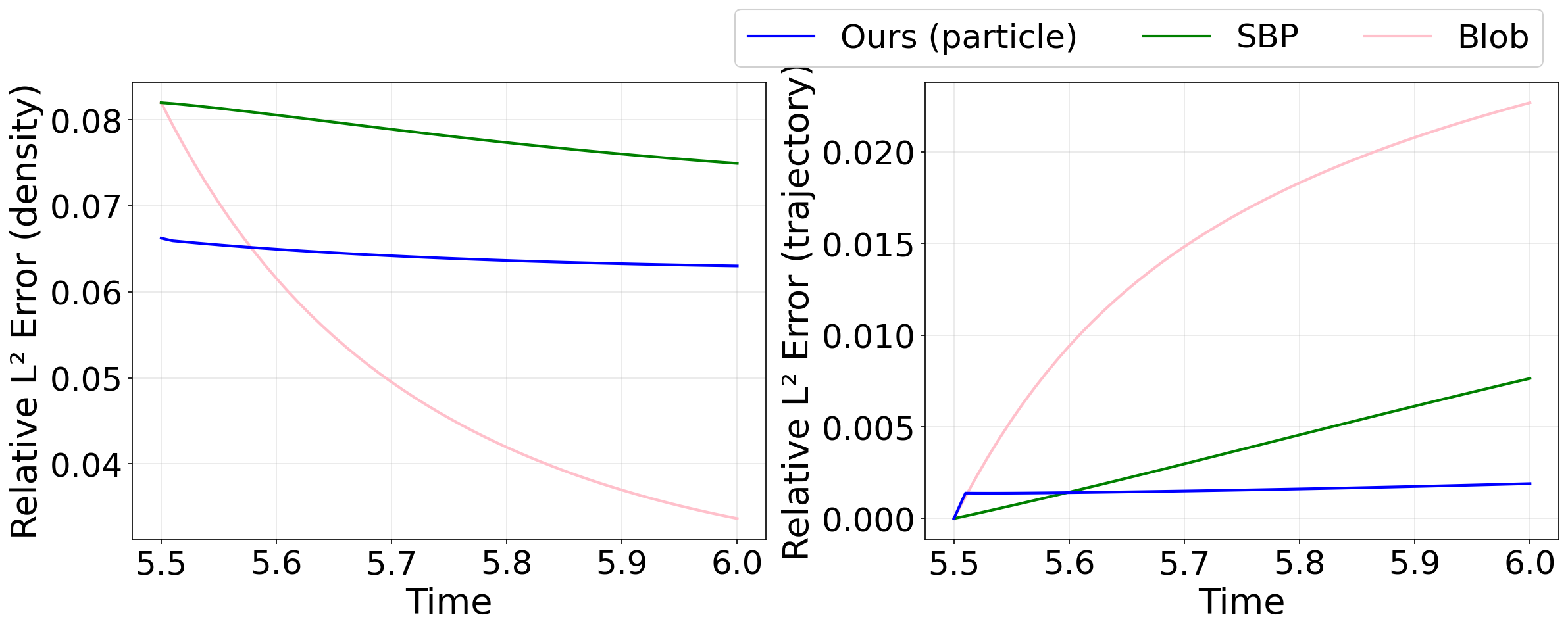}
\caption{\textbf{BKW-3D: density $L^2$ error.}
Relative $L^2$ error of KDE reconstruction . Comparison of PINN--PM (particle and score variants), SBP, and Blob. Density is reconstructed using KDE with bandwidth $\varepsilon=0.15$ and evaluated on a $30\times30\times30$ grid against the analytical solution.
}
\label{fig:bkw3d_l2}
\end{figure}

\subsection{Reference-free benchmarks: Gaussian mixtures, Rosenbluth, anisotropic, and truncated data}\label{subsec:ref_free}

We next consider benchmark problems for which no closed-form analytical solution is available. In contrast to the BKW tests, where quantitative $L^2_v$ and Fisher errors can be computed against an oracle reference, the goal here is to verify that PINN--PM preserves the {structural properties} of the Landau equation.

Recall that the spatially homogeneous Landau equation admits
the gradient-flow structure
\[
\partial_t f_t = \nabla_v \cdot \bigl(f_t \nabla_v \frac{\delta \mathcal{H}}{\delta f_t}\bigr),
\qquad
\mathcal{H}(f) = \int f \log f \, dv,
\]
which implies:

\begin{itemize}
\item conservation of mass and kinetic energy,
\[
\frac{d}{dt}\int f_t\,dv = 0,
\qquad
\frac{d}{dt}\int |v|^2 f_t\,dv = 0,
\]
\item monotone entropy dissipation,
\[
\frac{d}{dt}\mathcal{H}(f_t) \le 0.
\]
\end{itemize}

Accordingly, in the absence of an analytical solution, we assess numerical behavior through: (i) qualitative density evolution,
(ii) score regularity, and (iii) macroscopic invariant preservation.

\subsubsection{Gaussian mixture initial data}
\label{subsec:gm}

We first consider Gaussian mixture initial data of the form
\[
f_0(v) = \sum_{k=1}^{K} w_k \,\mathcal{N}(v;\mu_k,\Sigma_k),
\]
with separated means $\mu_k$. This configuration tests whether the learned transport correctly captures nonlinear interaction between modes without introducing artificial merging or oscillations.

Figure~\ref{fig:gm_slice} shows that the multi-peak structure
is preserved over time.  The learned score fields in Figure~\ref{fig:gm_scatter} remain smooth and aligned with the modal geometry, indicating that the implicit score-matching objective produces a stable global representation even for multi-modal densities. Finally, Figure~\ref{fig:physics_gm} reports the kinetic energy and entropy dissipation.
\[
\mathcal{E}(t)=\frac12 \int |v|^2 f_t(v)\,dv,
\]
which remains stable throughout the simulation, confirming that no artificial numerical drift is introduced despite the absence of explicit time stepping.

\begin{figure}[H]
\centering
\includegraphics[width=0.8\linewidth]{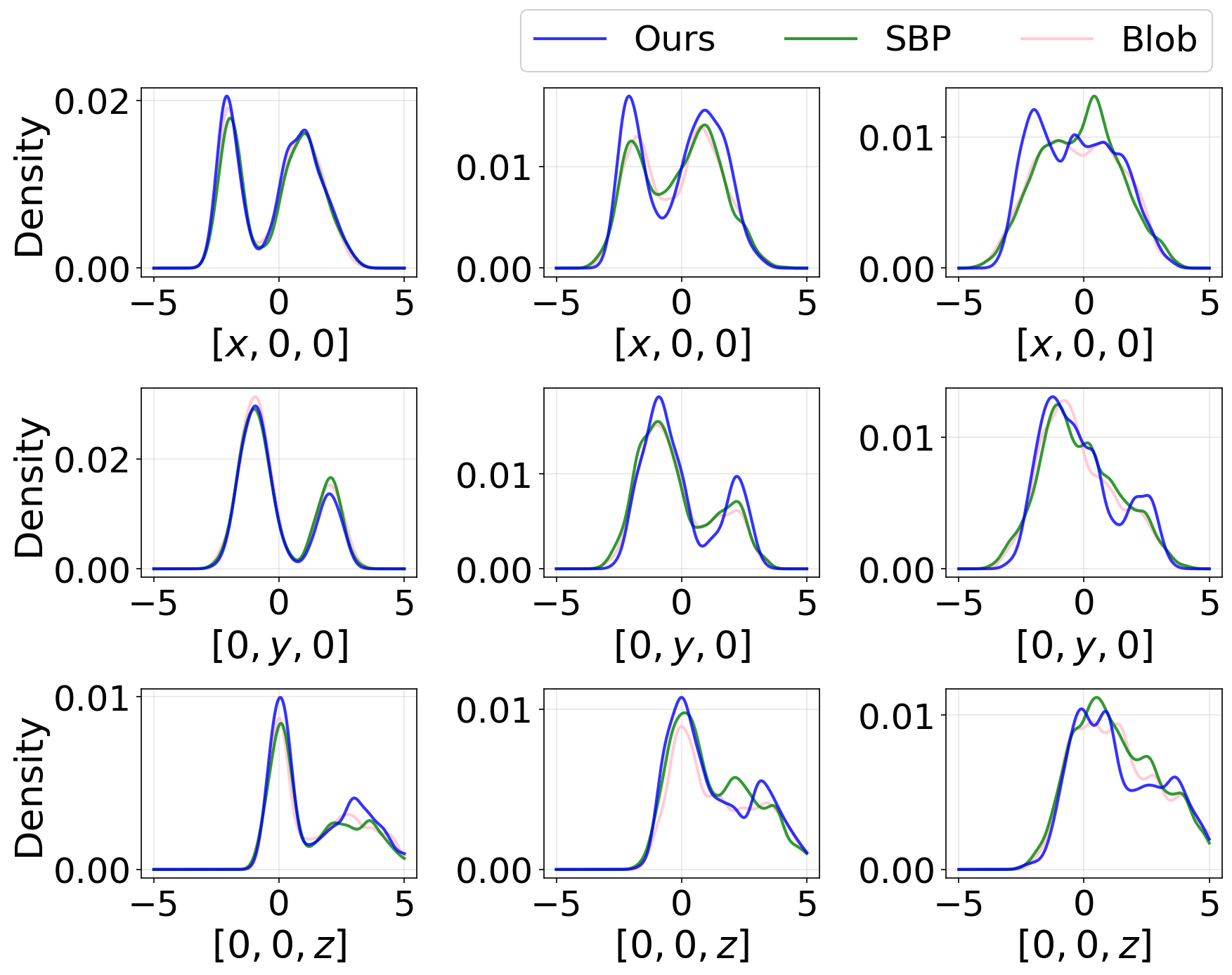}
\caption{\textbf{Gaussian mixture: density slices.}
One-dimensional slices at $t\in\{2.5,15,40\}$.
Comparison of PINN--PM, SBP, and Blob.
Density reconstructed via KDE (bandwidth $\varepsilon=0.15$).
}
\label{fig:gm_slice}
\end{figure}

\begin{figure}[H]
\centering
\includegraphics[width=0.8\linewidth]{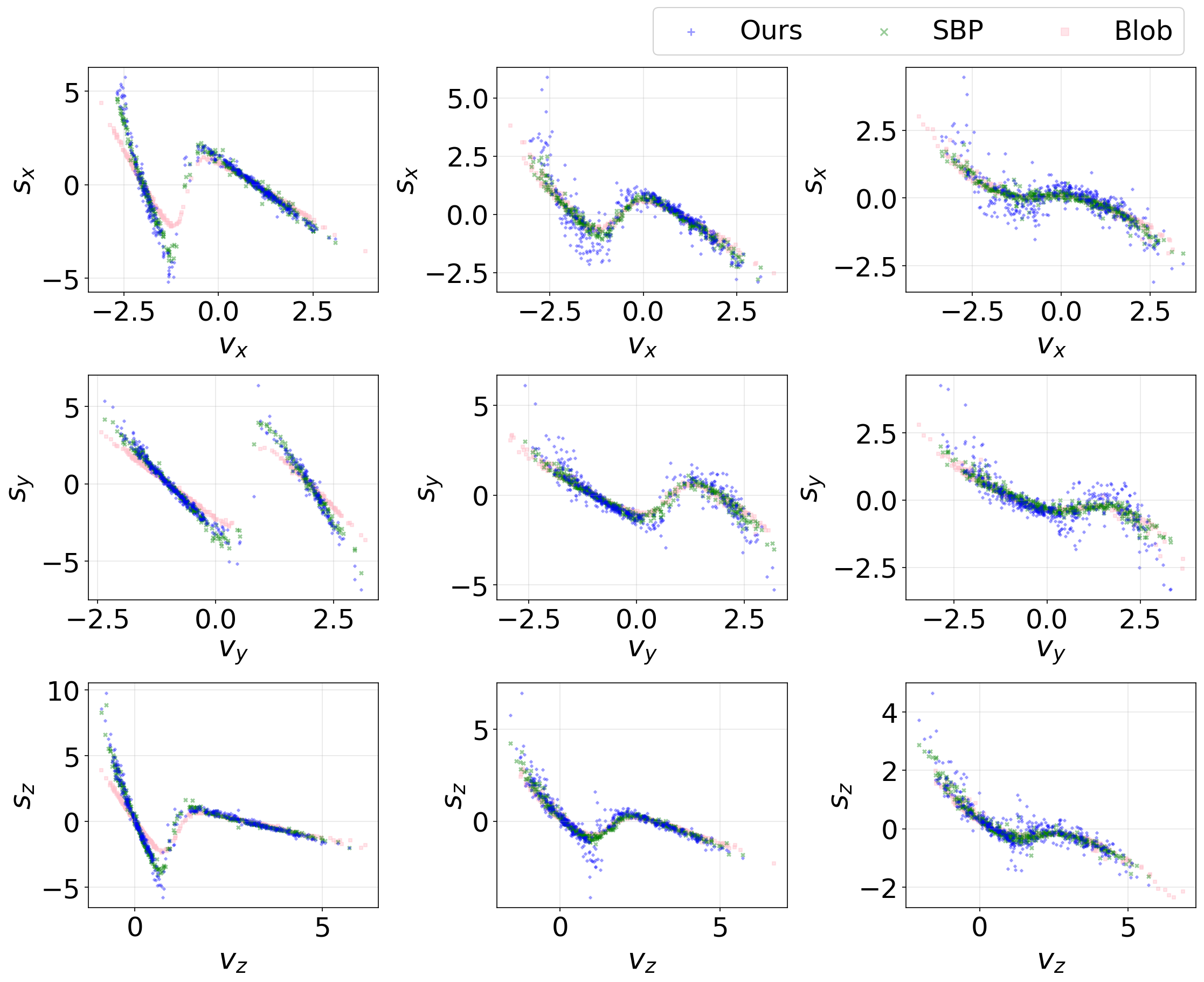}
\caption{
\textbf{Gaussian mixture: score scatter plots.} Score components versus velocity components at $t\in\{2.5,15,40\}$.
Comparison of PINN--PM, SBP and Blob.
}\label{fig:gm_scatter}
\end{figure}

\begin{figure}[H]
\centering
\includegraphics[width=0.8\linewidth]{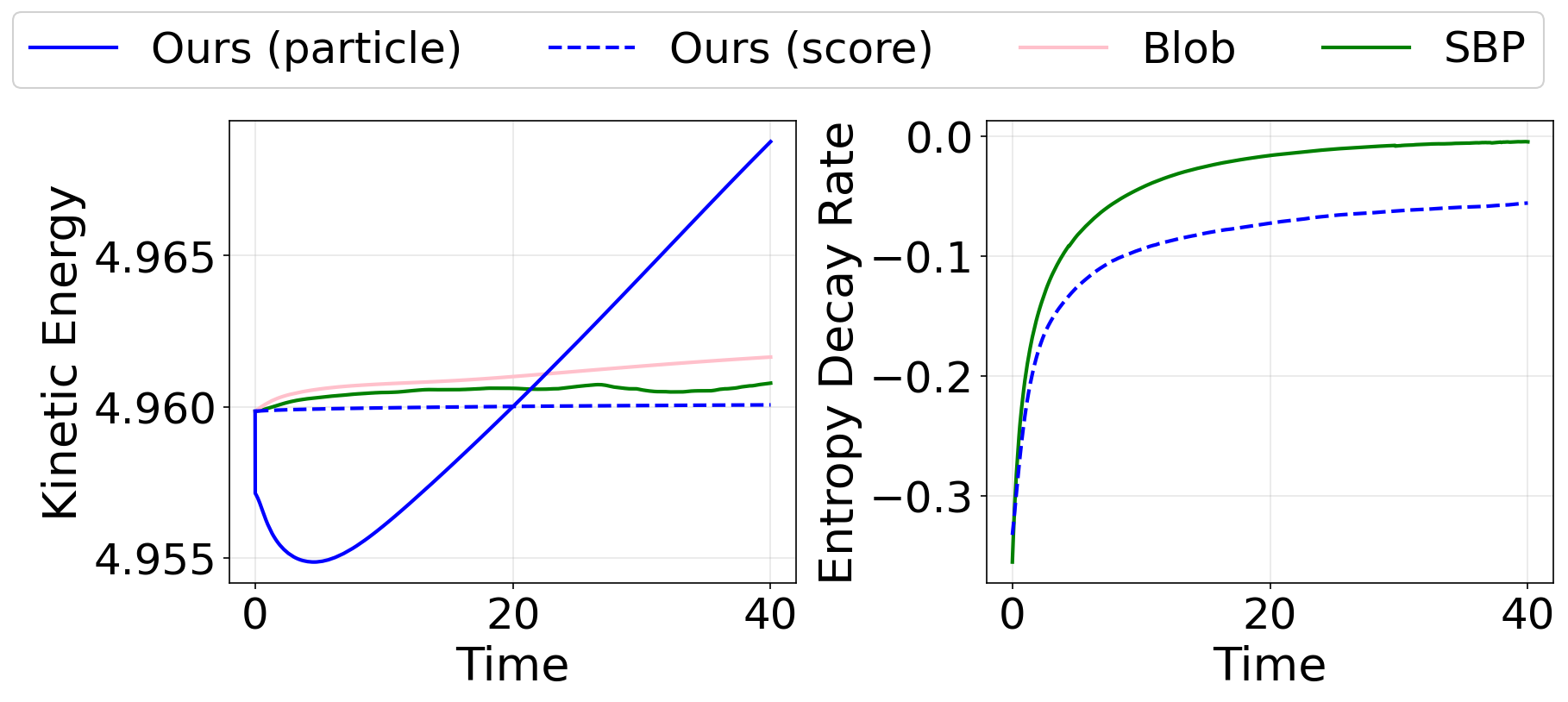}
\caption{
\textbf{Gaussian mixture: structural diagnostics.}
Left: kinetic energy. Right: entropy decay rate. Comparison of PINN--PM, SBP, and Blob.
}
\label{fig:physics_gm}
\end{figure}

\subsubsection{Rosenbluth distribution (Coulomb case)}\label{subsec:rosen}
We initialize particles using the Rosenbluth distribution
used in \cite{YL2025}.
The initial density is defined as
\begin{equation}
f_0(v)
=
\frac{1}{S^2}
\exp\!\left(
-\,S\,\frac{(|v|-\sigma)^2}{\sigma^2}
\right).
\end{equation}
The Rosenbluth distribution provides a challenging test in the Coulomb regime ($\gamma=-3$), where the interaction kernel exhibits near-singular behavior.
This example probes the robustness of the learned score and trajectory residual control under stronger nonlinear interactions.

Figure~\ref{fig:rosen_slice} shows that PINN--PM preserves
the characteristic ring-like geometry. The score scatter plots in Figure~\ref{fig:rosen_score} demonstrate that the learned score remains regular,
without spurious oscillations near high-curvature regions. The kinetic energy diagnostic in Figure~\ref{fig:rosen_kinetic} remains bounded and consistent with conservation laws, indicating that the residual-controlled dynamics remains stable even in the Coulomb setting.

\begin{figure}[H]
\centering
\includegraphics[width=0.8\linewidth]{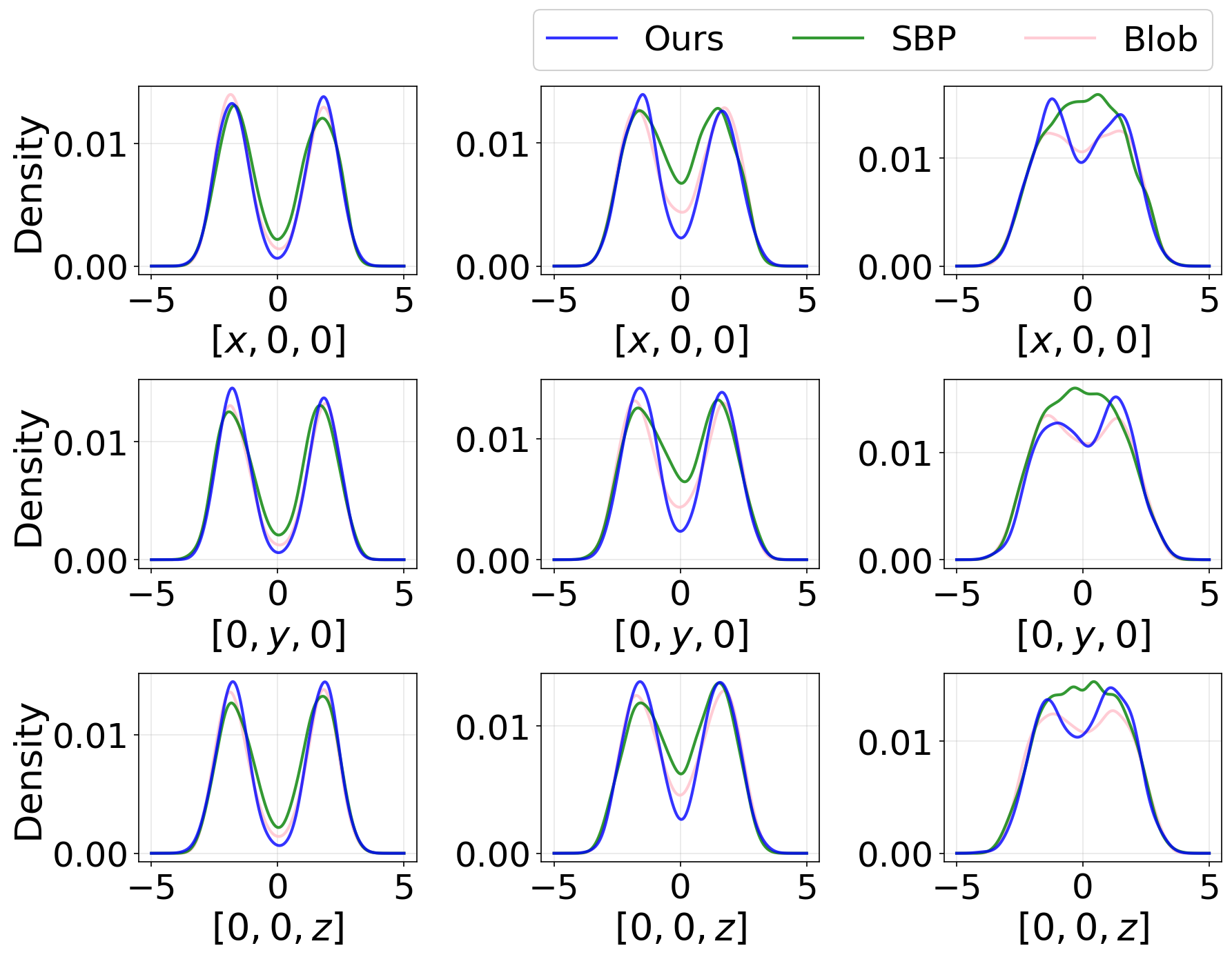}
\caption{
\textbf{Rosenbluth distribution: density slices.}
One-dimensional slices at $t\in\{5,10,20\}$. Comparison of PINN--PM, SBP, and Blob. KDE bandwidth $\varepsilon=0.3$.
}
\label{fig:rosen_slice}
\end{figure}

\begin{figure}[H]
\centering
\includegraphics[width=0.8\linewidth]{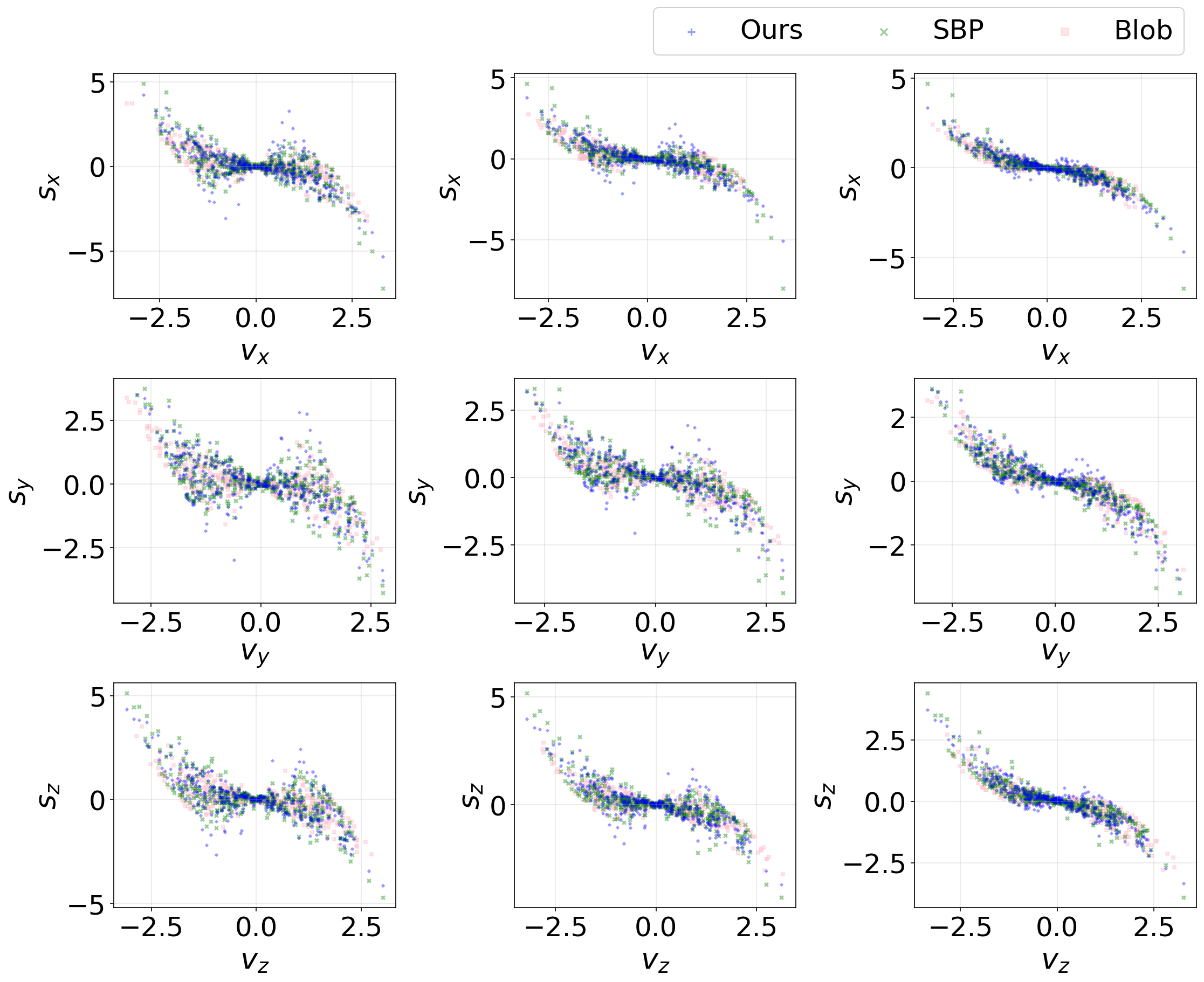}
\caption{
\textbf{Rosenbluth distribution: score scatter plots.} Score components versus velocity components at $t\in\{5,10,20\}$. Comparison of PINN--PM, SBP and Blob.
}
\label{fig:rosen_score}
\end{figure}

\begin{figure}[H]
\centering
\includegraphics[width=0.8\linewidth]{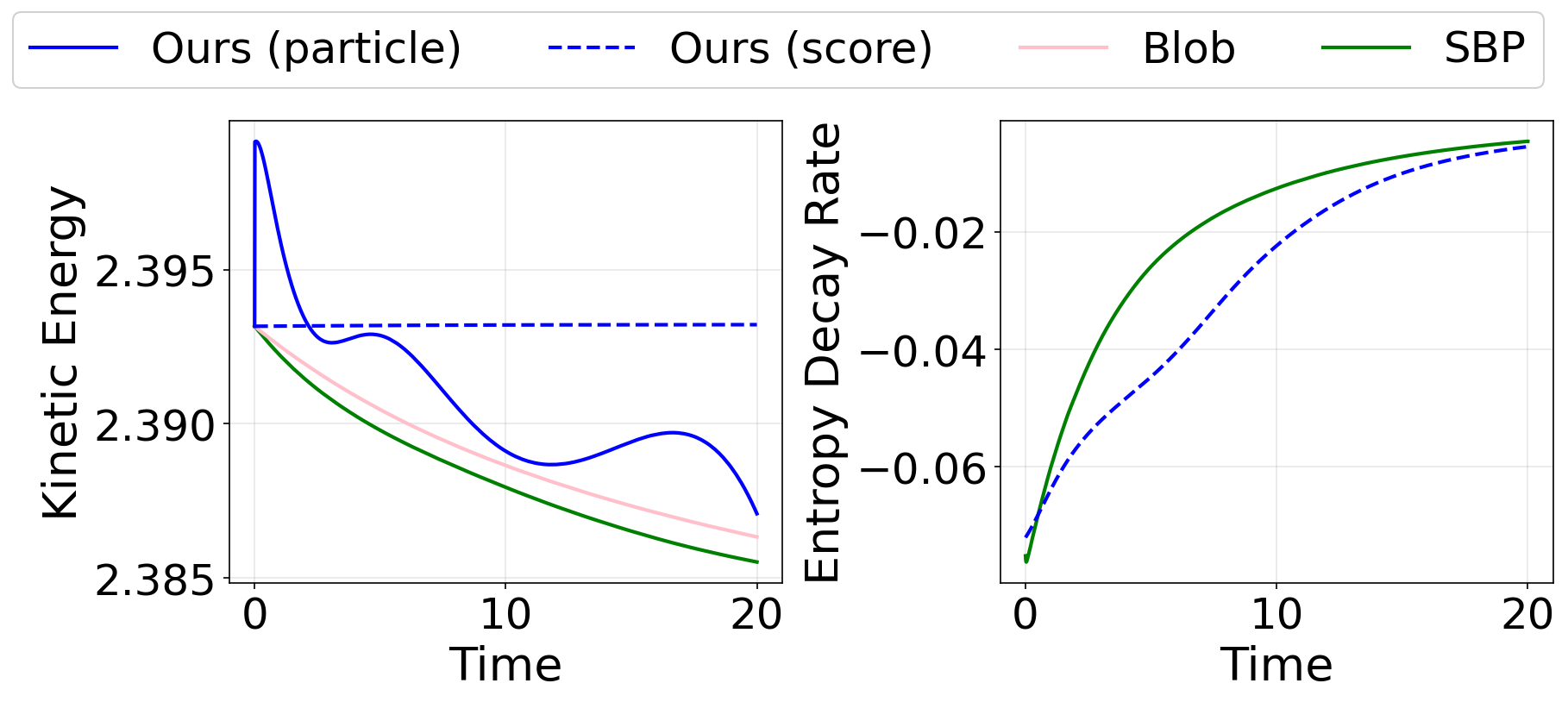}
\caption{
\textbf{Rosenbluth distribution: structural diagnostics.}
Left: kinetic energy. Right: entropy decay rate. Comparison of PINN--PM, SBP, and Blob.
}\label{fig:rosen_kinetic}
\end{figure}

\subsubsection{Anisotropic initial data}\label{subsec:anisotropic}
We finally consider anisotropic initial data, where the initial covariance matrix is non-isotropic:
\[
f_0(v) = \mathcal{N}(v;0,\Sigma_0), \qquad \Sigma_0 \neq \lambda I.
\]
Under Landau dynamics, the solution relaxes toward an
isotropic equilibrium. Figure~\ref{fig:ani_slice} shows progressive isotropization of the density. 

The score evolution in Figure~\ref{fig:ani_score}
illustrates increasing symmetry in velocity space,
consistent with entropy-driven relaxation.

The kinetic energy in Figure~\ref{fig:ani_kinetic} remains stable,
and the entropy decay rate (right panel) exhibits monotone dissipation, consistent with the gradient-flow structure of the Landau equation.
\begin{figure}[H]
\centering
\includegraphics[width=0.8\linewidth]{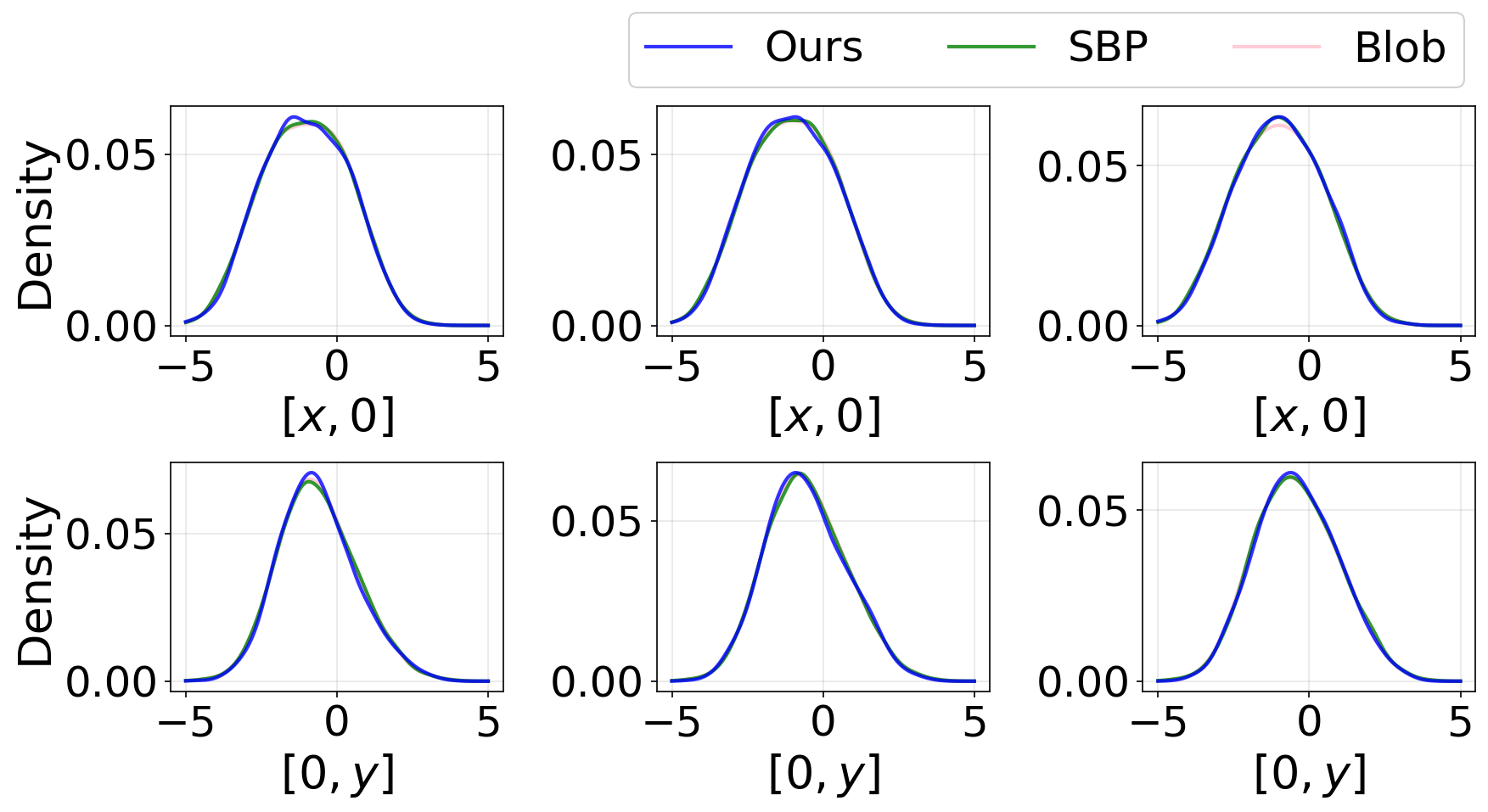} 
\caption{\textbf{Anisotropic initial data: score scatter plots.}
Scatter plots of score components versus velocity components,
showing increasing alignment over time at $t=\{10,20,40\}.$ KDE bandwidth $\varepsilon=0.3$.}
\label{fig:ani_slice}
\end{figure}

\begin{figure}[H]
\centering
\includegraphics[width=0.8\linewidth]{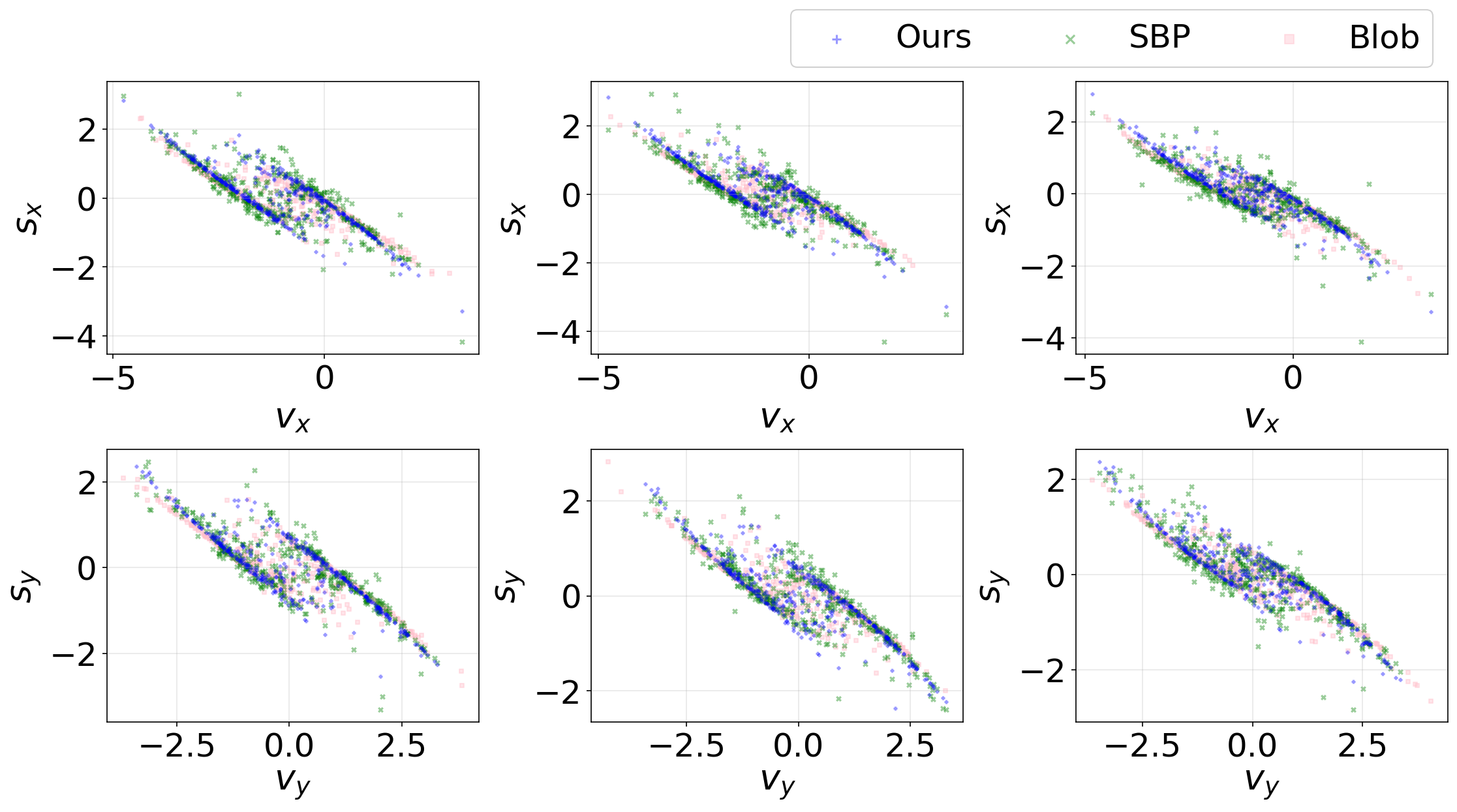} 
\caption{\textbf{Anisotropic initial data: score scatter plots.}
Scatter plots of score components versus velocity components,
showing increasing alignment over time at $t=\{10,20,40\}.$}
\label{fig:ani_score}
\end{figure}

\begin{figure}[H]
\centering
\includegraphics[width=0.8\linewidth]{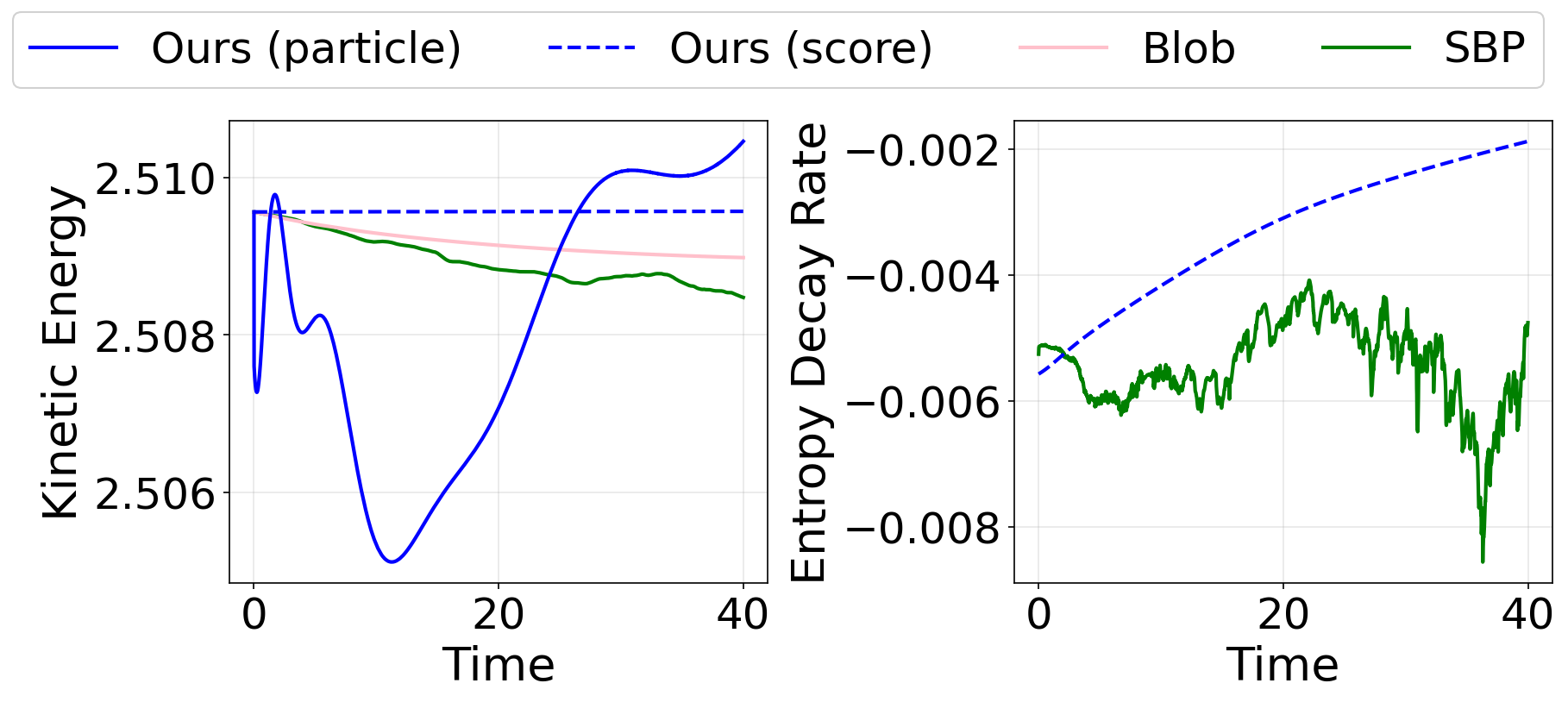} 
\caption{
\textbf{Anisotropic initial data: structural diagnostics.}
Left: kinetic energy. Right: entropy decay rate. Monotone dissipation indicates gradient-flow relaxation toward isotropy.
}
\label{fig:ani_kinetic}
\end{figure}

\subsubsection{Truncated distribution (2D)}
\label{subsec:truncated}

We finally consider a truncated initial distribution in two dimensions,
where the support of the density is compact and exhibits sharp gradients near the boundary.
This configuration probes the robustness of the learned score and transport map
under reduced smoothness and boundary-sensitive geometry.

Unlike Gaussian-type initial data, the truncated distribution
contains non-smooth features at the support boundary.
Such configurations are particularly challenging for
kernel-based particle methods due to smoothing artifacts.

Figure~\ref{fig:tr_slice} shows density slices at representative times.
PINN--PM maintains stability and avoids spurious oscillations near the truncated region.

Figure~\ref{fig:tr_score} illustrates the score scatter plots.
The learned score remains regular and aligned with the transport direction,
despite the reduced regularity of the initial data.

Finally, Figure~\ref{fig:tr_kinetic} reports the kinetic energy and entropy decay rate.
The evolution remains stable and consistent with the
dissipative structure of the Landau equation.

\begin{figure}[H]
\centering
\includegraphics[width=0.8\linewidth]{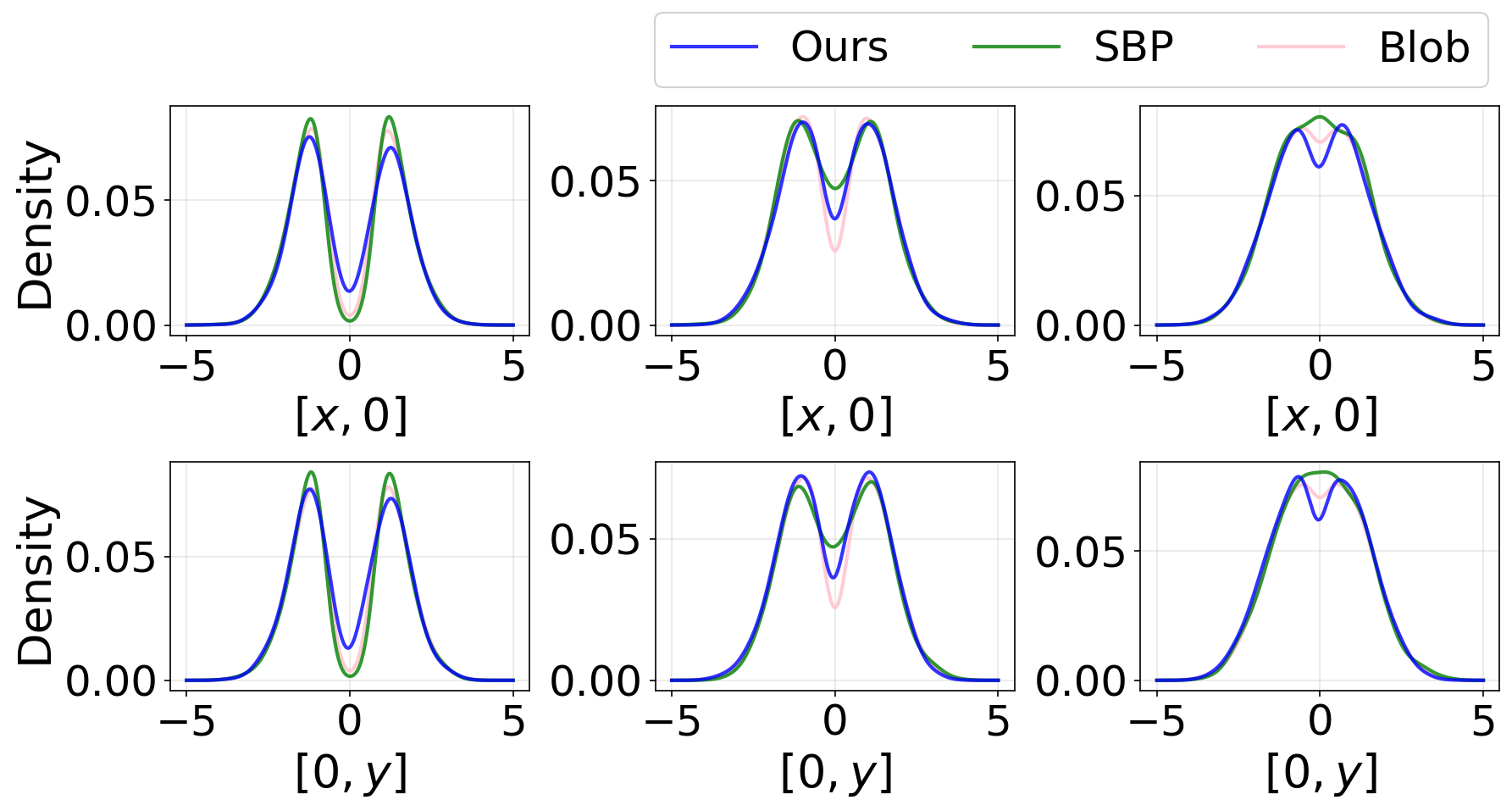}
\caption{
\textbf{Truncated 2D: density slices.}
One-dimensional slices at time at $t=\{1,2.5,5\}$.
Comparison of PINN--PM, SBP, and Blob.
KDE bandwidth $\varepsilon=0.3$.
}
\label{fig:tr_slice}
\end{figure}

\begin{figure}[H]
\centering
\includegraphics[width=0.8\linewidth]{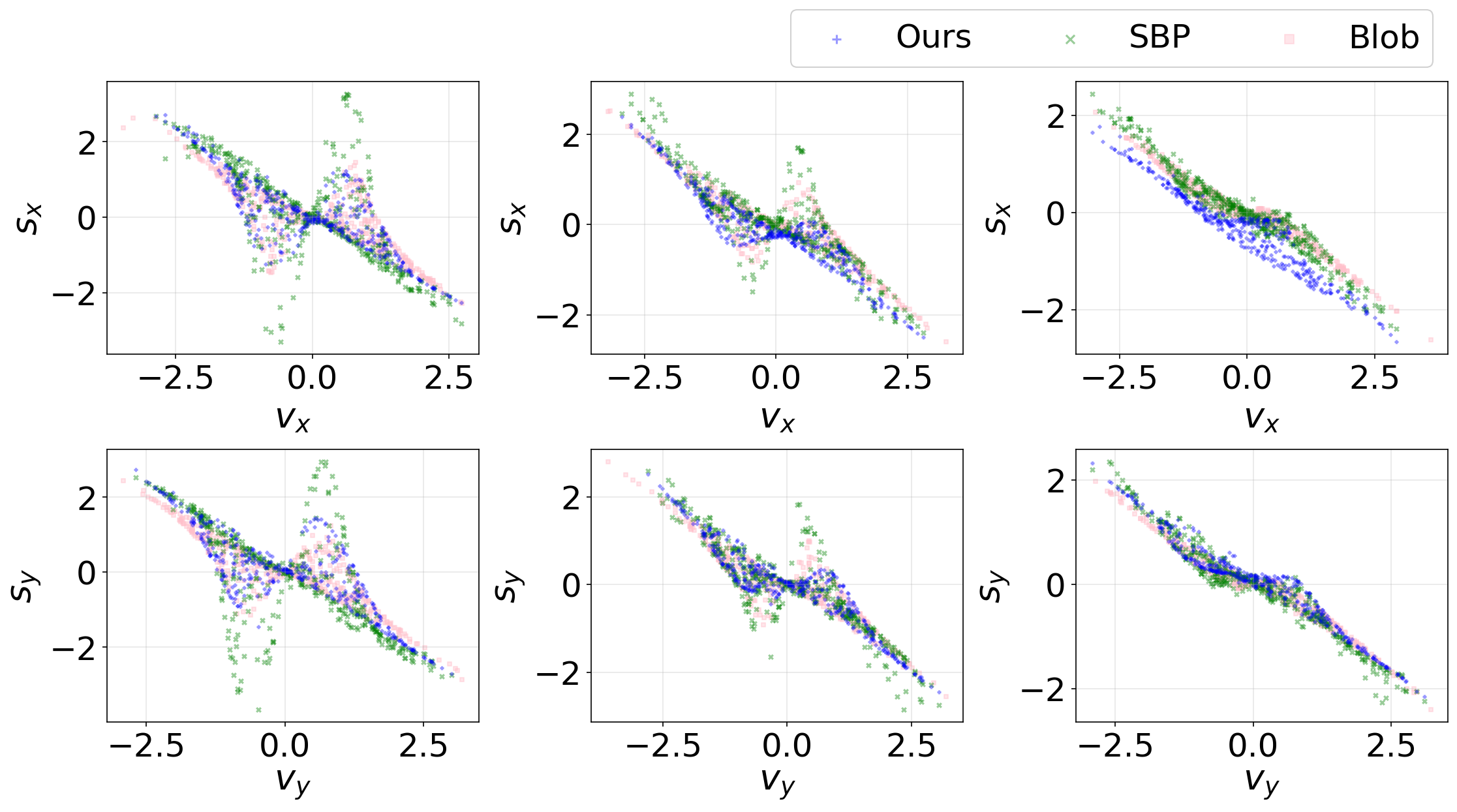}
\caption{
\textbf{Truncated 2D: score scatter plots.}
Score components versus velocity components over time at $t=\{1,2.5,5\}$.
}
\label{fig:tr_score}
\end{figure}

\begin{figure}[H]
\centering
\includegraphics[width=0.8\linewidth]{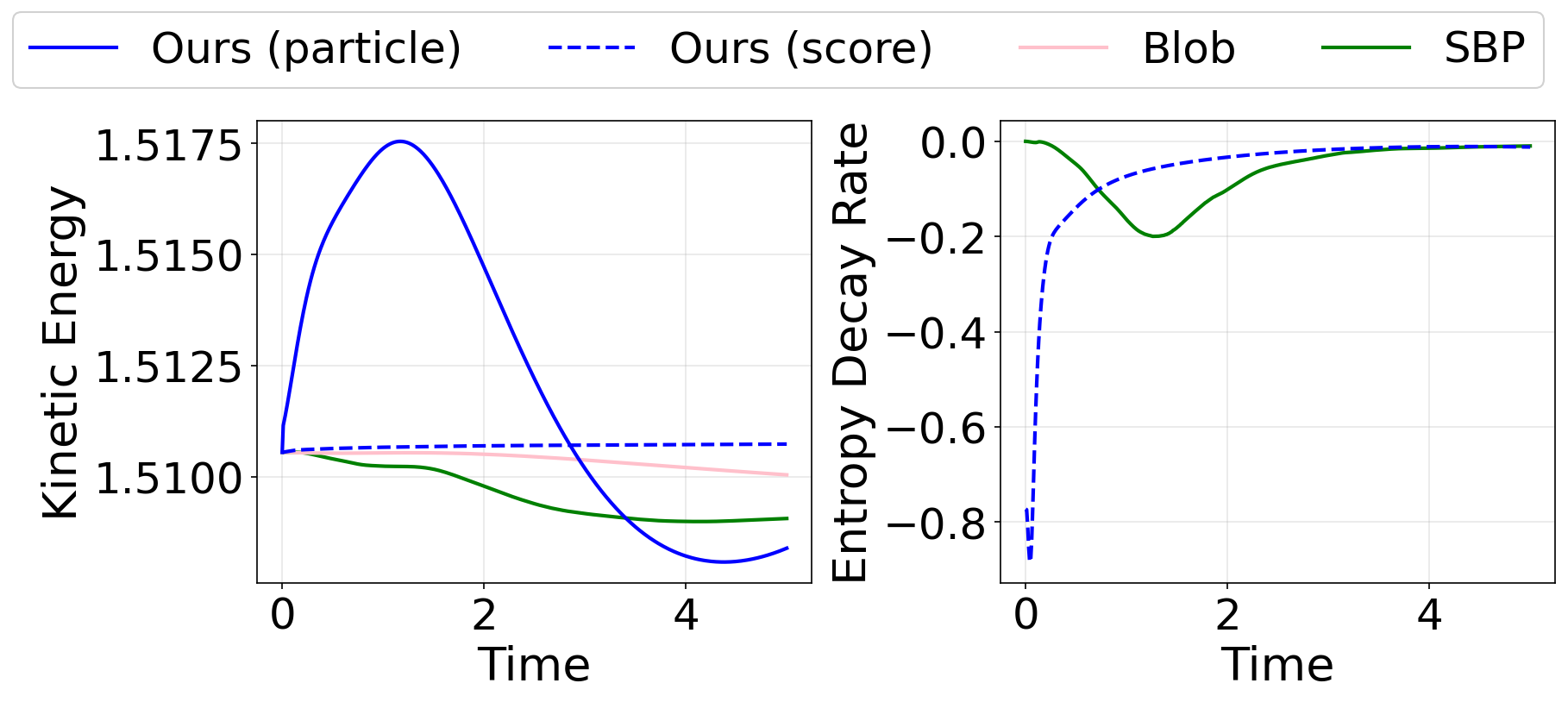}
\caption{
\textbf{Truncated 2D: structural diagnostics.}
Left: kinetic energy. Right: entropy decay rate.
}
\label{fig:tr_kinetic}
\end{figure}

\section{Discussion}\label{subsec:discussion}
We introduced PINN--PM, a global-in-time neural particle method for the spatially homogeneous Landau equation. By jointly parameterizing the score and the characteristic flow and enforcing the Landau dynamics through a continuous-time residual, the method removes explicit time stepping and yields a mesh-free neural particle simulator.

Our analysis establishes an end-to-end accuracy certificate: 
training-time quantities—the implicit score-matching excess risk and the physics residual—control deployment-time errors in trajectory and density. In particular, we proved that the Wasserstein discrepancy between the learned empirical measure and the true Landau solution is bounded by the accumulated score error, residual energy, and the intrinsic Monte Carlo rate. This stability result further propagates to density reconstruction via kernel density estimation, yielding an explicit bias–variance–trajectory decomposition.

Numerical experiments on analytical (BKW) and reference-free benchmarks confirm the theoretical structure: score accuracy, trajectory stability, and macroscopic invariant preservation are consistently aligned with the derived bounds. Empirically, PINN--PM achieves competitive or improved accuracy with significantly fewer particles than time-stepping score-based particle methods.

Overall, the proposed framework demonstrates that physics-informed global-in-time learning can transform interacting particle solvers into amortized neural simulators, while retaining quantitative error guarantees rooted in kinetic theory.

\section*{Reproducibility}

We provide detailed information to facilitate the reproducibility of all numerical
experiments presented in this paper. In addition to the descriptions below,
the complete source code used to generate all results is publicly available at

\begin{center}
\url{https://github.com/tomatofromsky/kinetic-score-landau-pinn}.
\end{center}

\section*{CRediT authorship contribution statement}

Minseok Kim: Software, Investigation, Validation, Visualization,
Writing -- original draft, Writing -- review \& editing. Sung-Jun Son: Software, Investigation, Validation, Writing -- review \& editing. Yeoneung Kim: Conceptualization, Methodology, Formal analysis, Writing -- original draft, Writing -- review \& editing, Supervision. Donghyun Lee: Conceptualization, Methodology,
Writing -- review \& editing, Supervision.

\section*{Declaration of competing interest}

The authors declare that they have no known competing financial interests
or personal relationships that could have appeared to influence the work
reported in this paper.

\section*{Data availability}

The source code used to reproduce the numerical experiments in this work
is publicly available at:

\begin{center}
\url{https://github.com/tomatofromsky/kinetic-score-landau-pinn}.
\end{center}

All datasets used in this study are synthetically generated from the
prescribed initial distributions described in the paper.

\section*{Acknowledgements}

Yeoneung Kim and Minseok Kim are supported by the National Research Foundation of Korea
(NRF) grant funded by the Korea government (MSIT) (RS-2023-00219980, RS-2023-00211503).  Donghyun Lee and Sung-Jun Son are supported by the National Research Foundation of Korea(NRF) grant funded by the Korea government(MSIT)(No.RS-2023-00212304 and No.RS-2023-00219980). Sung-Jun Son is supported by the National Research Foundation of Korea(NRF) grant funded by the Korea government(MSIT and MOE)(No.RS-2023-00210484 and No.RS-2025-25419038).

\bibliographystyle{plain}
\bibliography{ref.bib}

\appendix
\section{Proof of Theorem~\ref{thm:kde}}\label{app:density_proof}
We expand the argument in two steps.

Fix $t\in[0,T]$ and $\varepsilon>0$.  Introduce the {oracle KDE} built from the
exact characteristics $\{v_t^{(i)}\}_{i=1}^N$:
\[
\tilde f_{t,N,\varepsilon}(v)
:=
\frac1N\sum_{i=1}^N \varepsilon^{-d}K\!\Big(\frac{v-v_t^{(i)}}{\varepsilon}\Big)
=
\frac1N\sum_{i=1}^N K_\varepsilon(v-v_t^{(i)}),
\qquad
K_\varepsilon(x):=\varepsilon^{-d}K(x/\varepsilon).
\]
Then, by the triangle inequality and $(a+b+c)^2\le 3(a^2+b^2+c^2)$,
\begin{equation}\label{eq:decomp_three_terms}
\|f^{\mathrm{approx}}_{t,N,\varepsilon}-\tilde f_t\|_{L^2_v}^2
\;\le\;
3\|f^{\mathrm{approx}}_{t,N,\varepsilon}-\tilde f_{t,N,\varepsilon}\|_{L^2_v}^2
+3\|\tilde f_{t,N,\varepsilon}-\E[\tilde f_{t,N,\varepsilon}]\|_{L^2_v}^2
+3\|\E[\tilde f_{t,N,\varepsilon}]-\tilde f_t\|_{L^2_v}^2 .
\end{equation}
We bound each term on the right-hand side.

Since $v_t^{(i)}\stackrel{\mathrm{i.i.d.}}{\sim}\tilde f_t$, we have
$\E[\tilde f_{t,N,\varepsilon}]=K_\varepsilon * \tilde f_t$.
Assume $\tilde f_t\in W^{2,\infty}(\T^d)$ and $K$ is symmetric with $\int K=1$.
A second-order Taylor expansion gives the standard bias estimate
\begin{equation}\label{eq:bias_bound}
\|\E[\tilde f_{t,N,\varepsilon}]-\tilde f_t\|_{L^2_v}^2
=
\|K_\varepsilon * \tilde f_t-\tilde f_t\|_{L^2_v}^2
\;\le\;
C_{\mathrm{bias}}\,\varepsilon^{4},
\end{equation}
where $C_{\mathrm{bias}}$ depends only on $\|D^2\tilde f_t\|_{L^\infty_v}$ and the second moment
$\int_{\T^d}|u|^2|K(u)|\,du$.

Write
\[
\tilde f_{t,N,\varepsilon}(v)-\E[\tilde f_{t,N,\varepsilon}(v)]
=
\frac1N\sum_{i=1}^N\Big(K_\varepsilon(v-v_t^{(i)})-\E[K_\varepsilon(v-V_t)]\Big),
\qquad V_t\sim \tilde f_t.
\]
Using independence, centering, and Fubini,
\begin{equation}\label{eq:variance_bound}
\E\|\tilde f_{t,N,\varepsilon}-\E[\tilde f_{t,N,\varepsilon}]\|_{L^2_v}^2
=
\frac{1}{N}\,
\E\Big\|K_\varepsilon(\cdot-V_t)-\E[K_\varepsilon(\cdot-V_t)]\Big\|_{L^2_v}^2
\;\le\;
\frac{1}{N}\,\|K_\varepsilon\|_{L^2_v}^2
=
\frac{1}{N\,\varepsilon^{d}}\|K\|_{L^2_v}^2 .
\end{equation}

Let $e_t^{(i)}:=\hat v_t^{(i)}-v_t^{(i)}$.
Since $K\in C^1$ (hence $K_\varepsilon\in C^1$), by the mean value theorem,
for each $i$ and each $v$,
\[
K_\varepsilon(v-\hat v_t^{(i)})-K_\varepsilon(v-v_t^{(i)})
=
-\int_0^1 \nabla_v K_\varepsilon\!\big(v-(v_t^{(i)}+s e_t^{(i)})\big)\cdot e_t^{(i)}\,ds.
\]
Taking $L^2_v$ norms and using translation invariance of $\|\cdot\|_{L^2_v}$ yields
\[
\|K_\varepsilon(\cdot-\hat v_t^{(i)})-K_\varepsilon(\cdot-v_t^{(i)})\|_{L^2_v}
\le
\|\nabla_v K_\varepsilon\|_{L^2_v}\,\|e_t^{(i)}\|
=
\varepsilon^{-(d/2+1)}\|\nabla_v K\|_{L^2_v}\,\|e_t^{(i)}\|.
\]
Therefore, by Jensen and Cauchy--Schwarz,
\begin{align}
\|f^{\mathrm{approx}}_{t,N,\varepsilon}-\tilde f_{t,N,\varepsilon}\|_{L^2_v}^2
&=
\Big\|\frac1N\sum_{i=1}^N\big(K_\varepsilon(\cdot-\hat v_t^{(i)})-K_\varepsilon(\cdot-v_t^{(i)})\big)\Big\|_{L^2_v}^2 \notag\\
&\le
\frac1N\sum_{i=1}^N
\|K_\varepsilon(\cdot-\hat v_t^{(i)})-K_\varepsilon(\cdot-v_t^{(i)})\|_{L^2_v}^2 \notag\\
&\le
\varepsilon^{-(d+2)}\|\nabla_v K\|_{L^2_v}^2\cdot \frac1N\sum_{i=1}^N\|e_t^{(i)}\|^2
=
\varepsilon^{-(d+2)}\|\nabla_v K\|_{L^2_v}^2\,E(t).
\label{eq:traj_term_bound}
\end{align}
Taking expectation of \eqref{eq:traj_term_bound} gives
\[
\E\|f^{\mathrm{approx}}_{t,N,\varepsilon}-\tilde f_{t,N,\varepsilon}\|_{L^2_v}^2
\le
\varepsilon^{-(d+2)}\|\nabla_v K\|_{L^2_v}^2\,\E E(t).
\]

Taking expectations in \eqref{eq:decomp_three_terms} and using
\eqref{eq:bias_bound}, \eqref{eq:variance_bound}, and \eqref{eq:traj_term_bound},
we obtain
\[
\E\|f^{\mathrm{approx}}_{t,N,\varepsilon}-\tilde f_t\|_{L^2_v}^2
\le
C\Big(
\varepsilon^4+\frac{1}{N\varepsilon^d}+\varepsilon^{-(d+2)}\,\E E(t)
\Big),
\]
which is exactly \eqref{eq:density_bound_general} (absorbing the factor $3$ into $C$).

By the master $W_1$-certificate (Theorem~\ref{thm:master_w1}) and the intermediate estimate
\eqref{eq:E_gronwall_expect} therein, there exists $C(T)>0$ such that
\begin{equation}\label{eq:EE_from_master}
\E E(t)
\le
C(T)\Bigg(
\int_0^t \mathbb{E}\widehat{\mathcal E}_{\mathrm{ISM}}(\tau)\,d\tau
+\int_0^t \E\,\delta_{\mathrm{phys}}(\tau)^2\,d\tau
+N^{-1/2}
\Bigg).
\end{equation}
Substituting \eqref{eq:EE_from_master} into \eqref{eq:density_bound_general} yields
\eqref{eq:density_bound_training} (after renaming constants).
This completes the proof.

\section{Experimental details}\label{app:exp_config}

\subsection{Common Settings}
Unless otherwise stated, all experiments use time step $\Delta t = 0.01$. For SBP and Blob baselines, the number of iterations is fixed to 25, and the learning rate is $10^{-4}$ whenever neural networks are used. For all PINN-based models, Adam optimizer with learning rate $10^{-4}$ is used. The activation function is SiLU.

\begin{table}[h]
\centering
\caption{Training configurations across problems}
\begin{tabular}{lcccc}
\toprule
Problem & $N$~SBP\cite{YL2025} & $N$~Blob\cite{JKF2019} & KDE bw & $c_\gamma$ \\
\midrule
BKW 2D & 22,500 & 22,500 & 0.15 & 1 \\
BKW 3D & 64,000 & 64,000 & 0.15 & 1 \\
Anisotropic 2D & 14,400 & 14,400 & 0.3 & 1 \\
Rosenbluth 3D & 27,000 & 27,000 & 0.3 & 1 \\
Gaussian mixture 3D & 64,000 & 64,000 & 0.15 & 1 \\
Truncated 2D & 14,400 & 14,400 & 0.3 & 1 \\
\bottomrule
\end{tabular}
\end{table}

For all neural networks for SBP, hidden width is 32 and depth is 3,
with learning rate $10^{-4}$.

\subsection{PINN Architectures}
\label{app:pinn_arch}

All PINN experiments use $N=1000$ particles. The architecture consists of two networks:
(i) a {Particle Network} and (ii) a {Score Network}.
Both networks use SiLU activation and the Adam optimizer with learning rate $10^{-4}$.

\paragraph{Group A: 2D Standard Configuration.}
Applied to: \textbf{BKW 2D}, \textbf{Anisotropic 2D}. \\
The Particle Network uses a velocity embedding of width $32$ and depth $2$,
a time embedding of width $16$ and depth $1$,
followed by hidden layers of width $128$ and depth $4$.
The Score Network uses the same architecture.

\paragraph{Group B: 3D Large-Scale Configuration.}
Applied to: \textbf{BKW 3D}, \textbf{Gaussian mixture 3D}. \\
The Particle Network uses a velocity embedding of width $256$ and depth $2$,
a time embedding of width $128$ and depth $1$,
followed by hidden layers of width $256$ and depth $6$.
The Score Network uses the same architecture.

\paragraph{Group C: Moderate 3D Configuration.}
Applied to: \textbf{Rosenbluth 3D}. \\
The Particle Network uses a velocity embedding of width $32$ and depth $2$,
a time embedding of width $64$ and depth $1$,
followed by hidden layers of width $128$ and depth $4$.
The Score Network uses a velocity embedding of width $32$ and depth $2$,
a time embedding of width $16$ and depth $1$,
followed by hidden layers of width $128$ and depth $4$.

\paragraph{Group D: Mixed-Scale Configuration.}
Applied to: \textbf{Truncated 2D}. \\
The Particle Network follows the 2D standard configuration:
velocity embedding of width $32$ and depth $2$,
time embedding of width $16$ and depth $1$,
followed by hidden layers of width $128$ and depth $4$.
The Score Network follows the 3D large-scale configuration:
velocity embedding of width $256$ and depth $2$,
time embedding of width $128$ and depth $1$,
followed by hidden layers of width $256$ and depth $6$.

\subsection{Initial Distributions and Interaction Kernels}
\label{app:init_distributions}

This appendix summarizes the initial distributions and interaction
kernels used in the numerical experiments reported in
Section~\ref{sec:exp}. Throughout the experiments, particles are initialized by i.i.d. sampling from the prescribed initial density $f_0$.
\paragraph{Two-dimensional BKW solution.}
The analytical density is
\[
\tilde f_t(v)
=
\frac{1}{2\pi K}
\exp\!\left(-\frac{|v|^2}{2K}\right)
\left(
\frac{2K-1}{K}
+
\frac{1-K}{2K^2}|v|^2
\right),
\]
where $K(t)=1-\tfrac12 e^{-t/8}$. The associated score is available analytically via $\tilde s_t(v)=\nabla_v\log \tilde f_t(v).$

\paragraph{Three-dimensional BKW solution.}
The density is
\[
\tilde f_t(v)
=
\frac{1}{(2\pi K)^{3/2}}
\exp\!\left(-\frac{|v|^2}{2K}\right)
\left(
\frac{5K-3}{2K}
+
\frac{1-K}{2K^2}|v|^2
\right),
\]
with $K(t)=1-e^{-t/6}$. These analytical solutions provide reference trajectories and scores
for evaluating transport accuracy and score consistency.

\paragraph{Gaussian mixture distribution}
We consider a multimodal Gaussian mixture distribution in $\mathbb{R}^3$. The density is constructed in a separable form
\[
f_0(v)
=
f_{0,1}(v_1)\,f_{0,2}(v_2)\,f_{0,3}(v_3),
\]
where
\[
\begin{aligned}
f_{0,1}(v_1)
&=
0.4\,\mathcal N(v_1;-2,0.3^2)
+
0.6\,\mathcal N(v_1;1,0.8^2), \\
f_{0,2}(v_2)
&=
0.7\,\mathcal N(v_2;-1,0.5^2)
+
0.3\,\mathcal N(v_2;2,0.4^2), \\
f_{0,3}(v_3)
&=
0.5\,\mathcal N(v_3;0,0.2^2)
+
0.5\,\mathcal N(v_3;3,1.2^2).
\end{aligned}
\]
This construction yields a mixture with $K=8$ Gaussian components.

\paragraph{Rosenbluth distribution}
We consider the Rosenbluth-type initial distribution used in~\cite{YL2025}:
\[
f_0(v)
=
\frac{1}{S^2}
\exp\!\left(
-\,S\,\frac{(|v|-\sigma)^2}{\sigma^2}
\right).
\]
In the experiments, we use $\sigma = 2$ and $S = 12$. This distribution generates a ring-shaped density in velocity space, which provides a challenging configuration for Coulomb interactions.

\paragraph{Anisotropic Gaussian distribution}

To test isotropization under Landau dynamics, we consider anisotropic
initial data. Specifically, we use a bimodal Gaussian distribution
\[
f_0(v)
=
\frac{1}{4\pi}
\left(
\exp\!\left(-\frac{|v-u_1|^2}{2}\right)
+
\exp\!\left(-\frac{|v-u_2|^2}{2}\right)
\right),
\]
where $u_1=(-2,1)$ and $u_2=(0,-1)$. This configuration produces two separated clusters in velocity space.

\paragraph{Truncated Gaussian distribution}

We consider a radially truncated Gaussian distribution in
$\mathbb{R}^2$. Let
\[
\phi_2(v)
=
\frac{1}{2\pi}
\exp\!\left(-\frac{|v|^2}{2}\right)
\]
denote the standard Gaussian density. The truncated density is defined by
\[
f_0(v)
=
\frac{\phi_2(v)}{\mathbb P(|Z|>\eta)}
\mathbf 1_{\{|v|>\eta\}},
\]
where
\[
Z\sim \mathcal N(0,I_2).
\]
Since $|Z|$ follows a Rayleigh distribution,
\[
\mathbb P(|Z|>\eta)=\exp(-\eta^2/2),
\]
and therefore
\[
f_0(v)
=
\frac{1}{2\pi}
\exp\!\left(-\frac{|v|^2}{2}\right)
\exp\!\left(\frac{\eta^2}{2}\right)
\mathbf 1_{\{|v|>\eta\}}.
\]
For our experiment, we set $\eta=1$.

\subsection{Discussion of Design Choices}

\paragraph{Dimensional scaling.}
For fully three-dimensional problems with complex geometry (BKW 3D and Gaussian mixture 3D), we increase both width and depth (256, 6 layers) to enhance representation capacity. Two-dimensional problems use a lighter architecture (128, 4 layers).

\paragraph{Bandwidth adjustment in Blob.}
Bandwidth is set to 0.15 for near-Gaussian distributions, and increased to 0.3 for anisotropic or truncated distributions, where sharper local structures require stronger smoothing.

\paragraph{Particle efficiency.}
While Li--Wang and Blob require 14,400--64,000 particles depending on the problem,
the PINN method consistently uses 1,000 particles. This highlights the efficiency gained from learning the score function and transport map via neural representations.

\end{document}